\documentclass[a4paper,reqno,10pt]{amsart}

\usepackage{amsmath}
\usepackage{amsthm}
\usepackage{amssymb}
\usepackage{amscd} 

\usepackage{mathabx}

\usepackage{bbm} 
\usepackage{mathrsfs} 

\usepackage[applemac]{inputenc} 

\usepackage{verbatim} 

\swapnumbers 

\newtheoremstyle{exercise} 
  {3pt} 
  {3pt} 
  {\scriptsize\rmfamily} 
  {
\parindent} 
  {\rmfamily\scshape} 
  {.} 
  {.5em} 
  {} 

\newtheoremstyle{newplain}
  {5pt}
  {5pt}
  {\itshape}
  {}
  {\rmfamily\scshape}
  {. ---}
  {.5em}
  {}

\newtheoremstyle{newremark}
  {5pt}
  {5pt}
  {\rmfamily}
  {}
  {\rmfamily\scshape}
  {. ---}
  {.5em}
  {}

\theoremstyle{newplain}
\newtheorem*{Theorem*}{Theorem} 

\theoremstyle{newplain}
\newtheorem{Theorem}{Theorem}
\newtheorem{Lemma}[Theorem]{Lemma}
\newtheorem{Corollary}[Theorem]{Corollary}
\newtheorem{Proposition}[Theorem]{Proposition}
\newtheorem{Conjecture}[Theorem]{Conjecture}
\newtheorem{Definition}[Theorem]{Definition}

\theoremstyle{newremark}
\newtheorem{Remark}[Theorem]{Remark}
\newtheorem{Question}[Theorem]{Question}
\newtheorem{Example}[Theorem]{Example}
\newtheorem{Claim}[Theorem]{Claim}

\theoremstyle{exercise}

\numberwithin{Theorem}{subsection}
\numberwithin{Exercise}{subsection}

\newcommand{\N}{\mathbb{N}} 
\newcommand{\Q}{\mathbb{Q}} 
\newcommand{\R}{\mathbb{R}} 
\newcommand{\Rm}{\R^m}
\newcommand{\Rn}{\R^n}
\newcommand{\Z}{\mathbb{Z}} 
\newcommand{\K}{\mathbb{K}} 
\newcommand{\C}{\mathbb{C}} 

\newcommand{\ind}{\mathbbm{1}} 

\newcommand{\bbS}{\mathbb{S}}

\newcommand{\calB}{\mathscr{B}}

\newcommand{\calE}{\mathscr{E}}
\newcommand{\calF}{\mathscr{F}}
\newcommand{\calG}{\mathscr{G}}
\newcommand{\calH}{\mathscr{H}}

\newcommand{\calJ}{\mathscr{J}}

\newcommand{\calL}{\mathscr{L}}
\newcommand{\calM}{\mathscr{M}}

\newcommand{\calQ}{\mathscr{Q}}

\newcommand{\calS}{\mathscr{S}}
\newcommand{\calT}{\mathscr{T}}
\newcommand{\calU}{\mathscr{U}}
\newcommand{\calV}{\mathscr{V}}

\newcommand{\caA}{\boldsymbol{\frak A}}

\newcommand{\frA}{\frak A}
\newcommand{\frB}{\frak B}

\newcommand{\frL}{\frak L}

\newcommand{\frT}{\frak T}

\newcommand{\balpha}{\boldsymbol{\alpha}}

\newcommand{\boldeta}{\boldsymbol{\eta}}

\newcommand{\bpi}{\boldsymbol{\pi}}
\newcommand{\brho}{\boldsymbol{\rho}}
\newcommand{\bsigma}{\boldsymbol{\sigma}}

\newcommand{\bxi}{\boldsymbol{\xi}}
\newcommand{\bzeta}{\boldsymbol{\zeta}}
\newcommand{\bUpsilon}{\boldsymbol{\Upsilon}}

\newcommand{\bE}{\mathbf{E}}

\newcommand{\bG}{\mathbf{G}}

\newcommand{\bI}{\mathbf{I}}

\newcommand{\bL}{\mathbf{L}}
\newcommand{\bM}{\mathbf{M}}

\newcommand{\bU}{\mathbf{U}}

\newcommand{\bW}{\mathbf{W}}

\newcommand{\bY}{\mathbf{Y}}

\newcommand{\bc}{\mathbf{c}}

\newcommand{\bbf}{\mathbf{f}}

\DeclareMathOperator{\rmBdry}{\mathrm{Bdry}} 
\DeclareMathOperator{\rmcard}{\mathrm{card}} 
\DeclareMathOperator{\rmClos}{\mathrm{Clos}} 
\DeclareMathOperator{\rmdeg}{\mathrm{deg}}
\DeclareMathOperator{\rmdiam}{\mathrm{diam}} 

\DeclareMathOperator{\rmdim}{\mathrm{dim}} 
\DeclareMathOperator{\rmdist}{\mathrm{dist}} 

\DeclareMathOperator{\rmHom}{\mathrm{Hom}} 
\DeclareMathOperator{\rmid}{\mathrm{id}} 
\DeclareMathOperator{\rmim}{\mathrm{im}} 

\DeclareMathOperator{\rmlip}{\mathrm{lip}}
\DeclareMathOperator{\rmLip}{\mathrm{Lip}} 

\DeclareMathOperator{\rmosc}{\mathrm{osc}} 
\DeclareMathOperator{\rmspan}{\mathrm{span}} 
\DeclareMathOperator{\rmsplit}{\mathrm{split}}
\DeclareMathOperator{\rmsupp}{\mathrm{supp}} 

\newcommand{\hel} {
\hskip2.5pt{\vrule height7pt width.5pt depth0pt}
\hskip-.2pt\vbox{\hrule height.5pt width7pt depth0pt}
\, }

\newcommand{\lseg}{\boldsymbol{[}\!\boldsymbol{[}}
\newcommand{\rseg}{\boldsymbol{]}\!\boldsymbol{]}}

\newcommand{\bin}[2]{
\begin{pmatrix} #1 \\
#2
\end{pmatrix}}

\def\Xint#1{\mathchoice
{\XXint\displaystyle\textstyle{#1}}%
{\XXint\textstyle\scriptstyle{#1}}%
{\XXint\scriptstyle\scriptscriptstyle{#1}}%
{\XXint\scriptscriptstyle
\scriptscriptstyle{#1}}%
\!\int}
\def\XXint#1#2#3{{%
\setbox0=\hbox{$#1{#2#3}{\int}$}
\vcenter{\hbox{$#2#3$}}\kern-.5\wd0}}

\def\dashint{\Xint-}

\newcommand{\lno}{\left\bracevert}
\newcommand{\rno}{\right\bracevert}

\newcommand{\veps}{\varepsilon}

\newcommand{\vphi}{\varphi}

\newcommand{\la}{\langle}
\newcommand{\ra}{\rangle}

\renewcommand{\leq}{\leqslant}
\renewcommand{\geq}{\geqslant}
\renewcommand{\subset}{\subseteq}

\begin{document}


\title[Multiple valued maps]{Existence of $p$ harmonic multiple valued maps into a separable Hilbert space}
\author{Philippe Bouafia}
\author{Thierry De Pauw}
\email{depauw@math.jussieu.fr}
\author{Jordan Goblet}

\maketitle

\tableofcontents
\newpage


\section{Foreword}

\nocite{GOB.09}

Given a set $Y$ and a positive integer $Q$, we let $\calQ_Q(Y)$ denote the set of unordered $Q$-tuples of elements of $Y$, i.e. members of the quotient of $Y^Q$ by the action of the group of permutations $S_Q$. A $Q$-valued map from a set $X$ to $Y$ is a map $f : X \to \calQ_Q(Y)$.
\par
We let $\lseg y_1,\ldots,y_Q \rseg$ denote members of $\calQ_Q(Y)$. If $Y$ is a metric space then $\calQ_Q(Y)$ is given a metric
\begin{equation*}
\calG_\infty( \lseg y_1,\ldots,y_Q \rseg , \lseg y'_1,\ldots,y'_Q \rseg) = \min_{\sigma \in S_Q} \max_{i=1,\ldots,Q} d(y_i,y'_{\sigma(i)}) \,.
\end{equation*}
If $X$ is metric as well, we may thus consider Lipschitz maps $f : X \to \calQ_Q(Y)$. Although these do not necessarily decompose $f = \lseg f_1,\ldots,f_Q \rseg$ into Lipschitz branches $f_i : X \to Y$, $i=1,\ldots,Q$, (see the easy example at the end of Section \ref{section2.2}) we nevertheless establish, in case $X = \Rm$ and $Y$ is a Banach space with the Radon-Nikod\'{y}m property, their differentiability almost everywhere, for an appropriate notion of a derivative $Df$ that controls the variations of $f$ (Theorem \ref{258} and Proposition \ref{259}). In case $Y$ is finite dimensional, this had been obtained by F.J. Almgren \cite{ALMGREN}, the third author \cite{GOB.06}, and C. De Lellis and E. Spadaro \cite{DEL.SPA.11}. Our proof in the infinite dimensional setting follows essentially that given in the last two references.
\par
In case $X = \ell_2^m$ and $Y=\ell_2^n$ are finite dimensional Hilbert spaces, the Lipschitz $Q$-valued $f : \ell_2^m \to \calQ_Q(\ell_2^n)$ were considered by F.J. Almgren in \cite{ALMGREN} in order to approximate the support of a mass minimizing integral current $T \in \bI_m(\ell_2^{m+n})$ near a point $0 \in \rmsupp T$ such that $\bM(T \hel U(0,1)) \sim Q \balpha(m)$ and the ``excess'' of $T$ in $U(0,1)$ with respect to an $m$-plane $W \in \bG(n,m)$ is small. Thus $X=\ell_2^m$ is identified with $W$, $Y=\ell_2^n$ is identified with $W^\perp$ and the graph of $f$ approximates the support of $T$ in $U(0,1)$. The mass minimality of $T$ implies that $f$ is not too far from minimizing its Dirichlet energy $\int_{U(0,1)} \lno Df \rno^2 d\calL^m$, in an appropriate class of Sobolev competitors. F.J. Almgren's analysis (i.e. the definition of Sobolev $Q$-valued maps, their differentiability almost everywhere, the lower semicontinuity of their energy, their trace theory, the Poincar\'e inequality and the relevant compactness result) relied on his biLipschitz embedding
\begin{equation*}
\bxi : \calQ_Q(\ell_2^n) \to \R^N
\end{equation*}
where $N$ and $\rmLip \bxi^{-1}$ depend both upon $n$ and $Q$. Following C. De Lellis and E. Spadaro \cite{DEL.SPA.11}, we present this embedding in Theorem \ref{334}. We also include B. White's ``local isometric'' improvement (unpublished) as conclusion (B) of Theorem \ref{334}. Finally, we compare with an earlier biH\"olderian embedding due to H. Whitney \cite{WHITNEY.72}, Section \ref{31}.
\par
In this paper we concentrate on the Dirichlet problem for the $p$-energy, $1 < p < \infty$, of $Q$-valued maps $f : \ell_2^m \to \calQ_Q(\ell_2)$, i.e. $Y=\ell_2$ is an infinite dimensional separable Hilbert space. We don't know of any useful replacement of Almgren's embedding in that case. Thus we are led to develop further the intrinsic approach pioneered in \cite{GOB.06} and \cite{DEL.SPA.11} (yet we cannot dispense completely with the locally isometric embedding, in particular when proving the lower semicontinuity of the energy).
\par
Letting $U = U(0,1)$ be the unit ball of $\ell_2^m$, we consider the Borel measurable maps $f : U \to \calQ_Q(\ell_2)$ with finite $L_p$ ``norm'', $\int_U \calG(f,Q\lseg 0 \rseg)^p d\calL^m < \infty$. Their $L_p$-semidistance is defined as $d_p(f_1,f_2) = \left( \int_U \calG(f_1,f_2)^p d\calL^m \right)^\frac{1}{p}$; it is complete (Proposition \ref{411}). The Sobolev maps $f \in W^1_p(U;\calQ_Q(\ell_2))$ are defined to be the limits in this $L_p$-semidistance of sequences of Lipschitz maps $f_j : U \to \calQ_Q(\ell_2)$ such that $\sup_j \int_U \lno Df_j \rno^p d\calL^m < \infty$. This sort of ``weak density'' of Lipschitz $Q$-valued maps among Sobolev ones is justified, in case $Y = \ell_2^n$ is finite dimensional, by the fact that $U$ is an extension domain and that $\rmim \bxi$ is a Lipschitz retract of $\R^N$ (Theorem \ref{431}). That Sobolev $Q$-valued maps extend from $U$ to the whole $\ell_2^m$, with the appropriate control, is a matter of routine verification (Theorem \ref{ext}). We define the $p$-energy $\calE^p_p(f;U)$ of a Sobolev $Q$-valued map $f$ by relaxation, making it automatically lower semicontinuous with respect to convergence in the $L_p$-semidistance (Proposition \ref{441}), and we then embark on showing that $f$ is differentiable almost everywhere and that $\calE_p^p(f;U) = \int_U \lno Df \rno^p d\calL^m$. For this purpose we need to know the corresponding statement for finite dimensional approximating Sobolev maps $U \to \calQ_Q(\ell_2^n)$ (Proposition \ref{447}), a convergence result for the finite dimensional approximations (Theorem \ref{448}), a Poincar\'e inequality (Theorem \ref{452}) from which a stronger (Luzin type) approximation by Lipschitz $Q$-valued maps follows (Proposition \ref{453}(1)). The differentiability almost everywhere of a Sobolev $Q$-valued map (Theorem \ref{453}(4)) now becomes a consequence of our aforementioned Rademacher type result (Theorem \ref{258}). At that point we also obtain that $\calE_p^p(f;U) = \int_U \lno Df \rno^p d\calL^m$ (Theorem \ref{454}), thus the lower semicontinuity sought for. We prove the existence of a useful trace ``operator'' $\calT$ in Theorem \ref{472}, verifying the following continuity property: If $\{f_j\}$ is a sequence of Sobolev maps such that $\lim_j d_p(f,f_j)=0$ and $\sup_j \int_U \lno Df_j \rno^p d\calL^m < \infty$ then $\lim_j d_p(\calT(f),\calT(f_j))=0$. Finally, our Rellich compactness Theorem \ref{482} relies on a Fr\'echet-Kolmogorov compactness Theorem \ref{421} and a new embedding Theorem \ref{341}. Given a Lipschitz $g : \rmBdry U \to \calQ_Q(\ell_2)$ and $1 < p < \infty$, our main result states that the minimization problem
\begin{equation*}
\begin{cases}
\text{minimize } \int_U \lno Df \rno^p d\calL^m \\
\text{among } f \in W^1_p(U;\calQ_Q(\ell_2)) \text{ such that } \calT(f) = g
\end{cases}
\end{equation*}
admits a solution.
\par
Our section Preliminaries contains general results and proofs that can be found in \cite{DEL.SPA.11}. We verify that they apply with an infinite dimensional range when appropriate.

\section{Preliminaries}

\subsection{Symmetric powers}

Let $Q \in \N_0$ be a positive integer and let $Y$ be a metric space. Our aim is to consider unordered Q-tuples of elements of $Y$. For instance, letting $Y=\C$ and letting $P$ be a polynomial of degree $Q$ with coefficients in $\C$, the roots of $P$ form such an unordered $Q$-tuple of complex numbers. Thus the elements under consideration need not be distinct; if some agree they should be counted with their multiplicity.
\par
Formally the collection $\calQ_Q(Y)$ of unordered $Q$-tuples in $Y$ may be defined as the quotient of the Cartesian product $Y^Q$ under the action of the symmetric group $S_Q$. An element $\sigma \in S_Q$ is a permutation of $\{1,\ldots,Q\}$. It acts on $Y^Q$ in the obvious way :
\begin{equation*}
Y^Q \to Y^Q : (y_1,\ldots,y_Q) \mapsto (y_{\sigma(1)},\ldots,y_{\sigma(Q)}) \,.
\end{equation*}
We will denote by $\lseg y_1,\ldots,y_Q \rseg$ the equivalence class of $(y_1,\ldots,y_Q)$ in $\calQ_Q(Y)$, so that in particular $\lseg y_1,\ldots,y_Q\rseg = \lseg y_{\sigma(1)},\ldots,y_{\sigma(Q)} \rseg$ for every $\sigma \in S_Q$. On occasions we shall also denote by $v$ a generic element of $\calQ_Q(Y)$. Another way of thinking of a member $v = \lseg y_1,\ldots,y_Q \rseg \in \calQ_Q(Y)$ is to identify it with the finite measure $\mu_v = \sum_{i=1}^Q \delta_{y_i}$ where $\delta_{y_i}$ is the Dirac mass with atom $\{y_i\}$.
 The {\em support} of $v \in \calQ_Q(Y)$ is, by definition, the support of the corresponding measure, $\rmsupp v = \rmsupp \mu_v = \{y_1,\ldots,y_Q\}$ where $y_1,\ldots,y_Q$ is a {\em numbering} of $v$, i.e. a map $y : \{1,\ldots,Q\} \to Y$ such that $v= \lseg y_1,\ldots,y_Q \rseg$. The {\em multiplicity} of $y \in \rmsupp v$ is defined as $\mu_v\{y\}$.
\par
We now define a metric on $\calQ_Q(Y)$ associated with the given metric $d$ of $Y$. Let
\begin{equation*}
\calG( \lseg y_1,\ldots,y_Q \rseg , \lseg y'_1,\ldots,y'_Q \rseg) = \min_{\sigma \in S_Q} \sqrt{ \sum_{i=1}^Q d(y_i,y'_{\sigma(i)})^2} \,.
\end{equation*}
We will sometimes use the notation $\calG_2$ for $\calG$ in order to avoid confusion with two other useful metrics:
\begin{equation*}
\calG_1( \lseg y_1,\ldots,y_Q \rseg , \lseg y'_1,\ldots,y'_Q \rseg) = \min_{\sigma \in S_Q} \sum_{i=1}^Q d(y_i,y'_{\sigma(i)}) \,,
\end{equation*}
and
\begin{equation*}
\calG_{\infty}( \lseg y_1,\ldots,y_Q \rseg , \lseg y'_1,\ldots,y'_Q \rseg) = \min_{\sigma \in S_Q} \max_{i=1,\ldots,Q} d(y_i,y'_{\sigma(i)}) \,.
\end{equation*}
Thus $\calG_1$, $\calG_2$ and $\calG_{\infty}$ are equivalent metrics on $\calQ_Q(Y)$.

We begin with the following easy proposition.

\begin{Proposition}
\label{211}
The metric space $(Y,d)$ is complete (resp. compact, separable) if and only if $(\calQ_Q(Y),\calG)$ is complete (resp. compact, separable) for every $Q \in \N_0$.
\end{Proposition}


A {\em $Q$-valued function} from a set $X$ to $Y$ is a mapping $f : X \to \calQ_Q(Y)$. A {\em multiple-valued function} from $X$ to $Y$ is a $Q$-valued function for some $Q \in \N_0$. In case $X$ is a metric space, the notion of continuity (in particular Lipschitz continuity) of such $f$ now makes sense. If $\caA$ is a $\sigma$-algebra of subsets of $X$ we say that $f$ is $\caA$-measurable (or simply measurable when $\caA$ is clear from the context) whenever $f^{-1}(B) \in \caA$ for every Borel subset $B \subset \calQ_Q(Y)$.

Our coming observation will reveal ubiquitous. We define the {\em splitting distance} of $v = \lseg y_1,\ldots,y_Q \rseg \in \calQ_Q(Y)$ as follows:
\begin{equation*}
\rmsplit v = \begin{cases}
\min \{ d(y_i,y_j) : i,j=1,\ldots,Q \text{ and } y_i \neq y_j \} & \text{ if } \rmcard \rmsupp v > 1  \\
+ \infty & \text{ if } \rmcard \rmsupp v = 1 \,.
\end{cases}
\end{equation*}

\begin{Lemma}[Splitting Lemma]
Let $v = \lseg y_1,\ldots,y_Q \rseg \in \calQ_Q(\Rn)$ and $v' \in \calQ_Q(Y)$ be such that $\calG(v,v') \leq \frac{1}{2} \rmsplit v$. Choose a numbering of $v' = \lseg y'_1,\ldots,y'_Q \rseg \in \calQ_Q(Y)$ so that $d(y_i,y'_i) \leq \frac{1}{2} \rmsplit v$, $i=1,\ldots,Q$. It follows that
\begin{equation*}
\calG(v,v') = \sqrt{ \sum_{i=1}^Q d(y_i,y'_i)^2 }
\end{equation*}
(and the analogous statement for $\calG_1$ and $\calG_{\infty}$).
\end{Lemma}

\begin{proof}
We first observe that in case $\rmsplit v = \infty$ the conclusion indeed holds true. Thus we assume that $\rmsplit v < \infty$. Let $\sigma \in S_Q$ and $i=1,\ldots,Q$. We aim to show that $d(y_i,y'_i) \leq d(y_{\sigma(i)},y'_i)$. In case $y_{\sigma(i)} = y_i$ this is obvious. Otherwise, assuming if possible that $d(y_{\sigma(i)},y'_i) < d(y_i,y'_i)$ we would infer from the triangle inequality
\begin{equation*}
\begin{split}
\rmsplit v & \leq d(y_{\sigma(i)},y_i) \\
& \leq d(y_{\sigma(i)},y'_i) + d(y'_i,y_i) \\
& < 2 d(y_i,y'_i) \\
& \leq \rmsplit v \,,
\end{split}
\end{equation*}
a contradiction. Since $i=1,\ldots,Q$ is arbitrary we obtain
\begin{equation*}
\sum_{i=1}^Q d(y_i,y'_i)^2 \leq \sum_{i=1}^Q d(y_{\sigma(i)},y'_i)^2 \,.
\end{equation*}
Since $\sigma \in S_Q$ is arbitrary, the proof is complete.
\end{proof}

\begin{Proposition}
\label{3.2}
The function $\bsigma : \calQ_Q(Y) \to \N_0 : v \mapsto \rmcard \rmsupp v$ is lower semicontinuous.
\end{Proposition}

\begin{proof}
It follows easily from the definition of $\rmsplit v$ that if $v,v' \in \calQ_Q(Y)$ and if $\calG_{\infty}(v',v) < \frac{1}{2}\rmsplit v$ then $\rmcard \rmsupp v' \geq \rmcard \rmsupp v$.
\end{proof}

\subsection{Concatenation and splitting}
\label{section2.2}

Let $Q_1, Q_2 \in \N_0$. We define the {\em concatenation} operation
\begin{equation*}
\oplus : \calQ_{Q_1}(Y) \times \calQ_{Q_2}(Y) \to \calQ_{Q_1+Q_2}(Y) : (v_1,v_2) \mapsto v_1 \oplus v_2
\end{equation*}
as follows. Write $v_1 = \lseg y_{1,1},\ldots,y_{1,Q_1} \rseg$ and $v_2 = \lseg y_{2,1},\ldots,y_{2,Q_2} \rseg$, and put $v_1 \oplus v_2 = \lseg y_{1,1},\ldots,y_{1,Q_1},y_{2,1},\ldots,y_{2,Q_2} \rseg$. We observe that this operation is commutative, i.e. $v_1 \oplus v_2 = v_2 \oplus v_1$. We notice the following associativity property. If $Q_1,Q_2,Q_3 \in \N_0$ and $v_j \in \calQ_{Q_j}(Y)$, $j=1,2,3$, then $(v_1 \oplus v_2) \oplus v_3 = v_1 \oplus (v_2 \oplus v_3)$ so that $v_1 \oplus v_2 \oplus v_3$ is well defined. It is thus possible to iterate the definition to the concatenation of any finite number of members of some $\calQ_{Q_j}(Y)$. In this new notation we readily have the identity
\begin{equation*}
\lseg y_1,\ldots,y_Q \rseg = \lseg y_1 \rseg \oplus \ldots \oplus \lseg y_Q \rseg = \oplus_{i=1}^Q \lseg y_i \rseg \,.
\end{equation*}

We leave the obvious proof of the next result to the reader.

\begin{Proposition}
Let $Q_1,\ldots,Q_k \in \N_0$.
The concatenation operation
\begin{equation*}
\calQ_{Q_1}(Y) \times \ldots \times \calQ_{Q_k}(Y) \to \calQ_{Q_1+\ldots+Q_k}(Y) : (v_1,\ldots,v_k) \mapsto v_1 \oplus \ldots \oplus v_k
\end{equation*}
 is Lipschitz continuous.
\end{Proposition}

In fact if each $\calQ_Q(Y)$ appearing in the statement is equipped with the metric $\calG_1$, and if the Cartesian product is considered as an $\ell_1$ ``product'', then the Lipschitz constant of the above mapping equals 1.
\par
Given $Q$ maps $f_1,\ldots,f_Q : X \to Y$ we define their {\em concatenation} $f : X \to \calQ_Q(Y)$ by the formula
\begin{equation*}
f(x) = \lseg f_1(x),\ldots,f_Q(x) \rseg = \oplus_{i=1}^Q \lseg f_i(x) \rseg  \,,\,\,x \in X .
\end{equation*}
Abusing notation in the obvious way we shall also write
\begin{equation*}
f = \lseg f_1,\ldots,f_Q \rseg \,.
\end{equation*}
In writing $f$ as above we will call $f_1,\ldots,f_Q$ {\em branches} of $f$. It is most obvious that such {\em splitting} of $f$ into branches is always possible, and equally evident that branches are very much not unique unless $X$ is a singleton. It ensues from the above proposition that if $f_i : X \to Y$, $i=1,\ldots,Q$, are measurable (resp. continuous, Lipschitz continuous) then so is their concatenation $f = \oplus_{i=1}^Q \lseg f_i \rseg$.
Now, if $f$ has some of these properties, can it be split into branches $f_1,\ldots,f_Q$ having the same property? The answer is positive for measurability, as we shall see momentarily, but not for continuity.  Consider $f : \C \to \calQ_2(\C)$ defined by $f(z) = \lseg \sqrt{z} , -\sqrt{z} \rseg$. Thus $f$ is (H\"older) continuous (for a recent account of such continuity, consult e.g. \cite{BRI.10}). We claim however that $f$ does not decompose into two continuous branches. In fact we shall argue that the restriction of $f$ to the unit circle, still denoted $f$,
\begin{equation*}
f : \mathbb{S}^1 \to \calQ_2(\mathbb{S}^1) : z \mapsto \lseg \sqrt{z} , -\sqrt{z} \rseg
\end{equation*}
does not admit a continuous selection. Suppose if possible that there are continuous maps $f_1,f_2 : \mathbb{S}^1 \to \mathbb{S}^1$ such that $f = \lseg f_1,f_2 \rseg$. Let $g : \mathbb{S}^1 \to \mathbb{S}^1 : z \mapsto z^2$. From the identity $\rmid_{\mathbb{S}^1}=g \circ f_1$ we infer that $1 = \rmdeg(g \circ f_1) = \rmdeg(g) \circ \rmdeg(f_1)= 2 \rmdeg(f_1)$, contradicting $\rmdeg(f_1) \in \mathbb{Z}$.

\subsection{Measurability}

This section is also contained in \cite{DEL.SPA.11}.
The process of splitting $v \in \calQ_Q(Y)$ (such that $\rmsplit v < \infty$) into $v_1 \in \calQ_{Q_1}(Y)$ and $v_2 \in \calQ_{Q_2}(Y)$, $Q = Q_1 + Q_2$ and $Q_1 \neq 0 \neq Q_2$, is locally well-defined and continuous.

\begin{Proposition}
\label{3.3}
Let $v \in \calQ_Q(Y)$ be such that $s = \rmsplit v < \infty$. There then exist $Q_1,Q_2 \in \N_0$ with $Q=Q_1+Q_2$ and continuous mappings
\begin{equation*}
\psi_k : \calQ_Q(Y) \cap \{ v' : \calG_{\infty}(v,v') < s/2 \} \to \calQ_{Q_k}(Y) \,,\,\, k=1,2 \,,
\end{equation*}
such that
\begin{equation*}
v' = \psi_1(v') \oplus \psi_2(v') \,.
\end{equation*}
\end{Proposition}

When $Y$ is a metric space we let $\frB_Y$ denote the $\sigma$-algebra of Borel subsets of $Y$.

\begin{Proposition}
\label{232}
Let $(X,\frA)$ be a measurable space and let $Y$ be a separable metric space.
\begin{enumerate}
\item[(A)] If $f_1,\ldots,f_Q : X \to Y$ are $(\frA,\frB_Y)$-measurable then $f = \lseg f_1,\ldots,f_Q \rseg$ is $(\frA,\frB_{\calQ_Q(Y)})$-measurable.
\item[(B)] If $f : X \to \calQ_Q(Y)$ is $(\frA,\frB_{\calQ_Q(Y)})$-measurable then there exist $(\frA,\frB_Y)$-measurable maps $f_1,\ldots,f_Q : X \to Y$ such that $f = \lseg f_1,\ldots,f_Q \rseg$.
\end{enumerate}
\end{Proposition}

\begin{proof}
(A) Since $(\calQ_Q(Y),\calG_{\infty})$ is separable (Proposition \ref{211}), each open subset of $\calQ_Q(Y)$ is a finite or countable union of open balls. Thus it suffices to show that $f^{-1}(B_{\calG_{\infty}}(v,r)) \in \frA$ whenever $v \in \calQ_Q(Y)$ and $r > 0$. Writing $ v = \lseg y_1,\ldots,y_Q \rseg$ we simply notice that
\begin{equation*}
\begin{split}
f^{-1}(B_{\calG_{\infty}}(v,r)) & = X \cap \{ x : \calG_{\infty}(f(x),v) < r \} \\
& = X \cap \left\{ x : \min_{\sigma \in S_Q} \max_{i=1,\ldots,Q} d(f_i(x),y_{\sigma(i)}) < r \right\} \\
& = \bigcup_{\sigma \in S_Q} \bigcap_{i=1}^Q f_i^{-1} (B(y_{\sigma(i)},r)) \in \frA \,.
\end{split}
\end{equation*}
\par
(B) The proof is by induction on $Q$. The case $Q=1$ being trivial, we henceforth assume that $Q \geq 2$. We start by letting $F = \calQ_Q(Y) \cap \{ v : \rmcard \rmsupp v = 1 \}$. Notice $F$ is closed, according to Proposition \ref{3.2}, thus $A_0 = f^{-1}(F) \in \frA$. There readily exist identical $(\frA,\frB_Y)$-measurable maps $f^0_1,\ldots,f^0_Q : A_0 \to Y$ such that $f \restriction_{A_0} = \lseg f^0_1,\ldots,f^0_Q \rseg$. We next infer from Proposition \ref{3.3} that to each $v \in \calQ_Q(Y) \setminus F$ there correspond a neighborhood $\calU_v$ of $v$ in $\calQ_Q(Y) \setminus F$, integers $Q^v_1, Q^v_2 \in \N_0$ such that $Q = Q^v_1+Q^v_2$, and continuous maps $\psi_k^v : \calU_v \to \calQ_{Q_k^v}(Y)$, $k=1,2$, such that $\psi_{1}^v \oplus \psi_{2}^v = \rmid_{\calU_v}$. Since $\calQ_Q(Y) \setminus F$ is separable we find a sequence $\{v_j\}$ such that $\calQ_Q(Y) \setminus F = \cup_{j \in \N_0} \calU_{v_j}$. Thus we find a disjointed sequence $\{\calB_j\}$ of Borel subsets of $\calQ_Q(Y)$ such that $\calQ_Q(Y) \setminus F = \cup_{j \in \N_0} \calB_j$ and $\calB_j \subset \calU_{v_j}$ for every $j$. Define $A_j = f^{-1}(\calB_j) \in \frA$, $j \in \N_0$. For each $j \in \N_0$ the induction hypothesis applies to the two multiple-valued functions $\psi^{v_j}_k \circ (f \restriction_{A_j}) : A_j \to \calQ_{Q^{v_j}_k}(Y)$, $k=1,2$, to yield $(\frA,\frB_Y)$-measurable decompositions $\lseg f^j_1,\ldots,f^j_{Q^{v_j}_1} \rseg$ and $\lseg f^j_{1+Q^{v_j}_1},\ldots,f^j_Q \rseg$ (the numberings are chosen arbitrarily). We define $f_i : X \to Y$, $i=1,\ldots,Q$, by letting $f_i \restriction_{A_j} = f^j_i$, $j \in \N_0$. It is now plain that each $f_i$ is $(\frA,\frB_Y)$-measurable and that $f = \lseg f_1,\ldots,f_Q \rseg$.
\end{proof}

\subsection{Lipschitz extensions}

The Lipschitz extension Theorem \ref{243} is due to F.J. Almgren in case $Y$ is finite dimensional (see \cite[1.5]{ALMGREN}), a former version is found in \cite{ALM.86} for a different notion of {\em multiple-valued function}). Here we merely observe that it extends to the case when $Y$ is an arbitrary Banach space (in case $Q=1$ this observation had already been recorded in \cite{JOH.LIN.SCH.86}, the method being due to H. Whitney \cite{WHI.34}). Our exposition is very much inspired by that of \cite{DEL.SPA.11} (see also \cite{LAN.SCH.05} for a comprehensive study of the extension techniques used here). This extension Theorem in case $Y$ is finite dimensional is equivalent to the fact that $\calQ_Q(\Rn)$ is an absolute Lispschitz retract (see Theorem \ref{336}). The latter is proved ``by hand'' in \cite[1.3]{ALMGREN}.
\par
Given a map $f : X \to Y$ between two metric spaces, and $r > 0$, we recall that the {\em oscillation of $f$ at $r$} is defined as
\begin{equation*}
\rmosc(f;r) = \sup \{ d_Y(f(x_1),f(x_2)) : x_1, x_2 \in X \text{ and } d_X(x_1,x_2) \leq r \} \in [0,+\infty] \,.
\end{equation*}
\textbf{In this section $\calQ_Q(Y)$ will be equipped with its metric $\calG_\infty$}.

\begin{Proposition}
\label{241}
Let $Q \geq 2$. Assume that
\begin{enumerate}
\item[(1)] $X$ and $Y$ are metric spaces, $x_0 \in X$, and $\delta = \rmdiam X < \infty$;
\item[(2)] $f : X \to \calQ_Q(Y)$ and $f(x_0) = \lseg y_1(x_0),\ldots,y_Q(x_0) \rseg$;
\item[(3)] There are $i_1,i_2 \in \{1,\ldots,Q\}$ such that
\begin{equation*}
d_Y(y_{i_1}(x_0),y_{i_2}(x_0)) > 3(Q-1) \rmosc(f;\delta) \,.
\end{equation*}
\end{enumerate}
It follows that there are $Q_1,Q_2 \in \N_0$ such that $Q_1+Q_2=Q$, and $f_1,f_2 : X \to \calQ_Q(Y)$ such that $f=f_1 \oplus f_2$ and $\rmosc(f_j;\cdot) \leq \rmosc(f;\cdot)$, $j=1,2$.
\end{Proposition}

\begin{proof}
We let $\calJ$ denote the family of all those $J \subset \{1,\ldots,Q\}$ such that $i_1 \in J$ and for every $j_1,j_2 \in J$,
\begin{equation}
\label{utileaussi}
d_Y(y_{j_1}(x_0),y_{j_2}(x_0)) \leq 3 (\rmcard J -1) \rmosc(f;\delta) \,.
\end{equation}
Notice that $\calJ \neq \emptyset$ (because $\{i_1\} \in \calJ$), and let $J_1 \in \calJ$ be maximal with respect to inclusion. Also define $J_2 = \{1,\ldots,Q\} \setminus J_1$, 
so that $J_2 \neq \emptyset$: according to hypothesis (3), $J_2$ contains at least $i_2$. We notice that for every $j_1 \in J_1$ and every $j_2 \in J_2$ one has
\begin{equation}
\label{eq.13}
d_Y(y_{j_1}(x_0) ,y_{j_2}(x_0)) > 3 \rmosc(f;\delta) \,.
\end{equation}
\par
For each $x \in X$ we choose a numbering $f(x) = \lseg y_1(x),\ldots,y_Q(x) \rseg$ such that
\begin{equation*}
\calG_\infty(f(x_0),f(x)) = \max_{i=1,\ldots,Q} d_Y(y_i(x_0),y_i(x)) \,.
\end{equation*}
We let $Q_1 = \rmcard J_1$, $Q_2 = \rmcard Q_2$, and we define $f_j : X \to \calQ_Q(Y)$, $j=1,2$, by the formula $f_j(x) = \lseg y_i(x) : i \in J_j \rseg$, so that $f=f_1 \oplus f_2$.
\par
For each pair $x,x' \in X$ we choose $\sigma_{x,x'} \in S_Q$ such that
\begin{equation*}
\calG_\infty(f(x),f(x')) = \max_{i=1,\ldots,Q} d_Y(y_i(x) , y_{\sigma_{x,x'}(i)}(x')) \,.
\end{equation*}
We now claim that $\sigma_{x,x'}(J_1) = J_1$ and $\sigma_{x,x'}(J_2) = J_2$, and this will readily finish the proof. Assume if possible that there exist $j_1 \in J_1$ and $j_2 \in J_2$ such that $\sigma_{x,x'}(j_1)=j_2$. Thus $d_Y(y_{j_1}(x),y_{j_2}(x')) \leq \calG_\infty(f(x),f(x'))$, and it would follow from Equation \eqref{eq.13} that
\begin{equation*}
\begin{split}
3 \rmosc(f;\delta) & < d_Y(y_{j_1}(x_0) , y_{j_2}(x_0)) \\
& \leq d_Y(y_{j_1}(x_0),y_{j_1}(x)) + d_Y(y_{j_1}(x),y_{j_2}(x')) + d_Y(y_{j_2}(x'),y_{j_2}(x_0)) \\
& \leq \calG_\infty(f(x_0),f(x)) + \calG_\infty(f(x),f(x')) + \calG_\infty(f(x'),f(x_0)) \\
& \leq 3 \rmosc(f;\delta) \,,
\end{split}
\end{equation*}
a contradiction.
\end{proof}

\begin{Proposition}
\label{242}
For each $Q \in \N_0$ there is a constant $\bc_{\theTheorem}(Q) \geq 1$ with the following property. Assume that
\begin{enumerate}
\item[(1)] $X$ and $Y$ are Banach spaces;
\item[(2)] $C \subset X$ is a closed ball;
\item[(3)] $f : (\rmBdry C, \|\cdot\|) \to (\calQ_Q(Y), \calG_\infty)$ is Lipschitz. 
\end{enumerate}
It follows that $f$ admits an extension $\hat{f} : (C, \|\cdot\|) \to (\calQ_Q(Y), \calG_\infty)$ 
such that
\begin{equation*}
\rmLip \hat{f} \leq \bc_{\theTheorem}(Q) \rmLip f \,,
\end{equation*}
and
\begin{equation*}
\max \{ \calG_\infty(\hat{f}(x),v) : x \in C \} \leq (6Q+2) \max \{ \calG_\infty(f(x),v) : x \in \rmBdry C \}
\end{equation*}
for every $v \in \calQ_Q(Y)$.
\end{Proposition}

\begin{proof}
There is no restriction to assume that $C = B(0,R)$, $R > 0$, is a ball centered at the origin.  
Note that it is enough to construct a Lipschitz extension $\hat{f}$ of $f$ on a dense subset of $C$, for example on $C \setminus \{0\}$. 
 The proof is by induction on $Q$, and we start with the case $Q=1$. Choose $x_0 \in \rmBdry C$. We define
\begin{equation*}
\hat{f}(x) = \left(1 - \frac{\|x\|}{R} \right)f(x_0) + \frac{\|x\|}{R} f \left( \frac{R x}{\|x\|} \right) , \,\,x \in C \setminus \{0\} \,.
\end{equation*}
This is readily an extension of $f$ to $B(0,R) \setminus \{0\}$. In order to estimate its Lipschitz constant, we let $x,x' \in B(0,R) \setminus \{0\}$, we put $r=\|x\|$, $r' = \|x'\|$, and we assume $r \leq r'$. We define $x'' = \frac{rx'}{r'}$, such as $\|x\| = \|x''\|$ and we observe that
\begin{equation*}
\begin{split}
\| \hat{f}(x) - \hat{f}(x'') \| & = \left\| \frac{\|x\|}{R} f \left( \frac{Rx}{\|x\|} \right) - \frac{\|x''\|}{R} f \left( \frac{Rx''}{\|x''\|} \right) \right\| \\
& = \frac{r}{R} \left\| f \left( \frac{Rx}{\|x\|} \right) - f \left( \frac{Rx'}{\|x'\|} \right) \right\| \\
& \leq r (\rmLip f) \left\| \frac{x}{\|x\|} - \frac{x'}{\|x'\|} \right\| \\
& \leq 2 (\rmLip f) \|x-x'\| \,,
\end{split}
\end{equation*}
and
\begin{equation*}
\begin{split}
\| \hat{f}(x'') - \hat{f}(x') \| & = \left\| \frac{\|x''\| - \|x'\|}{R} \left(f\left(\frac{Rx'}{\|x'\|}\right) - f(x_0) \right)\right\| \\
& \leq 2 (\rmLip f) \|x-x'\| 
\end{split}
\end{equation*}
Therefore,
\begin{equation*}
\rmLip \hat{f} \leq 4 \rmLip f \,.  
\end{equation*}
Moreover, if $v \in Y$ and $x \in C$, we compute
\begin{equation*}
\begin{split}
\|\hat{f}(x) - v\| & \leq \left(1 - \frac{\|x\|}{R} \right) \|f(x_0) - v\| + \frac{\|x\|}{R} \left\|f\left(\frac{Rx}{\|x\|} \right) - v\right\| \\
& \leq \max_{\xi \in \rmBdry C} \|f(\xi) - v\|.
\end{split}
\end{equation*}
We are now ready to treat the case when $Q > 1$.
\par
{\em
First case. Assume there are $i_1,i_2 \in \{1,\ldots,Q\}$ and $x_0 \in \rmBdry C$ such that
\begin{equation*}
\begin{split}
\| y_{i_1}(x_0) - y_{i_2}(x_0) \| & > 3Q \rmosc (f; 2R) \\
& \geq 3(Q-1) \rmosc(f;2R) \, .
\end{split}
\end{equation*}
where $f(x_0) = \lseg y_1(x_0) , \ldots , y_Q(x_0) \rseg$.
}
We infer from Proposition \ref{241} (applied with $X = \rmBdry C$) that $f$ decomposes into $f=f_1 \oplus f_2$ with $f_j : \rmBdry C \to \calQ_{Q_j}(Y)$ and $\rmLip f_j \leq \rmLip f$, $j=1,2$. The induction hypothesis implies the existence of extensions $\hat{f}_j : C \to \calQ_{Q_j}(Y)$ of $f_j$, such that $\rmLip \hat{f}_j \leq \bc_{\ref{242}}(Q_j) \rmLip f_j$, $j=1,2$. We put $\hat{f} = \hat{f}_1 \oplus \hat{f}_2$ and we notice that
\begin{equation*}
\rmLip \hat{f} \leq \max \{ \bc_{\ref{242}}(Q_1) , \bc_{\ref{242}}(Q_2) \} \rmLip f
\end{equation*}
and
\begin{equation*}
\rmosc (\hat{f};2R) \leq \rmosc (f;2R) \, .
\end{equation*}
Let $K > 0$ a constant to be determined later.
\begin{itemize}
\item \emph{First subcase}. Suppose that
\begin{equation*}
K \rmosc(\hat{f};2R) \leq \calG_\infty(v,f(x_0)).
\end{equation*}
Then, for any $x \in C$, one has
\begin{equation*}
\begin{split}
\calG_\infty(v, \hat{f}(x)) & \leq \calG_\infty(v, f(x_0)) + \calG_\infty(f(x_0), \hat{f}(x)) \\
& \leq \calG_\infty(v, f(x_0)) + \rmosc(\hat{f};2R) \\
& \leq (1 + K^{-1}) \calG_\infty(v, f(x_0)) \\
& \leq (1+K^{-1}) \max_{\xi \in \rmBdry C} \calG_\infty (v, f(\xi)).
\end{split}
\end{equation*}
\item \emph{Second subcase}. Suppose that
\begin{equation*}
K \rmosc (\hat{f},2R) > \calG_\infty(v,f(x_0)).
\end{equation*}
We will use the same notations as in the proof of Proposition \ref{241}. Recall that $f_l(x) = \oplus_{i \in J_l} \lseg y_i(x) \rseg$ for $x \in \rmBdry C$ and $l \in \{1, 2\}$.

We choose a numbering $v = \lseg v_1, \dots, v_Q \rseg$ such that $\calG_\infty(v, f(x_0)) = \max_{1 \leq i \leq Q} \|v_i - y_i(x_0) \|$. We set $v_{J_1} = \oplus_{i \in J_1} \lseg v_i \rseg$ and $v_{J_2} = \oplus_{i \in J_2} \lseg v_i \rseg$.
We claim that for any $x \in C$,
\begin{equation}
\label{eq1234}
\calG_\infty(v, \hat{f}(x)) = \max (\calG_\infty(v_{J_1}, \hat{f}_1(x)), \calG_\infty(v_{J_2}, \hat{f}_2(x))).
\end{equation}
This together with the inductive hypothesis will complete the proof. Suppose if possible that (\ref{eq1234}) is not valid. Switching $J_1$ and $J_2$ if necessary, it follows that there are $j_1 \in J_1$ and $j_2\in J_2$ with $\calG_\infty (v, \hat{f}(x)) = \| v_{j_1} - \hat{y}_{j_2}(x)\|$. Therefore, using (\ref{utileaussi}) and (\ref{eq.13}),
\begin{equation*}
\begin{split}
K \rmosc (\hat{f};2R) & >  \calG_\infty(v, f(x_0))  \\
& \geq  \calG_\infty(v, \hat{f}(x)) - \calG_\infty(\hat{f}(x), f(x_0)) \\
& \geq  \|v_{j_1} - \hat{y}_{j_2}(x)\| - \rmosc (\hat{f};2R) \\
& \geq  \| y_{i_1}(x_0) - y_{i_2}(x_0)\| \\
  & \qquad - \| y_{i_1}(x_0) - y_{j_1}(x_0) \| - \| y_{j_1}(x_0) - v_{j_1}\| \\
 & \qquad - \| y_{i_2}(x_0) - y_{j_2}(x_0) \| - \| y_{j_2}(x_0) - \hat{y}_{j_2}(x) \| \\
  & \qquad - \rmosc (\hat{f};2R)  \\
 & \geq  \| y_{i_1}(x_0) - y_{i_2}(x_0)\| \\
  & \qquad - 3( \rmcard J_1 - 1) \rmosc (\hat{f};2R) - \calG_\infty (v, f(x_0)) \\
  & \qquad - 3 (\rmcard J_2 - 1) \rmosc (\hat{f};2R) - \calG_\infty(f(x_0), \hat{f}(x)) \\
  & \qquad - \rmosc (\hat{f};2R)  \\
 & \geq  \left(3Q - 3(Q-2) - K - 2  \right) \rmosc (\hat{f};2R)
\end{split}
\end{equation*}
If $K = 2$, one gets a contradiction.
\end{itemize}

\par
{\em
Second case. Assume that for every $i_1,i_2 \in \{1,\ldots,Q\}$ and for every $x \in \rmBdry C$ one has
\begin{equation*}
\|y_{i_1}(x) - y_{i_2}(x) \| \leq 3Q \rmosc(f;2R) \leq 6QR \rmLip f \, .
\end{equation*}
where $f(x) = \lseg y_1(x) , \ldots, y_Q(x) \rseg$ is an arbitrary numbering.
}
We pick some $x_0 \in \rmBdry C$ and we define $\hat{y}_i : C \setminus \{0\} \to Y$ by\footnote{Note we don't claim any regularity about the $y_i$ nor the $\hat{y}_i$, not even measurability}
\begin{equation*}
\hat{y}_i(x) = \frac{\|x\|}{R} y_i \left( \frac{Rx}{\|x\|} \right) + \frac{R - \|x\|}{R} y_1(x_0) \,,
\end{equation*}
$x \in C \setminus \{0\}$ and $i=1,\ldots,Q$. We define $\hat{f} : C \setminus \{0\} \to \calQ_Q(Y)$ by $\hat{f}(x) = \lseg \hat{y}_1(x),\ldots,\hat{y}_Q(x) \rseg$, $x \in C \setminus \{0\}$. We first show that $\rmLip ( \hat{f} \restriction \rmBdry B(0,r)) \leq \rmLip f$, $0 < r \leq R$. Indeed given $x,x' \in C$ with $\|x\|=\|x'\|=r$, we define $\tilde{x} = \frac{Rx}{r}$ and $\tilde{x}' = \frac{Rx'}{r}$, and we select $\sigma \in S_Q$ such that
\begin{equation*}
\calG_\infty(f(\tilde{x}),f(\tilde{x}')) = \max_{i=1,\ldots,Q} \|y_i(\tilde{x}) - y_{\sigma(i)}(\tilde{x}') \| \,.
\end{equation*}
We notice that
\begin{equation*}
\begin{split}
\|\hat{y}_i(x) - \hat{y}_{\sigma(i)}(x') \| & = \frac{r}{R} \|y_i(\tilde{x}) - y_{\sigma(i)}(\tilde{x}') \| \\
& \leq \frac{r}{R} (\rmLip f) \|\tilde{x} - \tilde{x}'\| \\
& = (\rmLip f) \|x-x'\| \,.
\end{split}
\end{equation*}
Therefore
\begin{equation*}
\calG_\infty(\hat{f}(x),\hat{f}(x')) \leq \max_{i=1,\ldots,Q} \|\hat{y}_i(x)-\hat{y}_{\sigma(i)}(x')\| \leq (\rmLip f) \|x-x'\| \,.
\end{equation*}
Next, given $x \in \rmBdry C$, we choose $j \in \{1,\ldots,Q\}$ such that $\|y_j(x)-y_1(x_0)\| \leq \calG_\infty(f(x),f(x_0)) \leq (\rmLip f) 2R$. For each $0 < t_1 < t_2 \leq 1$ and $i=1,\ldots,Q$ one has
\begin{equation*}
\begin{split}
\|\hat{y}_i(t_2x)-\hat{y}_i(t_1x)\| & = (t_2-t_1) \|y_i(x)-y_1(x_0)\| \\
& \leq (t_2-t_1) \left( \|y_i(x)-y_j(x)\| + \|y_j(x)-y_1(x_0)\| \right) \\
& \leq (t_2-t_1) \left( 3Q+ 1 \right) 2R \rmLip f \\
& = \|t_2x - t_1 x\| \left(6Q+ 2  \right) (\rmLip f) \,,
\end{split}
\end{equation*}
thus
\begin{equation*}
\calG_\infty(\hat{f}(t_2x),\hat{f}(t_1x)) \leq \left(6Q + 2  \right) (\rmLip f) \|t_2x-t_1x\| \,.
\end{equation*}
We conclude from the triangle inequality that 
for any $x, x' \in C \setminus \{0\}$, $r = \|x\|, r' = \|x'\|$
\begin{equation*}
\begin{split}
\calG_\infty(\hat{f}(x), \hat{f}(x')) & \leq
\calG_\infty\left(\hat{f}(x), \hat{f}\left(\frac{rx'}{r'}\right)\right) + \calG\left(\hat{f}\left(\frac{rx'}{r'}\right), \hat{f}(x')\right)\\
& \leq (\rmLip f) \left(\left\|x - \|x\| \frac{x'}{r'}\right\|  +  \left(6Q\max_{1 \leq k < Q} \bc_{\ref{242}}(k) + 2  \right) \left\| \frac{rx'}{r'} - x' \right\|    \right) \\
& \leq \left(6Q\max_{1 \leq k < Q} \bc_{\ref{242}}(k) + 4  \right) (\rmLip f) \| x - x' \|
\end{split}
\end{equation*}
Regarding the second part of the Proposition, we choose some $v = \lseg v_1, \dots, v_Q \rseg$, ordered such that
$\calG_\infty(f(x_0), v) = \max_{1 \leq i \leq Q} \| y_i(x_0) - v_i\|$. Let $x \in C \setminus 0$ and $\sigma$ such that
\begin{equation*}
\calG_\infty \left(f\left(\frac{Rx}{\|x\|}\right), v\right) = \max_{1 \leq i \leq Q} \left\|y_{\sigma(i)}\left(\frac{Rx}{\|x\|}\right) - v_i\right\| \, .
\end{equation*}
One has
\begin{equation*}
\begin{split}
\calG_\infty(\hat{f}(x), v) & \leq \max_{1 \leq i \leq Q} \left\| \frac{\|x\|}{R} y_{\sigma(i)}\left(\frac{Rx}{\|x\|}\right) + \frac{R - \|x\|}{R} y_1(x_0) - v_i \right\| \\
& = \max_{1 \leq i \leq Q} \left\| \frac{\|x\|}{R} \left(y_{\sigma(i)}\left(\frac{Rx}{\|x\|}\right) - v_i\right) + \frac{R - \|x\|}{R} (y_1(x_0) - v_i) \right\|  \\
& \leq \calG_\infty \left(f\left(\frac{Rx}{\|x\|}\right), v\right) + \max_{1 \leq i \leq Q} \|y_1(x_0) - v_i \| \\
& \leq \calG_\infty \left(f\left(\frac{Rx}{\|x\|}\right), v\right) + \max_{1 \leq i \leq Q} \left(\|y_i(x_0) - v_i \| + \|y_i(x_0) - y_1(x_0)\|\right) \\
& \leq 2 \max_{\xi \in \rmBdry C} \calG_\infty(f(\xi), v) + 3Q \rmosc(f;2R) \, .
\end{split}
\end{equation*}
Note that
\begin{equation*}
\begin{split}
\rmosc(f;2R) & = \max_{\xi_1, \xi_2 \in \rmBdry C} \calG_\infty(f(\xi_1), f(\xi_2)) \\
& \leq \max_{\xi_1, \xi_2 \in \rmBdry C} \left(\calG_\infty(f(\xi_1), v) + \calG_\infty(v, f(\xi_2))\right) \\
& \leq 2 \max_{\xi \in \rmBdry C} \calG_\infty(f(\xi),v) \,.
\end{split}
\end{equation*}
Thus the proof is complete.
\end{proof}

\begin{Theorem}
\label{243}
For every $Q \in \N_0$ and every $m \in \N_0$ there exists a constant $\bc_{\theTheorem}(m,Q) \geq 1$ with the following property. Assume that
\begin{enumerate}
\item[(1)] $X$ is a finite dimensional Banach space with $m = \dim X$, and $A \subset X$ is closed;
\item[(2)] $Y$ is a Banach space;
\item[(3)] $f : A \to \calQ_Q(Y)$ is Lipschitz.
\end{enumerate}
It follows that $f$ admits an extension $\hat{f} : X \to \calQ_Q(Y)$ with
\begin{equation*}
\rmLip \hat{f} \leq \bc_{\theTheorem}(m,Q) \rmLip f \,,
\end{equation*}
and
\begin{equation*}
\sup \{ \calG_\infty(\hat{f}(x),v) : x \in X \} \leq \bc_{\theTheorem}(m,Q) \sup \{ \calG_\infty(f(x),v) : x \in A \}
\end{equation*}
for every $v \in \calQ_Q(Y)$.
\end{Theorem}

\begin{proof}
Because they are lipeomorphic, there is no restriction to assume that $X$ and $\ell_\infty^m$ coincide: an isomorphism $T : X \to \ell_\infty^m$ will multiply the constant $\bc_{\theTheorem}(m,Q)$ by a factor $\| T \| \cdot \|T^{-1}\|$ (where $\| \cdot \|$ denotes the operator norm), yet one can always find a $T$ such that $\| T \| \cdot \|T^{-1}\|$ is smaller than a constant depending only on $m$, since the Banach-Mazur compactum of $m$ dimensional spaces is bounded. We consider a partition of $X \setminus A$ into {\em dyadic} semicubes $\{C_j\}_j$ with the following property
\begin{equation*}
\frac{\rmdist(C_j,A)}{2} \leq \rmdiam C_j < \rmdist(C_j,A)
\end{equation*}
for every $j \in \N$. With each $C_j$ and $k=0,\ldots,m$ we associate it $k$-skeleton $\calS_k(C_j)$, i.e. $\calS_m(C_j) = \{\rmClos C_j\}$ and $\calS_k(C_j)$ is the collection of those maximal $k$ dimensional convex subsets of the (relative) boundary of each $F \in \calS_{k+1}(C_j)$. We also set $\calS_k = \cup_{j \in \N} \calS_k(C_j)$. We now define, by upwards induction on $k$, mappings
\begin{equation*}
\hat{f}_k : A \cup \left( \bigcup\calS_k \right) \to \calQ_Q(Y)
\end{equation*}
which coincide with $f$ on $A$ and such that
\begin{equation}
\label{eq.14}
\rmLip \hat{f}_k \restriction \left( F \cap \left( \bigcup \calS_k(C_j)\right) \right) \leq C(k,Q) \rmLip f
\end{equation}
for each $F \in \calS_{k+1}(C_j)$, $j \in \N$, (where $C(k,Q)$ is a constant depending only on $k$ and $Q$). Furthermore, if $k \geq 1$ then $\hat{f}_k$ is an extension of $\hat{f}_{k-1}$.
\par
{\em Definition of $\hat{f}_0$.} With each $x \in A \cup \calS_0$ we associate $\xi_x \in A$ such that $\|x-\xi_x\| = \rmdist(x,A)$, and we put $\hat{f}_0(x) = f(\xi_x)$. For $x \in A$ we obviously have $\hat{f}_0(x)=f(x)$. If $x \in C_j$ then
\begin{equation*}
\|x-\xi_x\| = \rmdist(x,A) \leq \rmdiam C_j + \rmdist(C_j,A) \leq 3 \rmdiam C_j \,.
\end{equation*}
Consequently, if $x,x' \in C_j$ then
\begin{equation*}
\|\xi_x-\xi_{x'}\| \leq \|\xi_x-x\| + \|x-x'\| + \|x'-\xi_{x'}\| \leq 7 \rmdiam C_j = 7 \|x - x'\| \,.
\end{equation*}
Thus
\begin{equation*}
\calG_\infty(\hat{f}_0(x),\hat{f}_0(x')) = \calG_\infty(f(\xi_x),f(\xi_{x'})) \leq 7 (\rmLip f) \|x-x'\| \,.
\end{equation*}
This indeed proves \eqref{eq.14} in case $k=0$.
\par
{\em Definition of $\hat{f}_k$ by induction on $k \geq 1$.} We say a $k$-face $F \in \calS_k$ is {\em minimal} if there is no $k$-face $F' \in \calS_k$ such that $F' \subset F$ and $F' \neq F$. We observe that each $k$-face contains a minimal one, and that two distinct minimal $k$-faces have disjoint (relative) interiors. If $F \in \calS_k$ is a minimal $k$-face then its `` boundary'' $\partial F$ (relative to the $k$ dimensional affine subspace containing it) equals $F \cap \cup \calS_{k-1}(C_j)$ where $C_j$ is so that $F \in \calS_k(C_j)$, hence $\rmLip \hat{f}_{k-1} \restriction \partial F \leq C(k-1,Q) \rmLip f$ according to the induction hypothesis \eqref{eq.14}. Thus Proposition \ref{242} guarantees the existence of an extension $\hat{f}_k$ of $\hat{f}_{k-1}$ from $\partial F$ to $F$ so that $\rmLip (\hat{f}_k \restriction F) \leq \bc_{\ref{242}}(Q) C(k-1,Q) \rmLip f$. This completes the definition of $\hat{f}_k$ to $\cup \calS_k$. By construction $\hat{f}_k$ verifies \eqref{eq.14} for every minimal $k$-face $F \in \calS_k$. Since each $k$-face is the union of (finitely many) minimal $k$-faces all contained in the same $k$ dimensional affine subspace of $X$, it is an easy matter to check that \eqref{eq.14} is also verifies for arbitrary $F \in \calS_k$.

According to Proposition \ref{242}, for $k \geq 1$ and $v \in \calQ_Q(Y)$, one has
\begin{equation*}
\sup_{x \in A \cup (\cup \calS_k)} \calG_\infty(v, \hat{f}_k(x)) \leq 2Q \rmLip \phi_k \rmLip \phi_k^{-1} \sup_{x \in A \cup (\cup \calS_{k-1})} \calG_\infty(v, \hat{f}_{k-1}(x))
\end{equation*}
where $\phi_k$ denote a lipeomorphism from a $k$ ball to a $k$ cube. Moreover, one has easily
\begin{equation*}
\sup_{x \in A \cup (\cup \calS_0)} \calG_\infty(v, \hat{f}_0(x)) = \sup_{x \in A} \calG_\infty(v, f(x)) \, .
\end{equation*}
Those two facts implies that
\begin{equation*}
\sup \{ \calG_\infty(\hat{f}_m(x),v) : x \in X \} \leq \bc_{\theTheorem}(m,Q) \sup \{ \calG_\infty(f(x),v) : x \in A \}
\end{equation*}
if $\bc_{\theTheorem}(m,Q) \geq (2Q)^m \prod_{k=1}^m (\rmLip \phi_k \rmLip \phi_k^{-1})$.
\par
{\em We now check that $\hat{f}_m$ is Lipschitz}. Let $x,x' \in X$ and we define the line segment $[x,x'] = X \cap \{ x+ t(x'-x) : 0 \leq t \leq 1\}$. We distinguish between several cases according to the positions of these points. {\em First case : if $x,x' \in A$,} the clearly $\calG_\infty(\hat{f}_m(x),\hat{f}_m(x')) = \calG_\infty(f(x),f(x')) \leq (\rmLip f)\|x-x'\|$. {\em Second case : $x,x' \in \rmClos C_j$ for some $j \in \N$}. It then follows that $\calG_\infty(\hat{f}_m(x),\hat{f}_m(x')) \leq C(m,Q) (\rmLip f) \|x-x'\|$ according to \eqref{eq.14}. {\em Third case : $[x,x'] \cap A = \emptyset$}. One then checks that $J = \N \cap \{ j : [x,x'] \cap \rmClos C_j \neq \emptyset \}$ is finite and we apply the previous case to conclude that also $\calG_\infty(\hat{f}_m(x),\hat{f}_m(x')) \leq C(m,Q) (\rmLip f) \|x-x'\|$. {\em Fourth case : $x \not \in A$ and $x' \in A$.} We choose $j \in \N$ such that $x \in C_j$ and we choose arbitrarily $x'' \in \calS_0(C_j)$. It follows that
\begin{equation*}
\|x-x''\| \leq \rmdiam C_j \leq \rmdist(C_j,A) \leq \|x-x'\|
\end{equation*}
and
\begin{equation*}
\|x-\xi_{x''}\| \leq \|x-x''\| + \|x''-\xi_{x''}\| \leq \rmdiam C_j + 3 \rmdiam C_j \leq 4 \|x-x'\| \,.
\end{equation*}
Thus
\begin{multline*}
\calG_\infty(\hat{f}_m(x),\hat{f}_m(x'))  \leq \calG_\infty(\hat{f}_m(x),\hat{f}_m(x'')) + \calG_\infty(\hat{f}_0(x''),\hat{f}_0(x')) \\
\leq (\rmLip \hat{f}_m \restriction \rmClos C_j) \|x-x''\| + (\rmLip f) \|\xi_{x''}-x'\| \\ \leq 6 ( C(m,Q) +1) (\rmLip f) \|x-x'\| \,.
\end{multline*}
{\em Fifth case : $[x,x'] \cap A \neq \emptyset$ and either $x$ or $x'$ does not belong to $A$.} We let $a$ (resp. $a'$) denote the point $[x,x'] \cap A$ closest to $x$ (resp. $x'$) and we observe that
\begin{multline*}
\calG_\infty(\hat{f}_m(x),\hat{f}_m(x')) \\ \leq \calG_\infty(\hat{f}_m(x),\hat{f}_m(a)) + \calG_\infty(f(a),f(a')) + \calG_\infty(\hat{f}_m(a'),\hat{f}_m(x'))
\\ \leq 6 ( C(m,Q) +1) (\rmLip f) \|x-a\| \\ + (\rmLip f) \|a-a'\| + 6 ( C(m,Q) +1) (\rmLip f) \|a'-x'\| \\
\leq 6 ( C(m,Q) +1) (\rmLip f) \|x-x'\| \,.
\end{multline*}
\end{proof}

\begin{Question}
Given a pair of Banach spaces $X$ and $Y$ we here denote by $\bc(X,Y,Q)$ the {\em best} constant occurring in Theorem \ref{243} corresponding to these Banach spaces. Thus $\bc(X,Y,Q) \leq \bc_{\ref{243}}(\dim X,Q) < \infty$ in case $X$ is finite dimensional. Kirszbraun's Theorem says that $\bc(\ell_2^m,\ell_2^N,1)=1$ for every $n,N \in \N_0$, thus it follows from Theorem \ref{336} that $\bc(\ell^m_2,\ell_2^n,Q) \leq \rmLip \brho_{n,Q}$ is bounded independently of $m$. Is it true that $\bc(\ell_2^m,\ell_2,Q) < \infty$ for every $Q > 1$? That would be an analog of Kirszbraun's Theorem for multiple-valued functions. On the other hand, it is well-known that $\bc(X,\ell_\infty,1) = 1$ for every $X$. Is it true that $\bc(X,\ell_\infty,Q) < \infty$ for every $Q > 1$ and every finite dimensional Banach space $X$? See also Question \ref{339}.
\end{Question}

\subsection{Differentiability}

The results contained in this section are standard in case $Y = \ell_2^n$ is Euclidean. The notion of {\em (approximate) differentiability} was introduced (under the name {\em (approximate) affine approximatability}) by F.J. Almgren in \cite{ALMGREN}. We call {\em unambiguously differentiable} what Almgren calls {\em strongly affinely approximatable}. ``Intrinsic'' proofs (i.e. avoiding the embedding defined in section \ref{aw-embedding}) of the analog of Rademacher's Theorem have been given in \cite{GOB.06} and \cite{DEL.SPA.11}.
\par
\textbf{In this section $X$ is a finite dimensional Banach space, $m=\rmdim X$, $\lambda$ is a Haar measure on $X$, and $Y$ is a separable Banach space.}
\par
We say that $g : X \to \calQ_Q(Y)$ is {\em affine} (resp. {\em linear}) if there are affine maps $A_1,\ldots,A_Q$ from $X$ to $Y$ (resp. linear maps $L_1,\ldots,L_Q$ from $X$ to $Y$) such that $g = \oplus_{i=1}^Q \lseg A_i \rseg$ (resp. $g = \oplus_{i=1}^Q \lseg L_i \rseg$). Our first task is to observe that the $A_i$'s are uniquely determined by $g$.

\begin{Lemma}
\label{251}
Let $A_1,\ldots,A_Q,A'_1,\ldots,A'_Q$ be affine maps from $X$ to $Y$, $g = \oplus_{i=1}^Q \lseg A_i \rseg$, $g' = \oplus_{i=1}^Q \lseg A'_i \rseg$, and $S \subset X$. If $g(x)=g'(x)$ for every $x \in S$ and $\lambda(S) > 0$, then there exists $\sigma \in S_Q$ such that $A_i = A'_{\sigma(i)}$, $i=1,\ldots,Q$.
\end{Lemma}

\begin{proof}
For each $\sigma \in S_Q$ we define
\begin{equation*}
W_{\sigma} = X \cap \{ x : A_i(x) = A'_{\sigma(i)}(x) , \, i=1,\ldots,Q \} \,,
\end{equation*}
and we notice that $W_{\sigma}$ is an affine subspace of $X$. If $x \in S$ then $x \in W_{\sigma}$ for some $\sigma \in S_Q$. Thus $S \subset \cup_{\sigma \in S_Q} W_{\sigma}$. Therefore there exists $\sigma$ such that $\lambda(W_{\sigma}) > 0$, hence $W_{\sigma} = X$.
\end{proof}

Let $f,g : X \to \calQ_Q(Y)$ be Borel measurable, and $a \in X$. We say that {\em $f$ and $g$ are approximately tangent at $a$} if for every $\veps > 0$,
\begin{equation*}
\Theta^m(\lambda \hel  \{ x : \calG(f(x),g(x)) > \veps \|x-a\| \} , a) = 0 \,.
\end{equation*}
It is plain that the distance $\calG$ can be replaced by $\calG_1$ or $\calG_{\infty}$ without changing the scope of the definition.

\begin{Proposition}
\label{252}
Let $g,g' : X \to \calQ_Q(Y)$ be affine and approximately tangent at some $a \in X$. It follows that $g=g'$.
\end{Proposition}

\begin{proof}
Write $g = \oplus_{i=1}^Q \lseg A_i \rseg$, $g' = \oplus_{i=1}^Q \lseg A'_i \rseg$, where $A_1,\ldots,A_Q,A'_1,\ldots,A'_Q$ are affine from $X$ to $Y$, and $A_i = L_i + b_i$, $A'_i = L'_i +b'_i$, $b_i,b'_i \in Y$ and the $L_i$'s and $L'_i$'s are linear. With $0 < \veps \leq 1$ we associate
\begin{equation*}
G_{\veps}  = X \cap \{ x : \calG_1(g(x),g'(x)) \leq \veps \|x-a\| \}
\end{equation*}
so that $\Theta^m(\lambda \hel G_{\veps} , a ) = 1$ by assumption, because $G_\veps$ is Borel measurable. Define
\begin{equation*}
\eta = \inf \{ \|A_i(a) - A'_j(a) \| : i,j=1,\ldots,Q \text{ and } A_i(a) \neq A'_j(a) \} \in (0, \infty].
\end{equation*}
Suppose $\eta < \infty$, the case $\eta = \infty$ being easier to prove. 
Choose $\delta > 0$ small enough for
\begin{equation*}
\delta ( 1 + 2Q \max \{ \|L_1\|_{\infty} , \ldots, \|L_Q\|_{\infty} , \|L'_1\|_{\infty},\ldots,\|L'_Q\|_{\infty} \} ) < \eta \,.
\end{equation*}
Let $x \in G_{\veps} \cap B(a,\delta)$ and write $x = a+h$. There exists $\sigma \in S_Q$ such that
\begin{equation}
\label{eq.1}
\begin{split}
\veps \|h\| & \geq \calG_1(g(a+h),g'(a+h)) \\
& = \sum_{i=1}^Q \|A_i(a+h) - A'_{\sigma(i)}(a+h) \| \\
& = \sum_{i=1}^Q \| L_i(a) + b_i -L'_{\sigma(i)}(a) - b'_{\sigma(i)} + L_i(h) - L'_{\sigma(i)}(h) \| \\
& = \sum_{i=1}^Q \| A_i(a) - A'_{\sigma(i)}(a) + L_i(h) - L'_{\sigma(i)}(h) \| \\
& \geq \sum_{i=1}^Q \|A_i(a) - A'_{\sigma(i)}(a) \| - \sum_{i=1}^Q \|L_i(h) - L'_{\sigma(i)}(h) \| \,,
\end{split}
\end{equation}
whence
\begin{multline*}
\sum_{i=1}^Q \|A_i(a) - A'_{\sigma(i)}(a) \| \\ \leq \veps \|h\| + 2Q \|h\| \max \{ \|L_1\|_{\infty} , \ldots, \|L_Q\|_{\infty} , \|L'_1\|_{\infty},\ldots,\|L'_Q\|_{\infty} \} < \eta
\end{multline*}
since $\|h\| \leq \delta$. The definition of $\eta$ then implies that $A_i(a) = A'_{\sigma(i)}(a)$ for each $i=1,\ldots,Q$. Multiplying \eqref{eq.1} by $t > 0$ we obtain
\begin{equation*}
\begin{split}
\veps \|th\| & \geq \sum_{i=1}^Q \| L_i(th) - L'_{\sigma(i)}(th) \| \\
& = \sum_{i=1}^Q \|A_i(a+th) - A'_{\sigma(i)}(a+th) \| \\
& \geq \calG_1(g(a+th),g'(a+th)) \,.
\end{split}
\end{equation*}
In other words, we have established that for every $0 < \veps \leq 1$ and every $t > 0$, if $a+h \in G_{\veps} \cap B(a,\delta)$ then $a+th \in G_{\veps}$.
\par
Letting $\veps_k = k^{-1}$, $k \in \N_0$, we choose $0 < r_k < \delta$ such that
\begin{equation*}
\lambda(G_{\veps_k} \cap B(a,r_k)) \geq \left( 1 - \frac{1}{4^k} \right) \lambda(B(a,r_k)) \,.
\end{equation*}
If $h_k : B(a,r_k) \to B(a,1)$ maps $a+h$ to $a+r_k^{-1}h$ then the above paragraph says that $h_k(G_{\veps_k} \cap B(a,r_k)) \subset G_{\veps_k} \cap B(a,1)$. Consequently,
\begin{equation*}
\begin{split}
\lambda(G_{\veps_k} \cap B(a,1)) & \geq \lambda (h_k(G_{\veps_k} \cap B(a,r_k))) \\
& = r_k^{-m} \lambda(G_{\veps_k} \cap B(a,r_k)) \\
& \geq \left( 1 - \frac{1}{4^k} \right) \lambda(B(a,1)) \,. 
\end{split}
\end{equation*}
Summing over $k \in \N_0$ we obtain
\begin{equation*}
\lambda \left( \bigcap_{k \in \N_0} G_{\veps_k} \right) > 0
\end{equation*}
and the conclusion now follows from Lemma \ref{251}.
\end{proof}

\begin{Corollary}
If $f : X \to \calQ_Q(Y)$ is Borel measurable, $g,g' : X \to \calQ_Q(Y)$ are both affine and both approximately tangent to $f$ at $a$, then $g=g'$.
\end{Corollary}

\begin{proof}
Observe that $g$ and $g'$ are approximately tangent (to each other) at $a$ and apply Proposition \ref{252}.
\end{proof}

Let $f : X \to \calQ_Q(Y)$ and $a \in X$. We say that $f$ is {\em approximately differentiable} at $a$ if there exists an affine $Q$-valued $g : X \to \calQ_Q(Y)$ which is approximately tangent to $f$ at $a$. According to the above corollary, the existence of such $g$ implies its uniqueness. It will be subsequently denoted as $Af(a)$. Writing $Af(a) = \oplus_{i=1}^Q \lseg A_i \rseg$ we shall see later that $Af(a)(a) = \oplus_{i=1}^Q \lseg A_i(a) \rseg$ equals $f(a)$ in case $f$ is approximately continuous at $a$. Concatenation of the linear parts $L_i = A_i - A_i(0)$ yields $Df(a) = \oplus_{i=1}^Q \lseg L_i \rseg$ which is uniquely determined by $Af(a)$. It may occur (but not too often, as we shall later see) that for some pair of distinct indexes $i$ and $j$ one has $A_i(a)=A_j(a)$, yet $L_i \neq L_j$. We now state a definition to rule this out. We say that $f$ is {\em unambiguously approximately differentiable} at $a$ if $Af(a)$ fulfils the following additional requirement. For every $i,j=1,\ldots,Q$, {\em if} $A_i(a) = A_j(a)$ {\em then} $L_i = L_j$.

\begin{Example}
The affine 2-valued map
\begin{equation*}
g = \R \to \calQ_2(\R) : x \mapsto \lseg x \rseg \oplus \lseg -x \rseg
\end{equation*}
is everywhere (approximately) differentiable, but not unambiguously so at 0.
\end{Example}

The need for unambiguous differentiation appears when stating the Euler-Lagran\-ge equation for minimizing multiple-valued maps with respect to range deformation (so-called ``squash deformation'' by F.J. Almgren), see e.g. \cite[Theorem 2.6(4)]{ALMGREN}.
\par
Recall the function $\bsigma$ defined in Proposition \ref{3.2}.

\begin{Lemma} 
\label{255}
Assume $v \in \calQ_Q(Y)$, put $k = \bsigma(v)$, and let $Q_1,\ldots,Q_k \in \N_0$ and $y_1,\ldots,y_k \in Y$ be such that $v = \oplus_{j=1}^k Q_j \lseg y_j \rseg$ and $Q = \sum_{j=1}^k Q_j$. For every $0 < r < \frac{1}{2} \rmsplit v$ the following holds. Whenever $v' \in \calQ_Q(Y)$ is such that $\calG_\infty(v,v') < r$ and $\bsigma(v')=k$, there are $y'_1,\ldots,y'_k \in Y$ such that $v' = \oplus_{j=1}^k Q_j \lseg y'_j \rseg$, $\|y_j - y'_j \| < r$ for every $j=1,\ldots,k$, and $\calG_(v,v') = \sum_{j=1}^Q Q_j \|y_j - y'_j \|$.
\end{Lemma}

\begin{proof}
Let $v'$ be as in the statement and choose a numbering $v' = \lseg z_1,\ldots,z_Q \rseg$. Since $\calG_1(v,v') < r$, it follows that each $z_i$ is $r$ close to some $y_j$. In other words there exists $\tau : \{1,\ldots,Q \} \to \{1,\ldots,k \}$ such that $\|z_i - y_{\tau(i)} \| < r$, $i=1,\ldots,Q$. Thus
\begin{equation*}
\calG_1(v,v') = \sum_{i=1}^Q \|z_i - y_{\tau(i)}\|
\end{equation*}
according to the Splitting Lemma. We now observe that if $i,i' \in \{1,\ldots,Q \}$ are so that $\tau(i) \neq \tau(i')$ then $z_i \neq z_{i'}$. Indeed the converse would yield
\begin{equation*}
\|y_{\tau(i)} - y_{\tau(i')} \| \leq \|y_{\tau(i)} - z_i \| + \| z_{i'} - y_{\tau(i')} \| \leq 2r < \rmsplit v\,,
\end{equation*}
a contradiction. Since also $\bsigma(v')=k$ we infer that $\tau(i)=\tau(i')$ implies $z_i=z_{i'}$. The proof is complete.
\end{proof}

We now give a criterion implying unambiguous approximate differentiability. Since $\bsigma \circ f$ takes values in $\Z^+$, it is approximately continuous at a point if and only if it is approximately constant at that point.

\begin{Proposition}
\label{256}
Let $f : X \to \calQ_Q(Y)$ be Borel measurable, and $a \in X$. Assume that
\begin{enumerate}
\item[(A)] $f$ is approximately continuous at $a$;
\item[(B)] $\bsigma \circ f$ is approximately constant at $a$;
\item[(C)] $f$ is approximately differentiable at $a$.
\end{enumerate}
It follows that $f$ is unambiguously approximately differentiable at $a$, and that $Af(a)=f(a)$.
\end{Proposition}

\begin{proof}
Write $Af(a) = \oplus_{i=1}^Q \lseg A_i \rseg$, and define $\alpha = \max \{ \|A_1\|_{\infty},\ldots,\|A_Q\|_{\infty} \}$. Put $k = \bsigma(f(a))$.
 There exist $f_1(a), \ldots, f_k(a) \in Y$ and positive integers $Q_1, \ldots, Q_k$ with $\sum_{j=1}^k Q_j = Q$ such that
 $$ f(a) = \oplus_{j=1}^k Q_j \lseg f_j(a)\rseg.$$
 Let $0 < r_0 < \frac{1}{2} \rmsplit f(a)$ so that Lemma \ref{255} applies with any $0 < r \leq r_0$. For each $\veps > 0$ define
\begin{multline*}
G_{\veps} = X \cap \{ x : \calG_1(f(x),f(a)) \leq \veps \text{ and } \bsigma(f(x))=k \\ \text{ and } \calG_1(f(x),Af(a)(x)) \leq \veps \|x-a\| \} \,.
\end{multline*}
Given $\eta > 0$ there exists $r_1(\veps,\eta) > 0$ such that
\begin{equation*}
\lambda (G_{\veps} \cap B(a,r)) \geq (1-\eta) \lambda(B(a,r))
\end{equation*}
whenever $0 < r \leq r_1(\veps,\eta)$. From now on we shall further assume that
\begin{equation*}
\veps < \min \left\{ \frac{1}{8} \rmsplit f(a) , 1 \right\}
\end{equation*}
and
\begin{equation*}
r \leq \min \left\{ 1 , r_0 , r_1(\veps,\eta) , \frac{1}{4 \alpha} \rmsplit f(a) \right\} \,.
\end{equation*}
\par
For each $x \in G_{\veps} \cap B(a,r)$ there are $f_1(x),\ldots,f_k(x) \in Y$ such that
\begin{equation*}
f(x) = \oplus_{j=1}^k Q_j \lseg f_j(x) \rseg
\end{equation*}
and
\begin{equation}
\label{eq.2}
\calG_1(f(x),f(a)) = \sum_{j=1}^k Q_j \|f_j(x)-f_j(a)\| \leq \veps
\end{equation}
according to Lemma \ref{255}. Associated with such $x$, there are also partitions $I_{x,1},\ldots,I_{x,k}$ of $\{1,\ldots,Q\}$ such that
\begin{equation}
\label{eq.3}
\calG_1(f(x),Af(a)(x)) = \sum_{j=1}^k \sum_{i \in I_{x,j}} \|f_j(x)-A_i(x)\| \leq \veps \|x-a\| \,.
\end{equation}
In view of \eqref{eq.2} there also holds
\begin{equation*}
\sum_{j=1}^k \sum_{i \in I_{x,j}} \|f_j(a) - A_i(x)\| \leq \veps (1+\|x-a\|) \leq 2 \veps \,.
\end{equation*}
This already implies that $f(a)=Af(a)(a)$. Now let $x,x' \in G_{\veps} \cap B(a,r)$ and $j,j' \in \{1,\ldots,k\}$. If $i \in I_{x,j} \cap I_{x',j'}$ then
\begin{equation*}
\begin{split}
\|f_j(a)-f_{j'}(a) \| & \leq \|f_j(a) - A_i(x)\| + \|A_i(x)-A_i(x')\| + \|A_i(x')-f_{j'}(a)\|\\
& \leq 4 \veps +  \alpha \|x-x'\| \\
& < \rmsplit f(a)
\end{split}
\end{equation*}
according to our choice of $\veps$ and $r$, thus $j=j'$. This in turn readily implies that $I_{x,j} = I_{x',j} =: I_j$, $j=1,\ldots,k$. It follows from \eqref{eq.3} above that if $x \in G_{\veps} \cap B(a,r)$ and $i,i' \in I_j$, $j=1,\ldots,k$, then
\begin{equation*}
\begin{split}
\|A_i(x) - A_{i'}(x)\| & \leq \|A_i(x) - f_j(x)\| + \|f_j(x)-A_{i'}(x)\| \\
& \leq 2\veps \|x-a\| \,.
\end{split}
\end{equation*}
Since $\eta > 0$ and $\veps > 0$ are arbitrarily small we see that $A_i$ and $A_{i'}$ are approximately tangent at $a$. Thus $A_i = A_{i'}$ according to Proposition \ref{252} applied with $Q=1$. Finally if $i \in I_j$, $i' \in I_{j'}$, and $j\neq j'$ then $A_i(a) = f_j(a) \neq f_{j'}(a) = A_{i'}(a)$. The proof is complete.
\end{proof}

\begin{Example}
Consider
\begin{equation*}
f : \R \to \calQ_2(\R) : x \mapsto \begin{cases}
2 \lseg 0 \rseg & \text{ if } x < 0 \\
\lseg x^2 \rseg \oplus \lseg -x^2 \rseg & \text{ if } x \geq 0 \,.
\end{cases}
\end{equation*}
One readily checks that $f$ is (approximately) continuous at 0 and unambiguously (approximately) differentiable at 0, yet $\bsigma \circ f$ is not approximately constant at 0.
\end{Example}

We are ready to state and prove a useful generalization of Rademacher's Theorem. We recall that $X$ is a finite dimensional Banach space, and $Y$ a Banach space. In case $f : X \to \calQ_Q(Y)$ is approximately differentiable at $a \in X$ we let $Af(a) = \oplus_{i=1}^Q \lseg A_i \rseg$ and we define $L_i = A_i - A_i(0)$, $i=1,\ldots,Q$, the {\em linear part} of the affine approximation. We introduce the new notation
\begin{equation*}
Df(a) = \oplus_{i=1}^Q \lseg L_i \rseg \in \calQ_Q(\rmHom(X,Y))
\end{equation*}
where $\rmHom(X,Y)$ denotes the space of linear operators $X \to Y$ (these are automatically continuous). Letting the latter be equipped with some norm $\vvvert\cdot\vvvert$ we let
\begin{equation*}
\lno Df(a) \rno = \calG_2(Df(a),Q\lseg 0 \rseg) = \sqrt{\sum_{i=1}^Q \vvvert L_i \vvvert^2} \,.
\end{equation*}

\begin{Theorem}
\label{258}
Let $f : X \to \calQ_Q(Y)$ be Lipschitz continuous and assume that $Y$ has the Radon-Nikod\'ym property. It follows that
\begin{enumerate}
\item[(A)] For $\lambda$ almost every $a \in X$, $f$ is unambiguously approximately differentiable at $a$, and $Af(a)(a)=f(a)$;
\item[(B)] The map $X \to \calQ_Q(\rmHom(X,Y)) : x \mapsto Df(x)$ is $(\frB_X,\frB_{\calQ_Q(\rmHom(X,Y))}) $ measurable;
\item[(C)] If $f$ is approximately differentiable at $a \in X$ then it is differentiable at $a$ in the sense that
\begin{equation*}
\lim_{x \to a} \frac{\calG(f(x),Af(a)(x))}{\|x-a\|}=0 \,,
\end{equation*}
and $\lno Df(a) \rno \leq \sqrt{Q} \rmLip f$;
\item[(D)] For every injective Lipschitzian curve $\gamma : [0,1] \to X$ such that $\|\gamma'(t)\|=1$ and $f$ is approximately differentiable at $\gamma(t)$ for $\calL^1$ almost every $0 \leq t \leq 1$, one has
\begin{equation*}
\calG_2(\gamma(1),\gamma(0)) \leq \int_{\rmim \gamma} \lno Df(x) \rno d\calH^1(x) \,.
\end{equation*}
\end{enumerate}
\end{Theorem}

\begin{proof}
For each $k = 1,\ldots,Q$ define
\begin{equation*}
B_k = X \cap \{ x : \bsigma(f(x)) = k \} \,.
\end{equation*}
Notice $B_k$ is Borel since $\bsigma$ is Borel measurable according to Proposition \ref{3.2}. Fix $k$ and let $a \in B_k$. Put $\eta_a = \rmsplit f(a)$. Choose $0 < r < \frac{1}{2} \eta_a$ small enough for $\calG_1(f(x),f(a)) < \frac{1}{8} \eta_a$ whenever $x \in B(a,r)$. It follows from Lemma \ref{255} that 
there exist positive integers $Q_1, \ldots, Q_k$ such that $\sum_{j=1}^k Q_j = Q$ and each $f(x)$, $x \in B(a,r) \cap B_k$, can be decomposed as
\begin{equation*}
f(x) = \oplus_{j=1}^k Q_j \lseg f_j(x) \rseg
\end{equation*}
in such a way that $\|f_j(a) - f_j(x) \| < r$, $j=1,\ldots,k$, and
\begin{equation*}
\calG_1(f(x),f(a)) = \sum_{j=1}^k Q_j \|f_j(a) - f_j(x) \| \,.
\end{equation*}
In particular, for $j \neq j'$, we infer that
\begin{equation*}
\begin{split}
\|f_j(x)-f_{j'}(x) \| & \geq  \|f_j(a)-f_{j'}(a)\| - \| f_j(x)-f_j(a)\| - \|f_{j'}(a)-f_{j'}(x)\| \\
& \geq \frac{1}{2} \eta_a \,.
\end{split}
\end{equation*}
Thus $\eta_x := \rmsplit f(x) \geq \frac{1}{2} \eta_a$. If $x,x' \in B(a,r) \cap B_k$ then
\begin{equation*}
\|f_j(x)-f_j(x') \| \leq \|f_j(x)-f_j(a) \| + \|f_j(a) - f_j(x')\| < \frac{1}{4} \eta_a \leq \frac{1}{2} \eta_x
\end{equation*}
so that
\begin{equation*}
\calG_1(f(x),f(x')) = \sum_{j=1}^k Q_j \|f_j(x) - f_j(x') \|
\end{equation*}
according to the Splitting Lemma. Thus each $f_j$ is Lipschitz continuous on $B(a,r) \cap B_k$, and hence it is differentiable at $\lambda$ almost every point of $B(a,r) \cap B_k$ since it can be extended to the whole $X$ (ref ...) and $Y$ has the Radon-Nikod\'ym property. Now if each $f_j$ is differentiable 
at a density point $x$ of $B(a,r) \cap B_k$ one easily checks that
\begin{equation*}
g = \oplus_{j=1}^k Q_j \lseg f_j(a) + Df_j \rseg
\end{equation*}
is approximately tangent to $f$ at $x$. Thus we have shown that assumption (C) of Proposition \ref{256} occurs at $\lambda$ almost every $a \in X$. Since this is also the case of assumptions (A) and (B) (according to \cite[2.9.13]{GMT} and the Borel measurability of $f$ and of $\bsigma \circ f$), conclusion (A) is now a consequence of that proposition.
\par
In order to prove conclusion (B) we use the same notation $B_k$, $a \in B_k$ and $r > 0$ as above. It follows that the restriction
\begin{equation*}
Df : B_k \cap B(a,r) \to \calQ_Q(\rmHom(X,Y)) : x \mapsto \oplus_{i=1}^Q \lseg Df_j \rseg
\end{equation*}
is Borel measurable according to Proposition \ref{232}(A), because each $x \mapsto Df_j(x)$ is itself Borel measurable. Since $B_k$ is Lindel\"of the restriction $Df \restriction B_k$ is Borel measurable for each $k=1,\ldots,Q$, and the Borel measurability of $Df$ follows immediately.
\par
The proof of the first part of conclusion (C) is inspired by \cite[Lemma 3.1.5]{GMT} and exactly similar to \cite{GOB.06}. In order to prove the second part of conclusion (C) we assume that $f$ is differentiable at $a$ and we write $Af(a) = \oplus_{i=1}^Q \lseg A_i \rseg$ and $L_i = A_i - A_i(0)$, $i=1,\ldots,Q$. Observe that $A_i(h) = A_i(0) + L_i(h) = f_i(a) + L_i(h)$, $i=1,\ldots,Q$, according to (A). Observe that for each $x \in X$ we have
\begin{equation}
\label{eq.6}
\calG_2(Af(a)(x),f(a))^2 \leq \|x-a\|^2 \left( \sum_{i=1}^Q \|L_i\|^2 \right)
\end{equation}
and, given $x$, let $\sigma \in S_Q$ be a permutation such that
\begin{equation}
\label{eq.7}
\calG_2(Af(a)(x),f(a))^2 = \sum_{i=1}^Q \| f_i(a) + L_i(x-a) - f_{\sigma(i)}(a) \|^2 \,.
\end{equation}
Assuming that $f_i(a) \neq f_{\sigma(i)}(a)$, for some $i =1,\ldots,Q$, and that
\begin{equation*}
\|x-a\| \max \{\|L_1\|,\ldots,\|L_Q\|\} \leq \frac{1}{2} \rmsplit f(a)
\end{equation*}
 we infer that the right member of \eqref{eq.7} is bounded below by $\frac{1}{4}(\rmsplit f(a))^2$, in contradiction with \eqref{eq.6} provided $\|x-a\| \sqrt{\sum_{i=1}^Q \|L_i\|^2} < \frac{1}{2} \rmsplit f(a)$. Thus, if $\|x-a\|$ is small enough then \eqref{eq.7} becomes
\begin{equation*}
\begin{split}
\sqrt{\sum_{i=1}^Q \|L_i(x-a)\|^2} & = \calG_2(Af(a)(x),f(a)) \\
& \leq \calG_2(Af(a)(x),f(x)) + \calG_2(f(x),f(a)) \,.
\end{split}
\end{equation*}
Upon letting $x \to a$ we obtain
\begin{equation}
\label{eq.205}
\sup \left\{ \sqrt{\sum_{i=1}^Q \|L_i(h)\|^2} : h \in X \text{ and } \|h\| \leq 1 \right\} \leq \rmLip f \,.
\end{equation}
Let $j=1,\ldots,Q$ be such that $\|L_j\|=\max \{\|L_1\|,\ldots,\|L_Q\|\}$. The above inequality implies that $\|L_j\| \leq \rmLip f$. Finally $\lno Df(a) \rno = \left( \sum_{i=1}^Q \|L_i\|^2\right)^{\frac{1}{2}} \leq \sqrt{Q} \rmLip f$.
\par
It remains to establish conclusion (D). We define $g : [0,1] \to \R$ by the formula $g(t) = \calG_2(\gamma(t),\gamma(0))$, $0 \leq t \leq 0$. We will show that $g$ is Lipschitzian and that $|g'(t)| \leq \lno Df(\gamma(t)) \rno$ at each $t$ such that $f$ is differentiable at $\gamma(t)$, so that our conclusion will become a consequence of a Theorem of Lebesgue applied to $g$:
\begin{multline*}
\calG_2(\gamma(1),\gamma(0)) = g(1)-g(0) = \int_0^1 g'(t)d\calL^1(t) \\
\leq \int_0^1 \lno Df(\gamma(t)) \rno d\calL^1(t) = \int_{\rmim \gamma} \lno Df(x) \rno d\calH^1(x)
\end{multline*}
according to the area formula applied to $\gamma$. Write $Df(\gamma(t)) = \oplus_{i=1}^Q \lseg L_i(\gamma(t)) \rseg$. For each $t,t+h \in [0,1]$ one has
\begin{equation*}
\begin{split}
g(t+h)-g(t) & = \calG_2(f(\gamma(t+h)),\gamma(0)) - \calG_2(f(\gamma(t)),\gamma(0)) \\
& \leq \calG_2(f(\gamma(t+h)),f(\gamma(t))\\
\intertext{which shows that $\rmLip g \leq \rmLip (f \circ \gamma)$; and assuming further that $f$ is differentiable at $\gamma(t)$, we obtain:}
& \leq \calG_2(Af(\gamma(t))(\gamma(t+h)),f(\gamma(t))) \\
&\quad + \calG_2(Af(\gamma(t))(\gamma(t+h)),f(\gamma(t+h)))\\
& \leq \left( \sum_{i=1}^Q \|L_i(\gamma(t))(\gamma(t+h)-\gamma(t))\|^2 \right)^\frac{1}{2}\\
&\quad + \veps \| \gamma(t+h)-\gamma(t)\|
\end{split}
\end{equation*}
where the last inequality holds provided $h$ is small enough according to $\veps$, $\rmsplit f(\gamma(t))$ and $\|L_1(\gamma(t)\|,\ldots,\|L_Q(\gamma(t))\|$ (recall the proof of (C)). Dividing by $|h|$, letting $h \to 0$, and recalling that $\rmLip \gamma \leq 1$ we infer that $|g'(t)| \leq \lno Df(\gamma(t)) \rno$ provided that $g$ is differentiable at $t$.
\end{proof}

Given $f : X \to \calQ_Q(Y)$ and $a \in X$ we now define
\begin{equation*}
\rmlip_a f := \limsup_{r \to 0} \sup_{x \in B(x,r)} \frac{\calG(f(x),f(a))}{\|x-a\|} \,.
\end{equation*}
If $f$ is Lipschitz then clearly $\rmlip_a f \leq \rmLip f < \infty$ for every $a \in X$. We leave it to the reader to check the following partial ``product rule'': if $f$ and $\lambda : X \to \R$ are Lipschitz then
\begin{equation}
\label{eq.204}
\rmlip_a (\lambda f) \leq (\rmlip_a \lambda) \lno f(a) \rno + | \lambda(a) | (\rmlip_a f) \,.
\end{equation}

\begin{Proposition}
\label{259}
If $f : X \to \calQ_Q(Y)$ is Lipschitz then
\begin{equation*}
\rmlip_a f \leq \lno Df(a) \rno \leq \sqrt{Q} (\rmlip_a f)
\end{equation*}
for $\calL^m$ almost every $a \in X$.
\end{Proposition}

\begin{proof}
The second inequality is proved in exactly the same as Theorem \ref{258}(C) on noticing that in \eqref{eq.205} $\rmLip f$ can be replaced by $\rmlip_a f$. In order to prove the first inequality we assume that $f$ is differentiable at $a$ and that $a$ is a Lebesgue point of $x \mapsto \lno Df(x) \rno$. Given $0 < \veps < 1$ we define
\begin{equation*}
G_\veps = X \cap \{ x : \lno Df(x) \rno \leq \veps + \lno Df(a) \rno \} \,.
\end{equation*}
There exists $r_0 > 0$ such that for every $0 < r \leq r_0$ one has
\begin{equation*}
\calL^m(B(a,r) \cap G_\veps^c) \leq 2^{-m}\veps^m \balpha(m-1) r^m \,.
\end{equation*}
Fix $0 < r \leq r_0/2$. Given $x \in B(a,r)$, $x \neq a$, we put $\rho = \|x-a\|$ and we consider the set
\begin{equation*}
H = B(0,\veps \rho) \cap \rmspan \{x-a\}^\perp \,.
\end{equation*}
With each $h \in H$ we associate the line segment $S_j$ joining $a+h$ and $x+h$, and we define the ``cylinder''
\begin{equation*}
C = \cup_{h \in H} S_h \,.
\end{equation*}
We observe that $C \subset B(a,2\rho)$ and that
\begin{equation*}
\calL^m(C) = \rho \balpha(m-1) \veps^{m-1} \rho^{m-1} = \veps^{m-1} \balpha(m-1) \rho^{m} \,.
\end{equation*}
Therefore,
\begin{equation*}
\begin{split}
\calL^m(C \cap G_\veps) & = \calL^m(C) - \calL^m(C \cap G_\veps^c) \\
& \geq \veps^{m-1} \balpha(m-1) \rho^m - 2^{-m}\veps^m \balpha(m-1) (2 \rho)^m \\
& = \veps^{m-1} \balpha(m-1)\rho^{m-1}(1 - \veps) \rho \,.
\end{split}
\end{equation*}
According to Fubini's Theorem, Chebyshev's inequality and Theorem \ref{258}, there exists $h \in H$ such that
\begin{equation*}
\calH^1(S_h \cap G_\veps) \geq (1- \veps) \rho
\end{equation*}
and $f$ is differentiable $\calH^1$ almost everywhere on $S_h$. For such $h$, recalling Theorem \ref{258}(D), we infer that
\begin{equation*}
\begin{split}
\calG(f(x+h),f(a+h)) & \leq \int_{S_h} \lno Df(z) \rno d\calH^1(z) \\
& = \int_{S_h \cap G_\veps} \lno Df(z) \rno d\calH^1(z) + \int_{S_h \cap G_\veps^c} \lno Df(z) \rno d\calH^1(z) \\
& \leq (\veps + \lno Df(a) \rno ) \rho + \sqrt{Q} (\rmLip f) \veps \rho \,.
\end{split}
\end{equation*}
Since $\|h\| \leq \veps \rho$, the triangle inequality implies that
\begin{equation*}
\calG(f(x),f(a)) \leq \leq (\veps + \lno Df(a) \rno ) \rho + (2 + \sqrt{Q}) (\rmLip f) \veps \rho \,,
\end{equation*}
thus
\begin{equation*}
\frac{\calG(f(x),f(a))}{\|x-a\|} \leq \veps + \lno Df(a) \rno+ (2 + \sqrt{Q}) (\rmLip f) \veps \,.
\end{equation*}
\end{proof}

\section{Embeddings}

\subsection{Whitney bi-H\"older embedding --- The case $Y = \ell_2^n(\K)$}
\label{31}
Here we report on \cite[Appendix V]{WHITNEY.72}. We let $\K = \R$ or $\K = \C$. We start by recalling the usual embedding
\begin{equation*}
\boldeta : \calQ_Q(\K) \to \K^Q : v \mapsto (\eta_1(v),\ldots,\eta_Q(v)) \,.
\end{equation*}
Given $v = \lseg x_1,\ldots,x_Q \rseg$ we let $\eta_i(v) \in \K$, $i=1,\ldots,Q$, be the coefficients of the Weierstrass polynomial of $v$:
\begin{equation*}
P_v(x) = \prod_{i=1}^Q (x-x_i) = x^Q + \sum_{i=1}^Q \eta_i(v) x^{Q-i} \in \K[x] \,.
\end{equation*}
Readily the $\eta_i(v)$ are the $Q$ symmetric functions of $Q$ variables, and their (Lipschitz) continuity follows. In case $\K = \C$, $\boldeta$ is a bijection and $\boldeta^{-1}$ is H\"older continuous (see e.g. \cite[Theorem (1,4)]{MARDEN}).
\par
We now treat the case of $\K^n$. We will define a mapping
\begin{equation*}
\boldeta : \calQ_Q(\K^n) \to \K^N
\end{equation*}
where $N=N(n,Q)$. Given $u \in \C^n$ and $v = \lseg x_1,\ldots,x_Q \rseg \in \calQ_Q(\K^n)$ we define a polynomial
\begin{equation*}
P_v(u,x) = \prod_{i=1}^Q ( x- \la u,x_i \ra) \in \K[u_1,\ldots,u_n,x]
\end{equation*}
whose coefficients $\eta_\alpha(v)$ form the components of $\boldeta$:
\begin{equation*}
P_v(u,x) =x^Q + \sum_{i=1}^Q \sum_{\substack{ \alpha \in \N^n \\ |\alpha|=i}} \eta_\alpha(v) u_1^{\alpha_1} \ldots u_n^{\alpha_n} x^{Q-i} \,.
\end{equation*}
One computes that
\begin{equation*}
N(n,Q) = \bin{Q+n}{n} - 1 \,.
\end{equation*}
One shows (\cite[Appendix V Theorem 6A]{WHITNEY.72}) that $\boldeta$ is injective, continuous, that $\boldeta(\calQ_Q(\K^n))$ is closed in $\K^N$, and that $\boldeta^{-1}$ is continuous as well. In case $\K=\C$, it follows from the Proper Mapping Theorem that $\boldeta(\calQ_Q(\C^n))$ is an irreducible analytic variety in $\C^N$, see \cite[Chapter 5 Theorem 5A]{WHITNEY.72}. In fact $\boldeta(\calQ_Q(\C^n))$ is a H\"older continuous retract of $\C^N$, see Remark \ref{337}.

\subsection{Splitting in case $Y = \R$}

We now state an easy and important observation on how to compute the $\calG_2$ distance of two members of $\calQ_Q(\R)$. The order of $\R$ plays the essential role. This is taken from \cite[1.1(4)]{ALMGREN}.

\begin{Proposition}
\label{7.1}
Let $v,v' \in \calQ_Q(\R)$ and choose numbering $v = \lseg y_1,\ldots,y_Q \rseg$ and $v' = \lseg y'_1,\ldots,y'_Q \rseg$ such that $y_1 \leq y_2 \leq \ldots \leq y_Q$ and $y'_1 \leq y'_2 \leq \ldots \leq y'_Q$. It follows that
\begin{equation*}
\calG_2(v,v') = \sqrt{ \sum_{i=1}^Q |y_i - y'_i |^2} \,.
\end{equation*}
\end{Proposition}

\subsection{Almgren-White locally isometric embedding --- The case $Y = \ell_2^n(\R)$}
\label{aw-embedding}

This section is devoted to the case $Y = \ell_2^n$, i.e. $\Rn$ equipped with its Euclidean norm $\|\cdot\|$ and inner product $\la \cdot , \cdot \ra$. Proposition \ref{8.1} and Theorem \ref{334} are due to F.J. Almgren \cite[1.2]{ALMGREN}. The presentation we give here is (inspired by) that of C. De Lellis and E.N. Spadaro \cite{DEL.SPA.11}. Part (B) of Theorem \ref{334} is due to B. White \cite{WHITE.PERSONAL}.
\par
Let $e \in \Rn$ be such that $\|e\|=1$. We define a map
\begin{equation*}
\bpi_e : \calQ_Q(\Rn) \to \R^Q
\end{equation*}
by the requirement that $\bpi_e(\lseg y_1,\ldots,y_Q \rseg)$ be the list of inner products
\begin{equation*}
\la y_1,e \ra,\ldots, \la y_Q , e \ra \,.
\end{equation*}
numbered in increasing order.
Notice that we need indeed to explain how we choose to order these real numbers if we want the values of $\bpi_e$ to belong to $\R^Q$, for otherwise they would merely belong to $\calQ_Q(\R)$.

\begin{Proposition}
\label{8.1}
Let $e_1,\ldots,e_n$ be an orthonormal basis of $\Rn$. The mapping
\begin{equation*}
\bxi_0 : \calQ_Q(\Rn) \to \R^{Qn} : v \mapsto (\bpi_{e_1}(v),\ldots,\bpi_{e_n}(v))
\end{equation*}
has the following properties:
\begin{enumerate}
\item[(A)] $\rmLip \bxi_0 = 1$;
\item[(B)] For every $v \in \calQ_Q(\Rn)$ there exists $r > 0$ such that for each $v' \in \calQ_Q(\Rn)$, if $\calG_2(v,v') < r$ then $\|\bxi_0(v) - \bxi_0(v')\| = \calG_2(v,v')$;
\item[(C)] For every $v \in \calQ_Q(\Rn)$ one has $\|\bxi_0(v)\| = \calG_2(v,Q\lseg 0 \rseg)$.
\end{enumerate}
\end{Proposition}

\begin{proof}
(A) Let $v,v' \in \calQ_Q(\Rn)$ and write $v= \lseg y_1,\ldots,y_Q \rseg$ and $v' = \lseg y'_1,\ldots,y'_Q \rseg$. For each $j=1,\ldots,n$ there exists $\tau_j \in S_Q$ such that $\la y_{\tau_j(1)},e_j \ra \leq \ldots \leq \la y_{\tau_j(Q)} , e_j \ra$ and there exists $\tau'_j \in S_Q$ such that $\la y'_{\tau'_j(1)},e_j \ra \leq \ldots \leq \la y'_{\tau'_j(Q)} , e_j \ra$. By definition of $\bpi_{e_j}$ we have
\begin{equation*}
\| \bpi_{e_j}(v) - \bpi_{e_j}(v') \|^2 = \sum_{i=1}^Q | \la y_{\tau_j(i)},e_j \ra - \la y'_{\tau'_j(i)},e_j \ra |^2 \,.
\end{equation*}
There also exists $\sigma \in S_Q$ such that
\begin{equation*}
\calG_2(v,v')^2 = \sum_{i=1}^Q \|y_i - y'_{\sigma(i)} \|^2 \,.
\end{equation*}
It remains to observe that
\begin{equation*}
\begin{split}
\| \bxi_0(v) - \bxi_0(v') \|^2 & = \sum_{j=1}^n \sum_{i=1}^Q | \la y_{\tau_j(i)},e_j \ra - \la y'_{\tau'_j(i)},e_j \ra |^2 \\
\intertext{which, by Proposition \ref{7.1}, is bounded by}
& \leq \sum_{j=1}^n \sum_{i=1}^Q | \la y_i,e_j \ra - \la y_{\sigma(i)} , e_j \ra|^2 \\
& = \sum_{i=1}^Q \|y_i - y_{\sigma(i)} \|^2 \\
& = \calG_2(v,v')^2 \,.
\end{split}
\end{equation*}
\par
(B) Let $v \in \calQ_Q(\Rn)$ and write $v = \lseg y_1,\ldots,y_Q \rseg$. For each $j=1,\ldots,n$ choose $\tau_j \in S_Q$ such that $\la y_{\tau_j(1)} , e_j \ra \leq \ldots \leq \la y_{\tau_j(Q)} , e_j \ra$. Define $r = \frac{1}{2} \min \{ \rmsplit \pi_{e_j}(v) : j=1,\ldots,n \}$ and let $v' \in \calQ_Q(\Rn)$ be such that $\calG_2(v,v') < r$. Choose a numbering $v' = \lseg y'_1,\ldots,y'_Q \rseg$ so that
\begin{equation*}
\calG_2(v,v')^2 = \sum_{i=1}^Q \|y_i - y'_i \|^2 \,.
\end{equation*}
Notice that for every $j=1,\ldots,n$ one has
\begin{multline*}
\max_{i=1,\ldots,Q}| \la y_{\tau_j(i)} , e_j \ra - \la y'_{\tau_j(i)} , e_j \ra | \\ \leq \max_{i=1,\ldots,Q} \| y_{\tau_j(i)} - y'_{\tau_j(i)} \| \leq \calG_2(v,v') < \frac{1}{2} \rmsplit \bpi_{e_j}(v)
\end{multline*}
which implies, according to the Splitting Lemma, Proposition \ref{7.1} and the definition of $\bpi_{e_j}$, that
\begin{equation*}
\sum_{i=1}^Q | \la y_{\tau_j(i)} , e_j \ra - \la y'_{\tau_j(i)} , e_j \ra |^2 = \calG_2(\bpi_{e_j}(v),\bpi_{e_j}(v'))^2 = \sum_{i=1}^Q | \la y_{\tau_j(i)} , e_j \ra - \la y'_{\tau'_j(i)} , e_j \ra |^2
\end{equation*}
where $\tau'_j \in S_Q$ is such that $\la y'_{\tau'_j(1)} , e_j \ra \leq \ldots \leq \la y'_{\tau'_j(Q)} , e_j \ra$. Therefore,
\begin{equation*}
\begin{split}
\| \bxi_0(v) - \bxi_0(v') \|^2 & = \sum_{j=1}^n \sum_{i=1}^Q | \la y_{\tau_j(i)} , e_j \ra - \la y'_{\tau'_j(i)} , e_j \ra |^2 \\
& = \sum_{j=1}^n \sum_{i=1}^Q | \la y_{\tau_j(i)} , e_j \ra - \la y'_{\tau_j(i)} , e_j \ra |^2 \\
& = \sum_{i=1}^Q\sum_{j=1}^n  | \la y_i , e_j \ra - \la y'_i , e_j \ra |^2 \\
& = \sum_{i=1}^Q \|y_i - y'_i\|^2 \\
& = \calG_2(v,v')^2 \,.
\end{split}
\end{equation*}
\par
(C) Writing $v = \lseg y_1,\ldots,y_Q \rseg$, it suffices to observe that
\begin{equation*}
\|\bxi_0(v)\|^2  = \sum_{j=1}^n \sum_{i=1}^Q | \la y_i,e_j \ra|^2
 = \sum_{i=1}^Q \|y_i\|^2
 = \calG_2(v,Q\lseg 0 \rseg)^2 \,.
\end{equation*}
\end{proof}

\begin{Remark}
The Lipschitz mapping $\bxi_0$ defined above is usually not injective. Consider for instance the case when $Q=2$, $n=2$, and let $e_1,e_2$ be an orthonormal basis of $\R^2$. We define $v = \lseg -e_1+e_2 , e_1 \rseg$. It follows that $\bxi_0(v) = (-1,1,0,1) = \bxi_0(v')$ where $v' = \lseg -e_1 , e_1+e_2 \rseg$. Clearly $v \neq v'$.
\end{Remark}

The lack of injectivity of $\bxi_0$ is overcome by considering a lot of orthonormal bases instead of just one, i.e. we shall replace $\bxi_0$ by many copies of $\bxi_0$ corresponding to various bases. The main observation to obtain injectivity is the following.

\begin{Proposition}
\label{8.3}
Given integers $n$ and $L$ there are $\veps > 0$ and unit vectors $e_1,\ldots,e_K \in \bbS^{n-1}$ with the following property. For every $v_1,\ldots,v_L \in \Rn$ there exists $k = 1,\ldots,K$ such that
\begin{equation*}
| \la e_k , v_l \ra | \geq \veps \|v_l\|
\end{equation*}
for each $l=1,\ldots,L$.
\end{Proposition}

\begin{proof}
We first notice that the measure $\calH^{n-1} \hel \bbS^{n-1}$ is doubling, i.e. there exists $C \geq 1$ such that $\calH^{n-1}(\bbS^{n-1} \cap U(e,2r)) \leq C \calH^{n-1}(\bbS^{n-1} \cap U(e,r))$ whenever $e \in \bbS^{n-1}$ and $r > 0$. Given $e \in \bbS^{n-1}$ and $\veps > 0$ we define the slab
\begin{equation*}
S_{e,\veps} = \bbS^{n-1} \cap \{ w : | \la e,w \ra | < \veps \} \,.
\end{equation*}
Now we choose $\veps > 0$ small enough for
\begin{equation*}
\calH^{n-1}(S_{e,\veps}) \leq \frac{\calH^{n-1}(\bbS^{n-1})}{3CL}
\end{equation*}
whenever $e \in \bbS^{n-1}$.
 We choose a maximal collection of points $e_1,\ldots,e_k \in \bbS^{n-1}$ such that the (open) balls $U(e_k,\veps)$, $k=1,\ldots,K$, are pairwise disjoint. Such a collection exists because $\calH^{n-1}(\bbS^{n-1})$ is finite and $\calH^{n-1}(\bbS^{n-1} \cap U(e,\veps))$ does not depend on $e \in \bbS^{n-1}$. By maximality, we have that $\bbS^{n-1} = \bigcup_{k=1}^K U(e_k,2\veps)$.
\par
Let now $v_1,\ldots,v_L \in \Rn$ be arbitrary. We define $\frL = \{1,\ldots,L \} \cap \{ l : v_l \neq 0 \}$ and for $l \in \frL$ we set $w_l = v_l |v_l|^{-1}$. Our claim is that for some $k$, $e_k$ does not belong to any of the slabs $S_{w_l,\veps}$, $l \in \frL$. Suppose if possible that for each $k=1,\ldots,K$, $e_k \in S$ where
\begin{equation*}
S = \bigcup_{l \in \frL} S_{w_l,\veps} \,.
\end{equation*}
If $l \in \frL$ corresponds to $k$ so that $e_k \in S_{w_l,\veps}$ then in fact at least ``half'' the ball $U(e_k,\veps)$ must be contained in $S_{w_l,\veps}$, thus $\calH^{n-1}(S \cap U(e_k,\veps)) \geq \frac{1}{2} \calH^{n-1}(\bbS^{n-1} \cap U(e_k,\veps))$. We would then obtain
\begin{equation*}
\begin{split}
\calH^{n-1}(\bbS^{n-1}) & \leq \sum_{k=1}^K \calH^{n-1}(\bbS^{n-1} \cap U(e_k,2\veps)) \\
& \leq C \sum_{k=1}^K \calH^{n-1}(\bbS^{n-1} \cap U(e_k,\veps)) \\
& \leq 2C \sum_{k=1}^K \calH^{n-1}(S \cap U(e_k,\veps)) \\
& \leq 2C \calH^{n-1}(S)\\
& \leq 2C \sum_{l \in \frL} \calH^{n-1}(S_{w_l,\veps}) \\
& \leq \frac{2}{3} \calH^{n-1}(\bbS^{n-1}) \,,
\end{split}
\end{equation*}
a contradiction.
\end{proof}

\begin{Theorem}
\label{334}
There exist an integer $N=N(n,Q)$, a real number $\alpha=\alpha(n,Q) \leq 1$ and a mapping
\begin{equation*}
\bxi : \calQ_Q(\Rn) \to \R^N
\end{equation*}
with the following properties.
\begin{enumerate}
\item[(A)] For every $v,v' \in \calQ_Q(\Rn)$, $\alpha \calG_2(v,v') \leq \|\bxi(v)-\bxi(v')\| \leq  \calG_2(v,v')$;
\item[(B)] For every $v \in \calQ_Q(\Rn)$ there exists $r > 0$ such that for each $v' \in \calQ_Q(\Rn)$, if $\calG_2(v,v') < r$ then $\|\bxi(v)-\bxi(v')\| = \calG_2(v,v')$;
\item[(C)] For every $v \in \calQ_Q(\Rn)$ one has $\|\bxi(v)\| = \calG_2(v,Q\lseg 0 \rseg)$.
\end{enumerate}
\end{Theorem}

\begin{proof}
Letting $L=Q^2$ we choose $\veps$ and $e_1,\ldots,e_K$ according to Proposition \ref{8.3}. For each $k=1,\ldots,K$ we choose an orthonormal basis $e_{1,k},\ldots,e_{n,k}$ of $\Rn$ such that $e_{1,k}=e_k$. We then define
\begin{equation*}
\bxi : \calQ_Q(\Rn) \to \R^N : v \mapsto (\bxi_1(v),\ldots,\bxi_K(v))
\end{equation*}
where $N = QnK$ and we have abbreviated $\bxi_k(v) = (\bpi_{e_{1,k}}(v),\ldots,\bpi_{e_{n,k}}(v))$. Thus each $\bxi_k$ is a mapping of the type $\bxi_0$ considered in Proposition \ref{8.1}, corresponding to the basis $e_{1,k},\ldots,e_{n,k}$. We therefore infer from Proposition \ref{8.1}(A) that for every $v,v' \in \calQ_Q(\Rn)$,
\begin{equation*}
\|\bxi(v)-\bxi(v')\|^2 = \sum_{k=1}^K \|\bxi_k(v)-\bxi_k(v')\|^2 \leq K \calG_2(v,v')^2 \,.
\end{equation*}
On the other hand, letting $v = \lseg y_1,\ldots,y_Q \rseg$ and $v' = \lseg y'_1,\ldots,y'_Q \rseg$, we infer from Proposition \ref{8.3} that
there exists $k = 1,\ldots, K$ such that
\begin{equation*}
| \la e_{1,k} , y_i - y'_j \ra | \geq \veps \| y_i - y'_j \|
\end{equation*}
for every $i,j = 1,\ldots,Q$. Let $\sigma \in S_Q$ be such that $\la y_{\sigma(1)} , e_{1,k} \ra \leq \ldots \leq \la y_{\sigma(Q)} , e_{1,k} \ra$ and let $\tau \in S_Q$ be such that $\la y'_{\tau(1)} , e_{1,k} \ra \leq \ldots \leq \la y'_{\tau(Q)} , e_{1,k} \ra$. Observe that
\begin{equation*}
\begin{split}
\calG_2(v,v')^2 & \leq \sum_{i=1}^Q \|y_{\sigma(i)} - y'_{\tau(i)} \|^2 \\
& \leq \veps^{-2} \sum_{i=1}^Q | \la y_{\sigma(i)} , e_{1,k} \ra - \la y'_{\tau(i)} , e_{1,k} \ra |^2\\
& = \veps^{-2} \|\bpi_{e_{1,k}}(v) - \bpi_{e_{1,k}}(v') \|^2 \\
& \leq \veps^{-2} \|\bxi(v)-\bxi(v)\|^2 \,.
\end{split}
\end{equation*}
\par
We now turn to proving conclusions (B) and (C). Given $v \in \calQ_Q(\Rn)$ and $k=1,\ldots,K$ we choose $r_k > 0$ according to Proposition \ref{8.1}(B). Let $r = \min\{r_1,\ldots,r_K\}$. If $v \in \calQ_Q(\Rn)$ and $\calG_2(v,v') < r$ then
\begin{equation*}
\|\bxi(v)-\bxi(v')\|^2 = \sum_{k=1}^K \|\bxi_k(v)-\bxi_k(v')\|^2 = K \calG_2(v,v')^2 \,.
\end{equation*}
Also, regarding conclusion (C), we observe that for every $v \in \calQ_Q(\Rn)$,
\begin{equation*}
\|\bxi(v)\|^2 = K \calG_2(v,Q \lseg 0 \rseg) \,,
\end{equation*}
according to Proposition \ref{8.1}(C).
This means that the mapping $K^{-1/2}\bxi$ verifies the conclusions of the present proposition.
\end{proof}

B. White's addition (B) to F.J. Almgren's embedding Theorem \ref{334} has the following rather useful consequence. Here the linear spaces $\rmHom(\Rm,\R^\nu)$ ($\nu=n$ or $\nu=N$) are equipped with the norm
\begin{equation*}
\vvvert L \vvvert = \sqrt{ \sum_{j=1}^m \sum_{k=1}^\nu \la L(e_j), e_k \ra^2}
\end{equation*}
corresponding to the canonical bases of $\Rm$ and $\R^\nu$.

\begin{Proposition}
\label{335}
Assume that $f : \Rm \to \calQ_Q(\Rn)$, $a \in \Rn$, and that both $f$ and $\bxi \circ f$ are differentiable at $a$\footnote{For $f$ this is in the sense of Theorem \ref{258}(C)}. It follows that
\begin{equation*}
\lno Df(a) \rno = \vvvert D(\bxi \circ f) (a) \vvvert \,.
\end{equation*}
\end{Proposition}

\begin{proof}
For each $j=1,\ldots,m$ we have
\begin{equation*}
\begin{split}
\| \partial_j(\bxi \circ f)(a) \|^2 & = \lim_{t \to 0} \frac{ \| (\bxi \circ f)(a + te_j) - (\bxi \circ f)(a)\|^2}{t^2} \\
&= \lim_{t \to 0} \frac{ \calG_2(f(a+te_j),f(a))^2}{t^2} \\
\intertext{(according to Theorem \ref{334}(B))}
& = \lim_{t \to 0} \frac{\calG_2(Af(a)(a+te_j),f(a))^2}{t^2} \\
\intertext{(because $f$ is differentiable at $a$)}
& = \lim_{t \to 0}  \frac{\sum_{i=1}^Q \| f_i(a) - A_{\sigma_t(i)}(a) - L_{\sigma_t(i)}(te_j)\|^2}{t^2} \,,
\end{split}
\end{equation*}
where, as usual, $Df(a) = \oplus_{i=1}^Q \lseg A_i \rseg$, $L_i = A_i - A_i(0)$, $i=1,\ldots,Q$ and $\sigma_t$ is a permutation $\sigma \in S_Q$ for which the quantity
 \begin{equation*}
  \frac{\sum_{i=1}^Q \| f_i(a) - A_{\sigma(i)}(a) - L_{\sigma(i)}(te_j)\|^2}{t^2}
  \end{equation*}
  is minimal. Since the above limit exists and is finite, we infer that $\sigma_t \in S_Q$ must be such that $f_i(a) = A_{\sigma_t(i)}(a)$ when $t$ is small enough, $i=1,\ldots,Q$. Thus,
\begin{equation*}
\|\partial_j (\bxi \circ f)(a)\|^2 = \sum_{i=1}^Q \|L_i(e_j)\|^2 \,,
\end{equation*}
and in turn,
\begin{multline*}
\vvvert D(\bxi \circ f)(a) \vvvert^2 = \sum_{j=1}^m \|\partial_j (\bxi \circ f)(a) \|^2 = \sum_{j=1}^m \sum_{i=1}^Q \|L_i(e_j)\|^2 \\= \sum_{i=1}^Q \vvvert L_i \vvvert^2 = \lno Df(a) \rno^2 \,.
\end{multline*}
\end{proof}

\begin{Theorem}
\label{336}
Let $N=N(n,Q)$ and $\bxi$ be as in Theorem \ref{334}. There exists a Lipschitz retraction
\begin{equation*}
\brho : \R^N \to \bxi(\calQ_Q(\Rn)) \,.
\end{equation*}
\end{Theorem}

\begin{proof}
Apply Theorem \ref{243} with $X=\ell_2^N$, $A = \bxi(\calQ_Q(\Rn))$, $Y = \ell_2^n$ and
 $f = \bxi^{-1}$. Letting $\hat{f}$ be a Lipschitz extension of $f$, the mapping
  $\brho = \bxi \circ \hat{f}$ verifies the conclusion.
\end{proof}

\begin{Remark}
\label{337}
The exact same proof shows that there exists a H\"older continuous retraction
\begin{equation*}
\tilde{\brho} : \C^N \to \boldeta(\calQ_Q(\C^n))
\end{equation*}
where $N=N(n,Q)$ and $\boldeta$ are as in Section \ref{31}. This follows indeed from the fact that $\boldeta^{-1}$ is H\"older continuous (reference \cite[Theorem (1,4)]{MARDEN}). In the same vein one can prove the following, based on \cite[Theorem 1.12]{BENYAMINI.LINDENSTRAUSS} and Theorem \ref{336}: {\em If $\omega : \R^+ \to \R^+$ is concave then for every $A \subset \ell_2^m$ and every $f : A \to \calQ_Q(\ell_2^n)$ such that $\rmosc(f;\cdot) \leq \omega$, there exists an extension $\hat{f} : \ell_2^m \to \calQ_Q(\ell_2^n)$ of $f$ such that $\rmosc(\hat{f};\cdot) \leq (\rmLip \brho_{n,Q}) \omega$.} Here $\brho_{n,Q}$ is the Lipschitz retraction of Theorem \ref{336}, and $\calQ_Q(\ell_2^n)$ is equipped with its metric $\calG_2$.
\end{Remark}

We recall that a metric space $Z$ is an {\em absolute Lipschitz retract} if and only if each isometric embedding $Z \to Z'$ into another metric space $Z'$ has a Lipschitz right inverse $\rho : Z' \to Z$. In other words, $Z$ is a Lipschitz retract of any of its metric superspaces. This is equivalent to asking that any partially defined Lipschitz map into $Z$ extends to a Lipschitz map into $Z$, see \cite[Proposition 1.2]{BENYAMINI.LINDENSTRAUSS}. For instance $\ell_\infty^N$ is an absolute Lipschitz retract, and hence the following holds.

\begin{Corollary}
$\calQ_Q(\Rn)$ is an absolute Lipschitz retract.
\end{Corollary}

\begin{Question}
\label{339}
If $Y$ is an absolute Lipschitz retract, is $\calQ_Q(Y)$ also one? Are $\calQ_Q(\ell_\infty)$ and $\calQ_Q(C[0,1])$ absolute Lipschitz retracts? Are they absolute uniform retracts?
\end{Question}

\subsection{Lipeomorphic embedding into $\rmLip_{y_0}(Y)^*$}

Let $(Y,y_0)$ be a pointed metric space, i.e. a metric space $Y$ together with a distinguished point $y_0 \in Y$. We denote by $\rmLip_{y_0}(Y)$ the collection of those Lipschitz continuous functions $u : Y \to \R$ vanishing at $y_0$. This is a Banach space equipped with the norm $\|u\|_{\rmLip} = \rmLip u$.
With each $v = \lseg y_1,\ldots,y_Q \rseg \in \calQ_Q(Y)$ we associate a linear functional
\begin{equation}
\label{eq.11}
\bzeta(v) : \rmLip_{y_0}(Y) \to \R : u \mapsto \sum_{i=1}^Q u(y_i) \,.
\end{equation}
One readily checks that $\bzeta(v)$ is continuous and
\begin{equation*}
\|\bzeta(v) \|_{(\rmLip Y)^*} \leq \sum_{i=1}^Q d(y_i,y_0) \,.
\end{equation*}
In particular $\|\bzeta(v) \|_{(\rmLip Y)^*} \leq Q (\rmdiam Y)$ so that $\bzeta$ is bounded when $Y$ is. Notice also that $\bzeta(Q \lseg y_0 \rseg) = 0$.
We shall now show that
\begin{equation*}
\bzeta : \calQ_Q(Y) \to \rmLip_{y_0}(Y)^*
\end{equation*}
is a lipeomorphic embedding.

\begin{Theorem}
\label{341}
There exists $c_Q > 0$ such that for every pointed metric space $(Y,y_0)$ and every $v,v' \in \calQ_Q(Y)$ one has
\begin{equation*}
c_Q \calG_1(v,v') \leq \|\bzeta(v)-\bzeta(v') \|_{(\rmLip Y)^*} \leq \calG_1(v,v') \,.
\end{equation*}
\end{Theorem}

\begin{proof}
We start with the second inequality. Let $v,v' \in \calQ_Q(Y)$ and choose numberings $v=\lseg y_1,\ldots,y_Q \rseg$ and $v' = \lseg y'_1,\ldots,y'_Q \rseg$ so that $\calG_1(v,v') = \sum_{i=1}^Q d(y_i,y'_i)$. It is clear that
\begin{equation*}
\begin{split}
\|\bzeta(v) - \bzeta(v') \|_{(\rmLip Y)^*} & = \sup \Bigg\{ \left| \sum_{i=1}^Q u(y_i) - \sum_{i=1}^Q u(y'_i) \right| : u \in \rmLip_{y_0}(Y) \\
& \qquad\qquad\qquad\qquad\qquad\qquad\qquad\qquad\text{ and } \rmLip u \leq 1 \Bigg\} \\
& \leq \sum_{i=1}^Q d(y_i,y'_i) \\
& = \calG_1(v,v') \,.
\end{split}
\end{equation*}
\par
We now turn to proving the first inequality, by induction on $Q$. If $Q=1$ then the inequality is verified with $c_1=1$. Indeed, given $v = \lseg y_1 \rseg$ and $v' = \lseg y'_1 \rseg$ we let $u(y) = d(y_1,y)-d(y_1,y_0)$ so that $u \in \rmLip_{y_0}(Y)$, $\rmLip u \leq 1$, and
\begin{equation*}
\|\bzeta(v) - \bzeta(v')\|_{(\rmLip Y)^*} \geq | u(y_1) - u(y'_1) | = d(y_1,y'_1) = \calG_1(v,v') \,.
\end{equation*}
We now assume the conclusion holds for $Q$ and we establish it for $Q+1$. Let $v,v' \in \calQ_{Q+1}(Y)$ and write $v = \oplus_{i=1}^{Q+1} \lseg y_i \rseg$ and $v' = \oplus_{i=1}^{Q+1} \lseg y'_i \rseg$. We let $\alpha > 0$ to be determined later, and we distinguish between two cases.
\par
{\em First case}. We assume that
\begin{equation*}
\rmdist( \rmsupp \mu_v , \rmsupp \mu_{v'} ) = \min \{ d(y_i,y'_j) : i,j =1,\ldots,Q+1 \} > \alpha \calG_1(v,v') \,.
\end{equation*}
We define $u_0 : (\rmsupp \mu_v) \cup (\rmsupp \mu_{v'}) \to \R$ by letting $u_0(y_i) = 0$ and $u_0(y'_i) = \alpha \calG_1(v,v')$, $i=1,\ldots,Q+1$. It is most obvious that $\rmLip u_0 \leq 1$ and we let $\hat{u}_0$ be an extension of $u_0$ to $Y$ such that $\rmLip \hat{u}_0 \leq 1$, whose existence follows from the McShane-Whitney Theorem. Finally we let $u = \hat{u}_0 - \hat{u}_0(y_0) \ind_Y$ and we observe that
\begin{equation*}
\| \bzeta(v)-\bzeta(v')\|_{(\rmLip Y)^*} \geq \left| \sum_{i=1}^{Q+1} u(y_i) - \sum_{i=1}^{Q+1} u(y'_i) \right| = \alpha (Q+1) \calG_1(v,v') \,.
\end{equation*}
\par
{\em Second case}. We assume that
\begin{equation*}
\rmdist( \rmsupp \mu_v , \rmsupp \mu_{v'} ) \leq  \alpha \calG_1(v,v') \,.
\end{equation*}
Choose $i_0,j_0 \in \{1,\ldots,Q+1\}$ such that $d(y_{i_0},y'_{j_0}) = \rmdist(\rmsupp \mu_v , \rmsupp \mu_{v'} )$. Define $\tilde{v},\tilde{v}' \in \calQ_Q(Y)$ by
\begin{equation*}
\tilde{v} = \oplus_{i \neq i_0} \lseg y_i \rseg \text{ and } \tilde{v}' = \oplus_{j \neq j_0} \lseg y'_j \rseg \,.
\end{equation*}
According to the induction hypothesis there exists $u \in \rmLip_{y_0}(Y)$ with $\rmLip u \leq 1$ and
\begin{equation*}
\left| \sum_{i \neq i_0} u(y_i) - \sum_{j \neq j_0} u(y'_j) \right| \geq \frac{1}{2} \| \bzeta(\tilde{v}) - \bzeta(\tilde{v}') \|_{(\rmLip Y)^*} \geq \frac{c_Q}{2} \calG_1( \tilde{v} ,\tilde{v}') \,.
\end{equation*}
Since readily $\calG_1(\tilde{v},\tilde{v}') + d(y_{i_0},y'_{j_0}) \geq \calG_1(v,v')$ we infer that
\begin{equation*}
\begin{split}
\| \bzeta(v) - \bzeta(v') \|_{(\rmLip Y)^*} & \geq \left| \sum_{i=1}^{Q+1} u(y_i)- \sum_{j=1}^{Q+1} u(y'_j) \right| \\
& \geq \frac{c_Q}{2} \calG_1(\tilde{v},\tilde{v}') - | u(y_{i_0}) - u(y'_{j_0}) | \\
& \geq \frac{c_Q}{2} \calG_1(v,v') - \frac{c_Q}{2}d(y_{i_0},y'_{j_0}) - d(y_{i_0},y'_{j_0}) \\
& \geq  \left( \frac{c_Q}{2} - \alpha \left( 1 + \frac{c_Q}{2} \right) \right) \calG_1(v,v') \,.
\end{split}
\end{equation*}
We now choose $\alpha > 0$ small enough for $\frac{c_Q}{2} - \alpha \left( 1 + \frac{c_Q}{2} \right) > 0$ and we set
\begin{equation*}
c_{Q+1} = \min \left\{ \alpha (Q+1) , \frac{c_Q}{2} - \alpha \left( 1 + \frac{c_Q}{2} \right) \right\}
\end{equation*}
so that, in both cases,
\begin{equation*}
\| \bzeta(v)-\bzeta(v') \|_{(\rmLip Y)^*} \geq c_{Q+1} \calG_1(v,v') \,.
\end{equation*}
\end{proof}

\section{Sobolev classes}

\subsection{Definition of $L_p(X,\calQ_Q(Y))$}

Let $(Y,y_0)$ be a pointed metric space as usual, let $(X,\frA,\mu)$ be a measure space, and let $1 \leq p < \infty$. We denote by $L_p(X,\calQ_Q(Y))$ the collection of mappings $f : X \to \calQ_Q(Y)$ verifying the following requirements:
\begin{enumerate}
\item[(A)] $f$ is $(\frA,\frB_{\calQ_Q(Y)})$ measurable;
\item[(B)] The function $X \to \R : x \mapsto \calG_2(f(x),Q\lseg y_0\rseg)^p$ is $\mu$ summable.
\end{enumerate}
In the remaining part of this paper we shall abbreviate
\begin{equation*}
\lno f(x) \rno = \calG_2(f(x),Q\lseg y_0 \rseg) \,,
\end{equation*}
$x \in X$, and we keep in mind that no ambiguity should occur from the lack of mention of $y_0$ in the abbreviation\footnote{In case $Y$ is a Banach space it will be implicitly assumed that $y_0=0$}. If $f \in L_p(X,\calQ_Q(Y))$ we also set the notation
\begin{equation*}
\lno f \rno_{L_p} = \left( \int_X \lno f \rno^p d\mu \right)^{\frac{1}{p}} \,.
\end{equation*}
Of course $L_p(X,\calQ_Q(Y))$ need not be a linear space. It is most obvious that the formula
\begin{equation*}
d_p(f,g) = \left( \int_X \calG_2(f,g)^p d\mu \right)^{\frac{1}{p}}
\end{equation*}
defines a semimetric on $L_p(X,\calQ_Q(Y))$. As in the scalar case, we have:

\begin{Proposition}
\label{411}
Assume that $Y$ is a complete metric space. It follows that $L_p(X,\calQ_Q(Y))[d_p]$ is a complete semimetric space, and each Cauchy sequence contains a subsequence converging pointwise almost everywhere.
\end{Proposition}

\subsection{Analog of the Fr\'echet-Kolmogorov compactness Theorem}

\begin{Theorem}
\label{421}
Assume that $1 \leq p < \infty$ and that:
\begin{enumerate}
\item[(A)] $(X,\frB_X,\lambda)$ is a finite dimensional Banach space with a Haar measure $\lambda$ defined on the $\sigma$ algebra $\frB_X$ of Borel subsets of $X$;
\item[(B)] $Y$ is a compact metric space, and $y_0 \in Y$;
\item[(C)] $\calF \subset L_p(X,\calQ_Q(Y))$ is a family subjected to the following requirements:
\begin{enumerate}
\item[(i)] $\sup \{ \lno f \rno_{L_p} : f \in \calF \} < \infty$;
\item[(ii)] For every $\veps > 0$ there exists a neighbourhood $U$ of $0$ in $X$ such that
\begin{equation*}
\sup \{ d_p(\tau_h f , f) : f \in \calF \} < \veps
\end{equation*}
whenever $h \in U$, where $(\tau_h f)(x) := f(x+h)$;
\item[(iii)] For every $\veps > 0$ there exists a compact $K \subset X$ such that
\begin{equation*}
\sup \{ d_p(f,f_K) : f \in \calF \} < \veps \,,
\end{equation*}
where
\begin{equation*}
f_K(x) = \begin{cases} f(x) & \text{ if } x \in K \\
Q \lseg y_0 \rseg & \text{ if } x \not\in K \,.
\end{cases}
\end{equation*}
\end{enumerate}
\end{enumerate}
It follows that $\calF$ is relatively compact in $L_p(X,\calQ_Q(Y))[d_p]$.
\end{Theorem}

\begin{proof}
In this proof we will abbreviate $\|\cdot\| = \|\cdot\|_{(\rmLip Y)^*}$.
In view of the completeness of $L_p(X,\calQ_Q(Y))$ (Proposition \ref{411}) we need only to show that $\calF$ is totally bounded. Let $\veps > 0$ and choose $U$ and $K$ according to hypotheses (C)(ii) and (C)(iii). There is no restriction to assume that $\rmClos U$ is compact. We next secure a continuous function $\vphi : X \to \R^+$ such that $\rmsupp \vphi \subset U$ and $\int_X \vphi d \lambda = 1$. Given $f \in \calF$ we consider the map
\begin{equation*}
\bzeta \circ f_K : X \to \rmLip_{y_0}(Y)^*
\end{equation*}
and we observe that it is $(\frB_X,\frB_{\rmLip_{y_0}(Y)^*})$-measurable, separably valued (in fact $\rmim \bzeta \circ f_K \subset \rmim \bzeta$ and the latter is compact according to the continuity of $\bzeta$, Theorem \ref{341}, and the compactness of $\calQ_Q(Y)$, Proposition \ref{211}). It therefore ensues from the Pettis measurability Theorem, \cite[Chap. II \S 1 Theorem 2]{DIESTEL.UHL}, that $\bzeta \circ f_K$ is {\em strongly} measurable, i.e. the pointwise $\lambda$ almost everywhere limit of a sequence of $(\frB_X,\frB_{\rmLip_{y_0}(Y)^*})$-measurable functions with finite range. Furthermore $\bzeta \circ f_K$ is bounded (because $Y$ is) and compactly supported (because $\bzeta(Q \lseg y_0 \rseg)=0$), so that the Lebesgue integral $\int_X \|\bzeta \circ f_K \| d\lambda < \infty$. Thus $\bzeta \circ f_K$ is Bochner integrable. We define the convolution product of $\vphi$ and $\bzeta \circ f_K$ by means of the Bochner integral:
\begin{equation*}
(\vphi * \bzeta \circ f_K)(x) = (B) \int_X \vphi(h) (\bzeta \circ f_K)(x+h) d\lambda(h) \,,\qquad x \in X \,.
\end{equation*}
\par
We now claim that each $\vphi * (\bzeta \circ f_K)$ is continuous and, in fact, that the family $C(X,\rmLip_{y_0}(Y)^*) \cap \{ \vphi * (\bzeta \circ f_K) : f \in \calF \}$ is equicontinuous. Given $x,x' \in X$ we simply observe that
\begin{equation*}
\begin{split}
\| (\vphi * &\bzeta \circ f_K)(x) -  (\vphi * \bzeta \circ f_K)(x') \| \\
& = \left\| (B) \int_X (\vphi(h)-\vphi(h+x-x'))(\bzeta \circ f_K)(x+h)d\lambda(h) \right\| \\
& \leq \left( \int_X | \vphi(h) - \vphi(h+x-x')|^{\frac{p}{p-1}} d\lambda(h) \right)^{1 - \frac{1}{p}} \left( \int_X \|(\bzeta \circ f_K)(x+h)\|^pd\lambda(h) \right)^\frac{1}{p} \\
& \leq \rmosc(\vphi,\|x-x'\|_X) \lambda(U+B_X(0,\|x-x'\|_X))^{1-\frac{1}{p}} \lno f \rno_{L_p} \,,
\end{split}
\end{equation*}
according to \cite[Chap. II \S 2 Theorem 4(ii)]{DIESTEL.UHL}, H\"older's inequality, and Theorem \ref{341}. The equicontinuity follows from the uniform continuity of $\vphi$ and hypothesis (C)(i).
\par
We denote by $C$ the closed convex hull of $\rmim \bzeta$ in the Banach space $\rmLip_{y_0}(Y)^*$. As $\rmim \bzeta$ is compact it ensues from Mazur's Theorem that $C$ is compact as well. Furthermore, the definition of the convolution product guarantees that $(\vphi * \bzeta \circ f_K)(x) \in C$ for every $x \in X$. It therefore follows from Ascoli's Theorem, \cite[0.4.11]{EDWARDS}, that the family $C(X,\rmLip_{y_0}(Y)^*) \cap \{ \vphi * (\bzeta \circ f_K) : f \in \calF \}$ is relatively compact in $C_c(X,\rmLip_{y_0}(Y)^*)$ with respect to uniform convergence (note that $\rmsupp (\vphi * \bzeta \circ f_K) \subset K + \rmClos U$, a compact set
independent of $f$). Consequently there are $f^1,\ldots,f^{\kappa} \in \calF$ such that for every $f \in \calF$ there exists $k \in \{1,\ldots,\kappa\}$ with
\begin{equation}
\label{eq.12}
\| (\vphi * \bzeta \circ f_K)(x) - (\vphi * \bzeta \circ f^k_K)(x) \| < \veps \lambda(K + \rmClos U)^{- \frac{1}{p}}
\end{equation}
for every $x \in X$.
\par
Now given $f \in \calF$ we choose $k$ so that \eqref{eq.12} holds and we aim at showing that $d_p(f,f^k) < D \veps$ where $D$ is a suitable constant; this will complete the proof. We start with the observation that
\begin{equation*}
d_p(f,f^k) \leq d_p(f,f_K) + d_p(f_K,f^k_K) + d_p(f^k_K,f^k) \leq 2 \veps + d_p(f_K,f^k_K)
\end{equation*}
according to hypothesis (C)(iii). Next we infer from Theorem \ref{341} and \eqref{eq.12} that
\begin{equation*}
\begin{split}
c_Q d_p(f_K,f^k_K) & \leq \left( \int_X \| (\bzeta \circ f_K) - (\bzeta \circ f^k_K) \|^p d\lambda \right)^\frac{1}{p} \\
& \leq \left( \int_X \| (\bzeta \circ f_K) - (\vphi * \bzeta \circ f_K) \|^p d\lambda \right)^\frac{1}{p} \\
& \qquad + \left( \int_{K+\rmClos U} \| (\vphi * \bzeta \circ f_K) - (\vphi * \bzeta \circ f^k_K) \|^p d\lambda \right)^\frac{1}{p} \\
& \qquad + \left( \int_X \| (\vphi *\bzeta \circ f^k_K) - (\bzeta \circ f^k_K) \|^p d\lambda \right)^\frac{1}{p} \\
& \leq \left( \int_X \| (\bzeta \circ f_K) - (\vphi * \bzeta \circ f_K) \|^p d\lambda \right)^\frac{1}{p} + \veps\\
& \qquad + \left( \int_X \| (\vphi *\bzeta \circ f^k_K) - (\bzeta \circ f^k_K) \|^p d\lambda \right)^\frac{1}{p} \\
\end{split}
\end{equation*}
Thus it remains only to find a uniform small upper bound of
\begin{equation*}
\int_X \| (\bzeta \circ f_K) - (\vphi * \bzeta \circ f_K) \|^p d\lambda
\end{equation*}
whenever $f \in \calF$. Let $x \in X$, abbreviate $\mu = \lambda \hel \vphi$, and observe that
\begin{equation*}
\begin{split}
\|(\bzeta \circ f_K)(x)  &- (\vphi * \bzeta \circ f_K)(x) \| \\
& = \left\| (B) \int_X \vphi(h) \big( (\bzeta \circ f_K)(x) - (\bzeta \circ f_K)(x+h) \big) d\lambda(h) \right\|\\
& \leq \int_X \calG_1(f_K(x),\tau_hf_K(x)) d\mu(h) \,.
\end{split}
\end{equation*}
It then follows from Jensen's inequality applied to the probability measure $\mu$, and from Fubini's Theorem that
\begin{equation*}
\begin{split}
\int_X \| (\bzeta \circ f_K)(x) & - (\vphi * \bzeta \circ f_K)(x) \|^p d\lambda(x) \\
& \leq \int_X d\lambda(x) \left( \int_X \calG_1(f_K(x) , \tau_h f_K(x)) d\mu(h) \right)^p \\
& \leq \int_X d\lambda(x) \int_X \calG_1(f_K(x),\tau_h f_K(x))^p d\mu(h) \\
& = Q^{p/2}\int_U d\mu(h) d_p(f_K,\tau_h f_K)^p \\
& \leq Q^{p/2}\sup_{h \in U} d_p(f_K,\tau_h f_K)^p \,.
\end{split}
\end{equation*}
Consequently,
\begin{equation*}
\left( \int_X \| (\bzeta \circ f_K) - (\vphi * \bzeta \circ f_K) \|^p d\lambda \right)^\frac{1}{p} \leq 3 \sqrt{Q}\veps
\end{equation*}
according to hypotheses (C)(ii) and (iii). Therefore,
\begin{equation*}
d_p(f_K,f_K^k) \leq c_Q^{-1} (1 + 6\sqrt{Q}) \veps \,,
\end{equation*}
and finally,
\begin{equation*}
d_p(f,f^k) \leq \left( 2 + c_Q^{-1}(1+6\sqrt{Q}) \right) \veps \,.
\end{equation*}
\end{proof}

\subsection{Definition of $W^1_p(U;\calQ_Q(Y))$}

{\bf In this section $X$ is a finite dimensional Banach space with Haar measure $\lambda$, $U \subset X$ is either $X$ itself or a bounded open subset having the {\em extension property}\footnote{i.e. for every $1 < p < \infty$ there exists an extension operator $\bW^1_p(U) \to \bW^1_p(X)$ for classical Sobolev spaces; for instance $U$ has Lipschitz boundary}, $Y$ is a Banach space having the Radon-Nikod\' ym property, and $1 < p < \infty$.} The space $\rmHom(X,Y)$ is given a norm $\vvvert \cdot \vvvert$. We recall that each Lipschitz map $f : U \to \calQ_Q(Y)$ extends to a Lipschitz map $\hat{f} : X \to \calQ_Q(Y)$ according to Theorem \ref{243}, and that $\hat{f}$ is differentiable at $\lambda$ almost every $x \in U$, according to Theorem \ref{258}. For such $x$, writing $Df(x) = \oplus_{i=1}^Q \lseg L_i \rseg$, we recall that we have defined
\begin{equation*}
\lno Df(x) \rno = \sqrt{ \sum_{i=1}^Q \vvvert L_i \vvvert^2 } \,.
\end{equation*}
We define the {\em Sobolev class} $W^1_p(U;\calQ_Q(Y))$ to be the subset of $L_p(U;\calQ_Q(Y))$ consisting of those $f : U \to \calQ_Q(Y)$ for which there exists a sequence $\{f_j\}$ of Lipschitz mappings $X \to \calQ_Q(Y)$ with the following properties
\begin{enumerate}
\item[(1)] $f_j \in L_p(U;\calQ_Q(Y))$ and $\int_U \lno Df_j \rno^p d\lambda < \infty$ for every $j=1,2,\ldots$;
\item[(2)] $\sup_j \int_U \lno Df_j \rno^p d\lambda < \infty$;
\item[(3)] $d_p(f,f_j) \to 0$ as $j \to \infty$.
\end{enumerate}
In case $U$ is bounded, (1) is redundant.
\par
We define $\bW^1_p(U;\calQ_Q(Y))$ to be the quotient of $W^1_p(U;\calQ_Q(Y))$ relative to the equivalence relation $f_1 \sim f_2$ iff $\lambda \{f_1 \neq f_2 \} = 0$. We now recall the definition of F.J. Almgren's Sobolev class $Y_p(U;\calQ_Q(\ell_2^n))$. Here $X = \ell_2^m$ and $Y=\ell_2^n$. This is simply the collection of Borel functions $f : U \to \R^N$ (where $N=N(n,Q)$ is as in \ref{334}) such that $f $ is a member of the classical Sobolev space $W^1_p(U;\R^N)$, and $f(x) \in \bxi(\ell_2^n)$ for $\calL^m$ almost every $x \in U$. This is reminiscent of the definition of Sobolev mappings between Riemannian manifolds, except for $\calQ_Q(\ell_2^n)$ is not a Riemannian manifold, but merely a stratified space. We also let $\bY_p(U;\calQ_Q(\ell_2^n))$ denote the corresponding quotient relative to equality $\calL^m$ almost everywhere. We finally recall that $\rmHom(\ell_2^m,\ell_2^\nu)$ is equipped with the following norm
\begin{equation*}
\vvvert L \vvvert = \sqrt{ \sum_{j=1}^m \sum_{k=1}^\nu \la L(e_j) , e_k \ra^2 }
\end{equation*}
that appears in the following result.

\begin{Theorem}
\label{431}
Assuming that $X=\ell_2^m$ and $Y=\ell_2^n$, the mapping
\begin{equation*}
\Upsilon : W^1_p(U;\calQ_Q(\ell_2^n)) \to Y_p(U;\calQ_Q(\ell_2^n)) : f \mapsto \bxi \circ f
\end{equation*}
yields a bijection
\begin{equation*}
\bUpsilon : \bW^1_p(U;\calQ_Q(\ell_2^n)) \to \bY_p(U;\calQ_Q(\ell_2^n)) \,,
\end{equation*}
and
\begin{enumerate}
\item[(1)] $\int_U \calG_2(f(x),Q\lseg 0 \rseg)^p d\calL^m(x) = \int_U \| \Upsilon (f)(x) \|^p d\calL^m(x)$;
\item[(2)] If $f$ is Lipschitz then
\begin{equation*}
\int_U \lno Df(x) \rno^p d\calL^m(x) = \int_U \vvvert D \Upsilon(f)(x) \vvvert^p d\calL^m(x) \,.
\end{equation*}
\end{enumerate}
\end{Theorem}

\begin{proof}
We first show that $\bxi \circ f \in Y_p(U;\calQ_Q(\ell_2^n))$ whenever $f \in W^1_p(U;\calQ_Q(\ell_2^n))$. It is clear that $f : U \to \R^N$ is Borel measurable and also that
\begin{equation*}
\int_U \| (\bxi \circ f)(x) \|^p d\calL^m(x)
= \int_U \calG_2(f(x),Q\lseg 0 \rseg)^p d\calL^m(x) < \infty \,,
\end{equation*}
according to Theorem \ref{334}(C), thus $\bxi \circ f$ is a member of the classical Lebesgue space $L_p(U;\R^N)$ and conclusion (1) is proved. Assuming that $f$ be also Lipschitz then so is $\bxi \circ f$, thus conclusion (2) holds according to Proposition \ref{335} (in conjunction with Theorem \ref{258} and the classical Rademacher Theorem), whence $\bxi \circ f $ belongs to the classical Sobolev space $W^1_p(U;\R^N)$. If we now return to merely assuming that $f \in W^1_p(U;\calQ_Q(\ell_2^n))$ in our definition, there then exists a sequence $\{f_j\}$ of Lipschitz maps $X \to \calQ_Q(\ell_2^n)$ such that $\sup_j \int_U \lno Df_j \rno^p d\calL^m < \infty$ and $\lim_j \int_U \calG_2(f,f_j)^p d\calL^m = 0$. We infer from conclusions (1) and (2) that $\{ \bxi \circ f_j \}$ is a bounded sequence in $W^1_p(U;\R^N)$. Since $\bW^1_p(U;\R^N)$ is a reflexive Banach space, there exists a subsequence $\{ \bxi \circ f_{k(j)} \}$ converging weakly to some $g \in W^1_p(U;\R^N)$. Since $U$ has the extension property, the weak convergence corresponds to convergence in $L_p$ :
\begin{equation*}
\lim_j \int_U \| (\bxi \circ f_{k(j)}) - g \|^p d\calL^m = 0 \,,
\end{equation*}
and therefore $\bxi \circ f = g$ $\calL^m$ almost everywhere, which readily implies that $\bxi \circ f \in Y_p(U;\calQ_Q(\ell_2^n))$.
\par
We next observe that the equivalence class of $\Upsilon(f)$ depends only upon the equivalence class of $f$, because $\bxi$ maps null sets to null sets. Since the same is true about $\bxi^{-1}$, we infer that $\bUpsilon$ is injective. It remains to show that $\bUpsilon$ is surjective. Let $g \in Y_p(U;\calQ_Q(\ell_2^n))$. There is no restriction to assume that $g(x) \in \bxi(\ell_2^n)$ for {\em all} $x \in U$, and we define $f = \bxi^{-1} \circ g$; it is obviously Borel measurable. Since $g \in W^1_p(U;\R^N)$ and $U$ has the extension property, there exists $\hat{g} \in W^1_p(\R^n;\R^N)$ such that
$\hat{g}$ is compactly supported in a neighborhood of $U$ and $\hat{g} \restriction U = g$.  Choosing $\{ \vphi_{\veps_j} \}$ a smooth compactly supported approximation to the identity, we define
\begin{equation*}
f_j = \bxi^{-1} \circ \brho \circ ( \vphi_{\veps_j} * \hat{g} ) \,.
\end{equation*}
We observe that the $f_j : X \to \calQ_Q(\ell_2^n)$ are Lipschitz and
\begin{equation*}
\begin{split}
\int_U \calG_2(f_j,f)^p d\calL^m & = \int_U \calG_1 ( \bxi^{-1} \circ \brho \circ (\vphi_{\veps_j} * \hat{g}) , \bxi^{-1} \circ g )^p d\calL^m \\
& \leq \alpha(n,Q)^{-p} \int_U \| \brho \circ (\vphi_{\veps_j} * \hat{g}) - \brho \circ g \|^p d\calL^m \\
& \leq \alpha(n,Q)^{-p} (\rmLip \brho)^p \int_U \| \vphi_{\veps_j} * \hat{g} - g \|^p d\calL^m \\
& \to 0 \text{ as } j \to \infty \,.
\end{split}
\end{equation*}
\par
Finally, if $f_j$ and $\brho \circ (\vphi_{\veps_j} * \hat{g})$ are both differentiable at $a \in U$, then Proposition \ref{335} implies that
\begin{equation*}
\begin{split}
\lno Df_j(a) \rno^p & = \lno D( \bxi^{-1} \circ \brho \circ (\vphi_{\veps_j} * \hat{g}))(a) \rno^p \\
& = \vvvert D( \brho \circ (\vphi_{\veps_j} * \hat{g})) (a) \vvvert^p \\
& \leq (\rmLip \brho)^p \vvvert D( \vphi_{\veps_j} * \hat{g})(a) \vvvert^p \,.
\end{split}
\end{equation*}
Since this occurs that $\calL^m$ almost every $a \in U$, according to Theorem \ref{258} and the classical Rademacher Theorem, we infer that
\begin{multline*}
\sup_j \int_U \lno Df_j(a) \rno^p d\calL^m \leq (\rmLip \brho)^p \sup_j \int_U \vvvert D( \vphi_{\veps_j} * \hat{g}) \vvvert^p d\calL^m \\
\leq (\rmLip \brho)^p \int_U \vvvert D \hat{g} \vvvert^p d\calL^m \,.
\end{multline*}
Thus $f \in W^1_p(U;\calQ_Q(\ell_2^n))$.
\end{proof}

\begin{Remark}
It is worth observing that in case $p=1$ the above Theorem would not be valid, as our definition would yield a space of mappings $U \to \calQ_Q(\ell_2^n)$ of {\em bounded variation} rather than {\em Sobolev}.
\end{Remark}

\begin{Remark}
\label{433}
We recall that $U$ is assumed to have the extension property. This means that there exists a continuous linear operator
\begin{equation*}
\bE : \bW^1_p(U;\R^N) \to \bW^1_p(\R^m;\R^N) \,.
\end{equation*}
Given $\bbf \in \bW^1_p(\Rm;\R^N)$ and $f \in \bbf$, one easily checks that $\brho \circ f \in Y^1_p(\Rm;\calQ_Q(\ell_2^n))$ and that the equivalence class of $\brho \circ f$ depends only upon that of $f$. Thus the formula
\begin{equation*}
\tilde{\bE}(\bbf) = \bUpsilon^{-1} \left( \brho \circ \bE \left( \bUpsilon \left( \bbf \right) \right) \right)
\end{equation*}
defines an ``extension mapping''
\begin{equation*}
\bW^1_p(U;\calQ_Q(\ell_2^n)) \to \bW^1_p(\Rm; \calQ_Q(\ell_2^n)) \,.
\end{equation*}
\end{Remark}

\begin{Proposition}
\label{434}
Let $f \in W^1_p(U;\calQ_Q(\ell_2^n))$ and $t \geq 0$. Define
\begin{equation*}
A_t = U \cap \big\{ x : \calG_2(f(x),Q\lseg 0 \rseg)^p + \left( M \vvvert D E(\Upsilon(f)) \vvvert \right)^p(x) \leq t^p \big\} \,,
\end{equation*}
where $M$ denotes the maximal function operator 
and $E(\Upsilon(f))$ is a representant of the class $\bE(\Upsilon(f))$.
There then exists a Lipschitzian map $h : U \to \calQ_Q(\ell_2^n)$ such that
\begin{enumerate}
\item[(1)] $h(x) = f(x)$ for $\calL^m$ almost every $x \in A_t$;
\item[(2)] $\rmLip h \leq 4^{m+1} \alpha(n,Q)^{-1} \bc_{\ref{243}}(m,Q) t$ where $\alpha(n,Q)$ is as in Theorem \ref{334};
\item[(3)] $\calG_2(h(x),Q\lseg 0 \rseg) \leq \bc_{\ref{243}}(m,Q) t$ for every $x \in U$;
\item[(4)] For $\calL^m$ almost every $x \in A_t$, $f$ is approximately differentiable at $x$ and $\lno Df(x) \rno = \lno Dh(x) \rno$.
\end{enumerate}
\end{Proposition}

\begin{proof}
Write $u = E(\Upsilon(f)) \in W^1_p(\Rm;\R^N)$.
We let $\tilde{A}_t$ denote the Borel subset of $A_t$ consisting of those $x$ such that $u(x) = \lim_{\veps \to 0^+} (\vphi_\veps * u)(x)$ where $\{\vphi_\veps\}_{\veps > 0}$ is a given approximate identity. Given distinct $x,x' \in \tilde{A}_t$ we let $\Omega = \bU(x,2r) \cap \bU(x',2r)$, where $r = \|x-x'\| > 0$, and we infer from \cite[Lemma 1.50]{MALY.ZIEMER} (adapted in the obvious way to the case of vectorvalued maps) that
\begin{equation*}
\left\| u(x) - \dashint_\Omega u d\calL^m \right\| \leq \frac{4^m}{m\balpha(m)} \int_\Omega \frac{\vvvert Du(y) \vvvert}{\|x-y\|^{m-1}} \, d\calL^m(y)
\end{equation*}
and
\begin{equation*}
\left\| u(x') - \dashint_\Omega u d\calL^m \right\| \leq \frac{4^m}{m\balpha(m)} \int_\Omega \frac{\vvvert Du(y) \vvvert}{\|x'-y\|^{m-1}} \, d\calL^m(y)
\end{equation*}
It follows from the potential estimate \cite[Theorem 1.32(i)]{MALY.ZIEMER} that
\begin{equation*}
\begin{split}
\int_\Omega \frac{\vvvert Du(y) \vvvert}{\|x-y\|^{m-1}} d\calL^m(y) & \leq  \int_{\bU(x,2r)} \frac{\vvvert Du(y) \vvvert}{\|x-y\|^{m-1}} d\calL^m(y)  \\
& \leq  m\balpha(m) 2 \|x-x'\| M \left( \vvvert Du \vvvert \right)(x)  \,.
\end{split}
\end{equation*}
Since the same holds with $x$ replaced by $x'$, we obtain
\begin{equation*}
\|u(x) - u(x') \| \leq 4^{m+1} t \| x - x'\|
\end{equation*}
whenever $x,x' \in \tilde{A}_t$. The first three conclusions now follow from Theorems \ref{334} and \ref{243}.
Conclusion (4) follows from the fact that $h$ and $f$ are approximately tangent at each Lebesgue density point of $A_t$, together with the differentiability Theorem \ref{258}.
\end{proof}

\begin{Remark}
\label{435}
We shall see in Proposition \ref{453} that the constant in (2) does not in fact depend upon $n$.
\end{Remark}

\begin{Corollary}
\label{436}
Let $f \in W^1_p(U;\calQ_Q(\ell_2^n))$. It follows that $f$ and $\Upsilon(f)$ are approximately differentiable $\calL^m$ almost everywhere, and that
\begin{equation*}
\lno Df(x) \rno = \vvvert D\Upsilon(f)(x) \vvvert
\end{equation*}
 at each point $x \in U$ where both are approximately differentiable.
\end{Corollary}

\begin{proof}
That $\Upsilon(f)$ be approximately differentiable (in the usual sense) $\calL^m$ almost everywhere follows from standard Sobolev theory (see e.g. \cite[Theorem 1.72]{MALY.ZIEMER}). The analogous property of $f$ follows from Proposition \ref{434}(4) and the arbitrariness of $t \geq 0$. The last conclusion is a consequence of Proposition \ref{335}.
\end{proof}

\subsection{The $p$ energy}

In this section, $X$, $Y$, $U$ and $p$ are subject to the same requirements as in the last section, and sometimes more. Given $f \in W^1_p(U;\calQ_Q(Y))$ and an open subset $V \subset U$, we define the {\em $p$ energy of $f$ in $V$} by the formula
\begin{multline*}
\calE_p^p(f;V) = \inf \bigg\{ \liminf_j \int_V \lno Df_j \rno^p d\lambda : \{f_j\} \text{ is a sequence } \\ \text{of Lipschitz mappings } U \to \calQ_Q(Y)
\text{ such that } d_p(f,f_j) \to 0 \text{ as } j \to \infty \bigg\} \,.
\end{multline*}
We notice that $\calE_p^p(f;V) \leq \calE_p^p(f;U) < \infty$. Clearly,

\begin{Proposition}
\label{441}
Given $f \in W^1_p(U;\calQ_Q(Y))$ and an open subset $V \subset U$, there exists a sequence of Lipschitz mappings $U \to \calQ_Q(Y)$ such that $\lim_j d_p(f,f_j)=0$ and
\begin{equation*}
\calE_p^p(f;V) = \lim_j \int_V \lno Df_j \rno^p d\lambda \,.
\end{equation*}
\end{Proposition}

As the $p$-energy is defined by relaxation, we easily prove its lower semicontinuity with respect to weak convergence.

\begin{Proposition}
\label{442}
Let $f,f_1,f_2,\ldots$ be members of $W^1_p(U;\calQ_Q(Y))$ and assume that $d_p(f,f_j) \to 0$ as $j \to \infty$. It follows that
\begin{equation*}
\calE_p^p(f;V) \leq \liminf_j \calE_p^p(f_j;V)
\end{equation*}
for every open subset $V \subset U$.
\end{Proposition}

If $W \subset Y$ is a linar subspace and $P : Y \to W$ is a continuous linear retract we define
\begin{equation*}
\calQ_Q(P) : \calQ_Q(Y) \to \calQ_Q(W)
\end{equation*}
by the formula $\calQ_Q(P)(\lseg y_1,\ldots,y_Q \rseg) = \lseg P(y_1),\ldots,P(y_Q) \rseg$. It is a trivial matter to check that
\begin{equation*}
\calG_2(\calQ_Q(P)(v),\calQ_Q(P)(v')) \leq (\rmLip P) \calQ_2(v,v')
\end{equation*}
whenever $v,v' \in \calQ_Q(Y)$.

\begin{Proposition}
\label{443}
Assume that
\begin{enumerate}
\item[(1)] $W \subset Y$ is a linear subspace and $P : Y \to W$ is a continuous linear retract;
\item[(2)] $g : U \to \calQ_Q(Y)$ is approximately differentiable at $a \in U$;
\item[(3)] $\rmHom(X,Y)$ and $\rmHom(X,W)$ are equipped with norms such that $\vvvert P \circ L \vvvert \leq (\rmLip P) \vvvert L \vvvert$ whenever $L \in \rmHom(X,Y)$.
\end{enumerate}
It follows that $\calQ_Q(P) \circ g$ is approximately differentiable at $a$ and
\begin{equation*}
\lno D(\calQ_Q(P) \circ g)(a) \rno \leq (\rmLip P) \lno Dg(a) \rno \,.
\end{equation*}
\end{Proposition}

\begin{proof}
Write $Ag(a) = \oplus_{i=1}^Q \lseg A_i \rseg$, with $A_i : X \to Y$ affine maps. Observe that
\begin{multline*}
\mathrm{ap}\lim_{x \to a} \frac{ \calG_2 \left( (\calQ_Q(P) \circ g)(x) , \left( \calQ_Q(P) \circ \oplus_{i=1}^Q \lseg A_i \rseg \right) (x)\right)}{\|x-a\|} \\ \leq (\rmLip P) \mathrm{ap}\lim_{x \to a} \frac{\calG_2 \left( g(x) , \left( \oplus_{i=1}^Q \lseg A_i \rseg \right) (x)\right)}{\|x-a\|} = 0 \,.
\end{multline*}
Since $\calQ_Q(P) \circ \oplus_{i=1}^Q \lseg A_i \rseg = \oplus_{i=1}^Q \lseg P \circ A_i\rseg$, and the $P \circ A_i$ are affine as well, we infer that $\calQ_Q(P) \circ g$ is differentiable at $a$ and $A(\calQ_Q(P) \circ g)(a) = \oplus_{i=1}^Q \lseg P \circ A_i \rseg$. Next note that if $L_i$ is the linear part of $A_i$, then $P \circ L_i$ is the linear
 part of $P \circ A_i$. Consequently,
\begin{equation*}
\lno D(\calQ_Q(P) \circ g)(a) \rno^2 = \sum_{i=1}^Q \vvvert P \circ L_i \vvvert^2 \leq (\rmLip P)^2 \sum_{i=1}^Q \vvvert L_i \vvvert^2 = (\rmLip P)^2 \lno Dg(a) \rno^2 \,.
\end{equation*}
\end{proof}

\begin{Proposition}
\label{444}
Assume that
\begin{enumerate}
\item[(1)] $W \subset Y$ is a linear subspace and $\iota : W \to Y$ is the canonical injection;
\item[(2)] $g : U \to \calQ_Q(W)$ is differentiable at $a \in U$;
\item[(3)] $\rmHom(X,Y)$ and $\rmHom(X,W)$ are equipped with norms such that $\vvvert \iota \circ L \vvvert = \vvvert L \vvvert$ whenever $L \in \rmHom(X,W)$.
\end{enumerate}
It follows that $\calQ_Q(\iota) \circ g$ is differentiable at $a$ and
\begin{equation*}
\lno D( \calQ_Q(\iota) \circ g)(a) \rno = \lno Dg(a) \rno \,.
\end{equation*}
\end{Proposition}

\begin{proof}
The proof is similar to that of Proposition \ref{443}.
\end{proof}

\begin{Remark}
\label{445}
Hypotheses (3) of Proposition \ref{443} and \ref{444} are verified in two cases of interest. First when $\vvvert \cdot \vvvert$ is the operator norm. Second when
\begin{equation*}
\vvvert L \vvvert = \nu \left( \sum_{j=1}^m \|L(u_j)\|_Y e_j \right)
\end{equation*}
where $\nu$ is a norm on $\Rm$, $m=\rmdim X$, $e_1,\ldots,e_m$ is the canonical basis of $\Rm$, and $u_1,\ldots,u_m$ is a basis of $X$.
\end{Remark}

\begin{Proposition}
\label{446}
Assume that
\begin{enumerate}
\item[(1)] $W \subset Y$ is a linear subspace, $P : Y \to W$ is a continuous linear retraction, and $\iota : W \to Y$ is the canonical injection;
\item[(2)] $g \in W^1_p(U;\calQ_Q(W))$;
\item[(3)] $\rmHom(X,Y)$ and $\rmHom(X,W)$ are equipped with norms such that
$\vvvert P \circ L \vvvert \leq (\rmLip P) \vvvert L \vvvert$ whenever $L \in \rmHom(X,Y)$, and $\vvvert \iota \circ L \vvvert = \vvvert L \vvvert$ whenever $L \in \rmHom(X,W)$.
\end{enumerate}
It follows that $\calQ_Q(\iota) \circ g \in W^1_p(U;\calQ_Q(Y)$ and
\begin{equation*}
(\rmLip P)^{-p} \calE_p^p(g;V) \leq \calE_p^p(\calQ_Q(\iota) \circ g;V) \leq \calE_p^p(g;V) \,,
\end{equation*}
for every $V \subset U$ open.
\end{Proposition}

\begin{proof}
Choose a sequence $\{g_j\}$ of Lipschitz mappings $U \to \calQ_Q(W)$ such that $d_p(g,g_j) \to 0$ and $\calE_p^p(g;V) = \lim_j \int_V \lno Dg_j \rno^p d\lambda$, according to Proposition \ref{441}. Notice that $\calQ_Q(\iota) \circ g$ are Lipschitz mappings $U \to \calQ_Q(Y)$ and that
\begin{equation*}
\limsup_j d_p(\calQ_Q(\iota) \circ g , \calQ_Q(\iota) \circ g_j) \leq \lim_j d_p(g,g_j) = 0 \,.
\end{equation*}
Therefore,
\begin{equation*}
\begin{split}
\calE_p^p(\calQ_Q(\iota) \circ g;V) & \leq \liminf_j \int_V \lno D( \calQ_Q(\iota) \circ g_j) \rno^p d\lambda \\
& \leq \limsup_j \int_V \lno Dg_j \rno^p d\lambda \\
& = \calE_p^p(g;V) \,,
\end{split}
\end{equation*}
according to Proposition \ref{444}.
The case $V=U$ of this computation implies that $\calQ_Q(\iota) \circ g \in W^1_p(U;\calQ_Q(Y))$, by definition of this Sobolev class, and the general case yields the second inequality of our conclusion.
\par
The other way round choose a sequence $\{f_j\}$ of Lipschitz mappings $U \to \calQ_Q(Y)$ such that $d_p(\calQ_Q(\iota) \circ g,f_j) \to 0$ and $\calE_p^p(\calQ_Q(\iota) \circ g;V) = \lim_j \int_V \lno Df_j \rno^p d\lambda$. Notice that the mappings $\calQ_Q(P) \circ f_j : U \to \calQ_Q(W)$ are Lipschitz and, since $g = \calQ_Q(P) \circ \calQ_Q(\iota) \circ g$, one has
\begin{multline*}
d_p( g , \calQ_Q(P) \circ f_j) = d_p( \calQ_Q(P) \circ \calQ_Q(\iota) \circ g , \calQ_Q(P) \circ f_j) \\
\leq d_p (\calQ_Q(\iota) \circ g , f_j) \to 0 \text{ as } j \to \infty \,.
\end{multline*}
Thus,
\begin{equation*}
\begin{split}
\calE_p^p(g;V) & \leq \liminf_j \int_V \lno D(\calQ_Q(P) \circ f_j) \rno^p d\lambda \\
& \leq (\rmLip P)^p \liminf_j \int_V \lno D f_j \rno^p d\lambda \\
& = (\rmLip P)^p \calE_p^p(\calQ_Q(\iota) \circ g;V) \,. 
\end{split}
\end{equation*}
\end{proof}

{\bf For the remaining part of this paper we will only consider the cases when either $Y = \ell_2^n$ for some $n \in \N \setminus \{0\}$ or $Y = \ell_2$, and $X$ is a finite dimensional Banach space as usual. The norm $\vvvert \cdot \vvvert$ on $\rmHom(X,Y)$ is associated with a basis $u_1,\ldots,u_m$ of $X$ as follows:
\begin{equation*}
\vvvert L \vvvert = \sqrt{ \sum_{j=1}^m \|L(u_j)\|^2}
\end{equation*}
where $\|\cdot\|$ is the Hilbert norm on $Y$.} According to Remark \ref{445}, Propositions \ref{443} and \ref{444} apply. When $Y = \ell_2$ and $n \in \N \setminus \{0\}$ we also define an $n$ dimensional subspace of $Y$, $W_n = \rmspan \{e_1,\ldots,e_n \}$, and we let $P_n : Y \to W_n$ be the orthognal projection and $\iota_n : W_n \to Y$ be the canonical injection.
\par
The following guarantees that the $p$ energy is the expected quantity in case $Y = \ell_2^n$. Notice the statement makes sense since $g$ is almost everywhere approximately differentiable (recall Corollary \ref{436}).
\begin{Proposition}
\label{447}
If $g \in W^1_p(U;\calQ_Q(\ell_2^n))$ for some $n \in \N \setminus \{0\}$ then
\begin{equation*}
\calE_p^p(g;V) = \int_V \lno Dg \rno^p d\lambda
\end{equation*}
for every open set $V \subset U$.
\end{Proposition}

\begin{proof}
If $\{g_j\}$ is a sequence of Lipschitz mappings $U \to \calQ_Q(\ell_2^n)$ such that $d_p(g,g_j) \to 0$ as $j \to \infty$, then $\|\bxi \circ g - \bxi \circ g_n \|_{L_p} \to 0$ as $j \to \infty$ where $\bxi$ is the Almgren embedding described in Theorem \ref{334}. Thus
\begin{equation*}
\int_V \vvvert D( \bxi \circ g) \vvvert^p d\lambda \leq \liminf_j \int_V \vvvert D(\bxi \circ g_j) \vvvert^p d\lambda
\end{equation*}
according to classical finite dimensional Sobolev theory: the above functional is weakly lower semicontinuous because it satisfies the hypotheses of \cite[Section 3.3, Theorem 3.4]{DACOROGNA}. It then follows from Corollary \ref{436} that
\begin{equation*}
\int_V \lno Dg \rno^p d\lambda \leq \liminf_j \int_V \lno Dg_j \rno^p d\lambda \,.
\end{equation*}
Choosing the sequence $\{g_j\}$ according to Proposition \ref{441} we infer that
\begin{equation*}
\int_V \lno Dg \rno^p d\lambda \leq \calE_p^p(g;V) \,.
\end{equation*}
\par
We turn to proving the reverse inequality.
We let $u = E(\Upsilon(g)) \in W^1_p(\R^m;\R^N)$ so that the maximal function $M (\vvvert Du \vvvert ) \in L_p(U)$ (see e.g. \cite[Theorem 1.22]{MALY.ZIEMER}). For each $j \in \N \setminus \{0\}$ we define
\begin{equation*}
A_j = U \cap \{ x : \calG_2(g(x),Q\lseg 0 \rseg)^p + M( \vvvert Du \vvvert)^p(x) \leq j^p \}
\end{equation*}
and we infer that
\begin{equation*}
\lim_j j^p \lambda( U \setminus A_j) \leq \lim_j \int_{U \setminus A_j} \left( \calG_2(g(x),Q\lseg 0 \rseg)^p + M( \vvvert Du \vvvert)^p(x) \right) d\lambda(x) = 0 \,.
\end{equation*}
We let $g_j : U \to \calQ_Q(\ell_2^n)$ be the Lipschitz mapping associated with $f=g$ and $t=j$ in Proposition \ref{434}. We see that
\begin{equation*}
\begin{split}
\lim_j d_p(g_j,g) & \lim_j \left( \int_{U \setminus A_j} \calG_2(g_j,g)^p d\lambda(x) \right)^\frac{1}{p} \\
& \leq \lim_j \left( \int_{U \setminus A_j} \calG_2(g_j,Q \lseg 0 \rseg)^p d\lambda(x) \right)^\frac{1}{p} \\
& \qquad\qquad\qquad + \lim_j \left( \int_{U \setminus A_j} \calG_2(g,Q \lseg 0 \rseg)^p d\lambda(x) \right)^\frac{1}{p} \\
& \leq \lim_j \left( \bc_{\ref{243}}(m,Q) j^p \lambda(U \setminus A_j) \right)^\frac{1}{p}
 + \lim_j \left( \int_{U \setminus A_j} \calG_2(g,Q \lseg 0 \rseg)^p d\lambda(x) \right)^\frac{1}{p} \\
& = 0 \,,
\end{split}
\end{equation*}
thus $\calE^p_p(g;V) \leq \liminf_j \int_V \lno Dg_j \rno^p d\lambda$.
Furthermore,
\begin{equation*}
\begin{split}
\liminf_j \int_V \lno Dg_j \rno^p d\lambda & \leq  \liminf_j \int_{V \cap A_j} \lno Dg_j \rno^p d\lambda + \limsup_j \int_{U \setminus A_j} \lno Dg_j \rno^p d\lambda \\
& \leq \liminf_j \int_{V \cap A_j} \lno Dg \rno^p d\lambda + Q^\frac{p}{2} \limsup_j \int_{U \setminus A_j} \left( \rmLip g_j \right)^p d\lambda \\
& \leq \int_V \lno Dg \rno^p d\lambda \\
& \qquad + \limsup_j Q^\frac{p}{2} 4^{p(m+1)} \alpha(n,Q)^{-p} \bc_{\ref{243}}(m,Q)^p j^p \lambda(U \setminus A_j) \\
& = \int_V \lno Dg \rno^p d\lambda \,.
\end{split}
\end{equation*}
This completes the proof.
\end{proof}

\begin{Theorem}
\label{448}
Let $f \in W^1_p(U;\calQ_Q(\ell_2))$. The following hold.
\begin{enumerate}
\item[(A)] $\calQ_Q(P_n) \circ f \in W^1_p(U;\calQ_Q(\ell_2^n))$ for each $n \in \N \setminus \{0\}$;
\item[(B)] For every open set $V \subset U$ one has
\begin{equation*}
\calE_p^p(f;V) = \lim_n \int_V \lno D (\calQ_Q(P_n) \circ f) \rno^p d\lambda \,;
\end{equation*}
\item[(C)] The sequence $\{ \lno D(\calQ_Q(P_n) \circ f) \rno^p \}_n$ is nondecreasing $\lambda$ almost everywhere and bounded in $L_1(U)$.
\end{enumerate}
\end{Theorem}

\begin{proof}
(A) Choose a sequence $\{f_j\}$ of Lipschitz mappings $U \to \calQ_Q(\ell_2)$ such that $d_p(f_j,f) \to 0$ and $\sup_j \int_U \lno Df_j \rno^p d\lambda < \infty$. Notice that the $\calQ_Q(P_n) \circ f : U \to \calQ_Q(\ell_2^n)$ are Lipschitz,
\begin{equation*}
\lim_j d_p( \calQ_Q(P_n) \circ f_j , \calQ_Q(P_n) \circ f) \leq \lim_j d_p(f_j,f) = 0
\end{equation*}
and
\begin{equation*}
\sup_j \int_U \lno D( \calQ_Q(P_n) \circ f_j) \rno^p d\lambda \leq \sup_j \int_U \lno Df_j \rno^p d\lambda < \infty
\end{equation*}
according to Proposition \ref{443}. Thus $f \in W^1_p(U;\calQ_Q(\ell_2^n))$.
\par
(B) We note that for every $x \in U$ one has
\begin{equation*}
\lim_n \calG_2 \big( f (x), \left( \calQ_Q(\iota_n) \circ \calQ_Q(P_n) \circ f \right)(x) \big) = 0 \,,
\end{equation*}
and also
\begin{equation*}
\begin{split}
\calG_2(f , \calQ_Q(\iota_n) \circ \calQ_Q(P_n) \circ f) & \leq \calG_2( f , Q \lseg 0 \rseg) + \calG_2 ( \calQ_Q(\iota_n) \circ \calQ_Q(P_n) \circ f , Q \lseg 0 \rseg) \\
& \leq 2 \calG_2 ( f , Q \lseg 0 \rseg) \,.
 \end{split}
\end{equation*}
Thus
\begin{equation*}
\lim_n d_p ( f , \calQ_Q(\iota_n) \circ \calQ_Q(P_n) \circ f) = 0
\end{equation*}
according to the Dominated Convergence Theorem. Therefore,
\begin{equation*}
\begin{split}
\calE^p_p(f;V) & \leq \liminf_n \calE_p^p( \calQ_Q(\iota_n) \circ \calQ_Q(P_n) \circ f ; V) \\
& = \liminf_n \calE_p^p( \calQ_Q(P_n) \circ f ; V) \\
& = \liminf_n \int_V \lno D( \calQ_Q(P_n) \circ f) \rno^p d\lambda \,,
\end{split}
\end{equation*}
according respectively to Propositions \ref{442}, \ref{446} and \ref{447}. The other way round, we choose a sequence $\{f_j\}$ of Lipschitz mappings $U \to \calQ_Q(\ell_2)$ such that $d_p(f_j,f) \to 0$ and $\calE_p^p(f;V) = \lim_j \int_V \lno Df_j \rno^p d\lambda$, according to Proposition \ref{441}. For each fixed $n$ we have
\begin{equation*}
\begin{split}
\int_V \lno D( \calQ_Q(P_n) \circ f) \rno^p d\lambda & \leq \liminf_j \int_V \lno D(\calQ_Q(P_n) \circ f_j) \rno^p d\lambda \\
\intertext{(according to the proof of (A) and Proposition \ref{442} and \ref{447})}
& \leq \liminf_j \int_V \lno D f_j \rno^p d\lambda \\
\intertext{(according to Proposition \ref{443})}
& = \calE_p^p(f;V) \,.
\end{split}
\end{equation*}
Therefore $\limsup_n \int_V \lno D(\calQ_Q(P_n) \circ f) \rno^p d\lambda \leq \calE_p^p(f;V)$.
\par
(C) That the sequence $\{ \lno D(\calQ_Q(P_n) \circ f) \rno^p \}_n$ be nondecreasing follows as in the proof of Proposition \ref{443}; its boundedness in $L_1(U)$ is a consequence of (B).
\end{proof}

We now turn to defining the function $\lno \delta f \rno \in L_p(U)$ associated with $f \in W^1_p(U;\calQ_Q(\ell_2))$. It follows from Theorem \ref{448}(C) and the Monotone Convergence Theorem that
\begin{equation*}
\int_V \lim_n \lno D( \calQ_Q(P_n) \circ f) \rno^p d\lambda = \lim_n \int_V \lno D(\calQ_Q(P_n) \circ f) \rno^p d\lambda \,,
\end{equation*}
$V \subset U$ open. We define, for $\lambda$ almost every $x \in U$,
\begin{equation}
\label{eq.204}
\lno \delta f \rno(x) = \lim_n \lno D( \calQ_Q(P_n) \circ f)(x) \rno \,.
\end{equation}
It follows therefore from Theorem \ref{448}(B) that
\begin{equation}
\label{eq.205}
\calE_p^p(f;V) = \int_V \lno \delta f \rno^p d\lambda \,.
\end{equation}

\subsection{Extension}

The following is the obvious analog of \cite[Theorem 1.63]{MALY.ZIEMER}.

\begin{Theorem}
\label{ext}
Let $U = U(0,1)$ be the unit ball in $\Rm$. There exists a mapping
\begin{equation*}
E : W^1_p(U;\calQ_Q(\ell_2)) \to W^1_p(\Rm;\calQ_Q(\ell_2))
\end{equation*}
with the following properties.
\begin{enumerate}
\item[(A)] For every $f \in W^1_p(U;\calQ_Q(\ell_2))$ one has $E(f)(x) = f(x)$ for every $x \in U$;
\item[(B)] For every $f_1,f_2 \in W^1_p(U;\calQ_Q(\ell_2)$ one has
\begin{equation*}
d_p(E(f_1),E(f_2)) \leq 2^\frac{1}{p} d_p(f_1,f_2) \,;
\end{equation*}
\item[(C)] For every $f \in W^1_p(U;\calQ_Q(\ell_2))$ one has
\begin{equation*}
\calE_p^p(E(f);\Rm) \leq \left( 1 +Q^\frac{p}{2} 2^p \right) \left( \calE_p^p(f;U) + \lno f \rno_p^p \right) \,;
\end{equation*}
\item[(D)] For every $x \in \Rm \setminus U(0,2)$ one has $E(f)(x) = Q \lseg 0 \rseg$;
\item[(E)] If $0 \in C \subset \ell_2$ is convex and $f(x) \in \calQ_Q(C)$ for every $x \in U$, then $E(f)(x) \in \calQ_Q(C)$ for every $x \in \Rm$.
\end{enumerate}

\end{Theorem}

\begin{proof}
We start the proof by associating with each {\em Lipschitz} map $f : U \to \calQ_Q(\ell_2)$ a {\em Lipschitz} map $E_0(f) : \Rm \to \calQ_Q(\ell_2)$ verifying (A), (B), (D) and (E) above (for Lipschitz maps $f$, $f_1$, $f_2$) and (C) replaced with
\begin{enumerate}
\item[(C')] For every Lipschitz $f  : U \to \calQ_Q(\ell_2)$ one has
\begin{equation*}
\int_{\Rm} \lno D E_0(f)(x) \rno^p d\calL^m(x) \leq C(m,p) \left( \int_U \lno Df(x) \rno^p d\calL^m(x) + \lno f \rno_p^p \right) \,.
\end{equation*}
\end{enumerate}
Given $f$ we write $f(x) = \oplus_{i=1}^Q \lseg f_i(x) \rseg$, $x \in U$, and we define
\begin{equation*}
g(x) = \begin{cases}
\oplus_{i=1}^Q \lseg (2\|x\|-1)f_i(x) \rseg & \text{ if } \|x\| \geq \frac{1}{2} \\
Q \lseg 0 \rseg & \text{ if } \|x\| < \frac{1}{2} \,.
\end{cases}
\end{equation*}
The conscientious reader will check that $g$ is Lipschitz on $U$. In fact, it follows from Proposition \ref{259} and the paragraph preceding it (in particular equation \eqref{eq.204}) that
\begin{equation*}
\lno Dg(x) \rno \leq \sqrt{Q} \bigg( \lno Df(x) \rno + 2 \lno f(x) \rno \bigg)
\end{equation*}
for almost every $x \in U(0,1) \setminus B(0,1/2)$. Therefore,
\begin{equation}
\label{eq.206}
\int_U \lno Dg \rno^p d\calL^m \leq Q^\frac{p}{2} 2^p \left( \int_U \lno Df \rno^p d\calL^m + \int_U \lno f \rno^p d\calL^m \right) \,.
\end{equation}
\par
We now define a Lipschitz mapping $\vphi : U(0,3/2) \setminus U(0,1) \to B(0,1) \setminus B(0,1/2)$ by the formula
\begin{equation*}
\vphi(x) = \left( \frac{2}{\|x\|} - 1 \right) x \,,
\end{equation*}
and $E_0(f)$ by
\begin{equation*}
E_0(f)(x) = \begin{cases}
Q \lseg 0 \rseg & \text{if } x \geq \frac{3}{2} \\
g(\vphi(x)) & \text{if } 1 \leq \|x\| < \frac{3}{2} \\
f(x) & \text{if } \|x\| \leq 1 \,.
\end{cases}
\end{equation*}
We notice that conclusions (A), (D) and (E) are verified by $E_0(f)$. Regarding conclusions (B) and (C') we first observe that the differential of $x/ \|x \|$ at a point $x \neq 0$ is the orthogonal projection onto the plane orthogonal to $x$. Therefore $J\vphi = 1$ and we apply the change of variable formula:
\begin{equation*}
\begin{split}
d_p(E_0(f_1),& E_0(f_2))^p  \leq \int_U \calG(f_1,f_2)^p d\calL^m + \int_{B(0,3/2)\setminus B(0,1)} \calG(g_1 \circ \vphi , g_2 \circ \vphi)^p d\calL^m \\
& \leq \int_U \calG(f_1,f_2)^p d\calL^m + \int_{B(0,3/2)\setminus B(0,1)} \calG(g_1 \circ \vphi , g_2 \circ \vphi)^p J \vphi d\calL^m \\
& \leq 2\int_U \calG(f_1,f_2)^p d\calL^m
\end{split}
\end{equation*}
(because $\calG(g_1,g_2) \leq \calG(f_1,f_2)$), and similarly,
\begin{equation*}
\begin{split}
\int_{\Rm} \lno DE_0(f) \rno^p & d\calL^m  \leq \int_U \lno Df \rno^p d\calL^m + \int_{B(0,3/2)\setminus B(0,1)} \lno D(g \circ \vphi) \rno^p d\calL^m \\
& \leq \int_U \lno Df \rno^p d\calL^m + \int_{B(0,3/2)\setminus B(0,1)} \left( \lno D g \rno^p \circ \vphi \right) d\calL^m \\
\intertext{(because $\rmLip \vphi \leq 1$)}
& \leq \int_U \lno Df \rno^p d\calL^m + \int_{B(0,3/2)\setminus B(0,1)} \left( \lno D g \rno^p \circ \vphi \right) J\vphi d\calL^m \\
& \leq \left( 1 + Q^\frac{p}{2} 2^p \right) \left( \int_U \lno Df \rno ^p d\calL^m + \int_U \lno f \rno^p d\calL^m \right)
\end{split}
\end{equation*}
according to \eqref{eq.206}.
\par
We now define $E(f)$, $f \in W^1_p(U;\calQ_Q(\ell_2))$, as follows. We {\em choose} a sequence $\{ f_j\}$ of Lipschitz mappings $U \to \calQ_Q(\ell_2)$ associated with $f$ as in Proposition \ref{441} and we observe that $\{ E_0(f_j) \}$ is Cauchy in $L_p(\Rm)$ : for $E(f)$ we choose a limit of this sequence (that verifies conclusion (A)). That conclusions (B), (C), (D) and (E) are valid is now a matter of routine verification.
\end{proof}

\subsection{Poincar\'e inequality and approximate differentiability almost everywhere}

We start with a modification of Theorem \ref{448}.

\begin{Theorem}
\label{451}
Let $f \in W^1_p(U;\calQ_Q(\ell_2))$. There then exist a sequence $\{f_n\}$ of Lipschitz mappings $U \to \calQ_Q(\ell_2)$ and a sequence $\{A_n\}$ of Borel subsets of $U$ such that
\begin{enumerate}
\item[(A)] $\lim_n d_p(f_n,f) = 0$;
\item[(B)] For every open set $V \subset U$,
\begin{equation*}
\calE_p^p(f;V) = \lim_n \int_V \lno Df_n \rno^p d\lambda \,.
\end{equation*}
\item[(C)] $\lim_n \calL^m(U \setminus A_n) = 0$ and, for each $n$, $\lno Df_n(x) \rno \leq \lno \delta f \rno(x)$ for $\calL^m$ almost every $x \in A_n$;
\item[(D)] $\lim_n \lno Df_n(x) \rno = \lno \delta f \rno (x)$ for $\calL^m$ almost every $x \in U$.
\end{enumerate}
\end{Theorem}

\begin{proof}
With each $n \in \N \setminus \{0\}$ we associate $g_n = \calQ_Q(P_n) \circ f \in W^1_p(U;\calQ_Q(\ell_2^n))$ as well as
\begin{equation*}
u_n = \calG_2(g_n,Q \lseg 0 \rseg)^p + M( \vvvert D\Upsilon(g_n) \vvvert)^p \in L_1(U) \,.
\end{equation*}
Letting $A_n = U \cap \{ x : u_n(x) \leq t_n^p \}$ we can choose $t_n > 0$ large enough for
\begin{equation}
\label{eq.201}
\max \left\{ \calL^m(U \setminus A_n) , \bc_{\ref{434}}(n,m,Q)^p \int_{U \setminus A_n} u_n d\lambda \right\}< \frac{1}{n} \,,
\end{equation}
where we have put $\bc_{\ref{434}}(n,m,Q) = 4^{m+1} \alpha(n,Q)^{-1} \bc_{\ref{243}}(m,Q)$.
We then let $h_n : U \to \calQ_Q(\ell_2^n)$ be a Lipschitz mapping associated with $g_n$ and $t_n$ as in Proposition \ref{434}, and we define $f_n = \calQ_Q(\iota_n) \circ h_n : U \to \calQ_Q(\ell_2)$ which is Lipschitz as well. We observe that
\begin{equation*}
\begin{split}
d_p(f_n,f) & \leq d_p( \calQ_Q(\iota_n) \circ h_n , \calQ_Q(\iota_n) \circ \calQ_Q(P_n) \circ f) + d_p( \calQ_Q(\iota_n) \circ \calQ_Q(P_n) \circ f , f) \\
& \leq d_p(h_n,g_n) + d_p( \calQ_Q(\iota_n) \circ \calQ_Q(P_n) \circ f , f) \,.
\end{split}
\end{equation*}
Notice that
\begin{equation*}
\lim_n d_p(\calQ_Q(\iota_n) \circ \calQ_Q(P_n) \circ f , f) = 0
\end{equation*}
according to the Dominated Convergence Theorem, whereas
\begin{equation*}
\begin{split}
d_p(h_n,g_n) & = \left( \int_U \calG_2(h_n,g_n)^p d\lambda \right)^\frac{1}{p} \\
& =\left( \int_{U \setminus A_n} \calG_2(h_n,g_n)^p d\lambda \right)^\frac{1}{p} \\
& \leq \left( \int_{U \setminus A_n} \calG_2(h_n,Q \lseg 0 \rseg)^p d\lambda \right)^\frac{1}{p}
+ \left( \int_{U \setminus A_n} \calG_2(g_n,Q \lseg 0 \rseg)^p d\lambda \right)^\frac{1}{p} \\
& \leq \left( \int_{U \setminus A_n} \bc_{\ref{243}}(m,Q)^p t_n^p d\lambda \right)^\frac{1}{p}
+ \left( \int_{U \setminus A_n} \calG_2(g_n,Q \lseg 0 \rseg)^p d\lambda \right)^\frac{1}{p}\\
& \leq (1+ \bc_{\ref{243}}(m,Q)) \left( \int_{U \setminus A_n} u_n d\lambda \right)^\frac{1}{p} \\ 
& \to 0 \text{ as } n \to \infty \,,
\end{split}
\end{equation*}
from what conclusion (A) follows. Consequently,
\begin{equation*}
\calE_p^p(f;V) \leq \liminf_n \int_V \lno D f_n \rno^p d\lambda \,.
\end{equation*}
Furthermore, for each $n$ we have
\begin{equation*}
\begin{split}
\int_V \lno Df_n \rno^p d\lambda & =  \int_V \lno D h_n \rno^p d\lambda\\ 
\intertext{(according to Proposition \ref{446})}
& \leq \int_{V \cap A_n} \lno Dg_n \rno^p d\lambda + \int_{V \setminus A_n} Q^\frac{p}{2}\bc_{\ref{434}}(n,m,Q)^p t_n^p d\lambda  \\
\intertext{(according to Proposition \ref{434})}
& \leq \int_V \lno D(\calQ_Q(P_n) \circ f) \rno^p d\lambda + Q^\frac{p}{2} \bc_{\ref{434}}(n,m,Q)^p \int_{U \setminus A_n} u_n d\lambda \,.
\end{split}
\end{equation*}
It now follows from \eqref{eq.201} and Theorem \ref{448}(B) that
\begin{multline*}
\limsup_n \int_V \lno Df_n \rno^p d\lambda \\ \leq \lim_n \int_V \lno D(\calQ_Q(P_n) \circ f) \rno^p d\lambda + \limsup_n Q^\frac{p}{2} \bc_{\ref{434}}(n,m,Q)^p \int_{U \setminus A_n} u_n d\lambda \\ = \calE_p^p(f;V) \,. 
\end{multline*}
This proves conclusion (B).
\par
The first part of conclusion (C) is a consequence of \eqref{eq.201} and the second part follows from the fact that $h_n = g_n$ on $A_n$, therefore $D h_n(x) = D g_n(x)$ at $\calL^m$ almost every $x \in A_n$, and for those $x$ it follows from Proposition \ref{444}, the definition of $\lno \delta f \rno$ and Theorem \ref{448}(C) that
\begin{multline*}
\lno Df_n(x) \rno = \lno D(\calQ_Q(\iota_n) \circ h_n) (x) \rno \leq \lno D h_n(x) \rno \\ = \lno D g_n (x) \rno = \lno D(\calQ_Q(P_n) \circ f)(x) \rno \leq \lno \delta f \rno(x) \,.
\end{multline*}
\par
Conclusion (D) is an easy consequence of (B) and (C).
\end{proof}

We are now ready to prove the analog of the Poincar\'e inequality.

\begin{Theorem}
\label{452}
There exists a constant $\bc_{\theTheorem}(m) \geq 1$ with the following property.
Let $f \in W^1_p(U;\calQ_Q(\ell_2))$, $1 \leq q \leq p$, and let $V \subset U$ be a bounded open {\em convex} subset of $U$. It follows that for $\calL^m$ almost every $x \in V$,
\begin{equation*}
\int_V \calG_2(f(x),f(y))^q d\calL^m(y) \leq (\rmdiam V)^{q+m-1} \int_V \frac{\lno \delta f \rno^q(z)}{\|z-x\|^{m-1}}d\calL^m(z) \,.
\end{equation*}
Furthermore there exists $v \in \calQ_Q(\ell_2)$ such that
\begin{equation*}
\int_V \calG_2(f(x),v)^q d\calL^m(x) \leq \bc_{\theTheorem}(m) (\rmdiam V)^q \left( \frac{(\rmdiam V)^m}{\calL^m(V)} \right)^{1-\frac{1}{m}} \int_V \lno \delta f \rno^q d\calL^m \,.
\end{equation*}
\end{Theorem}

\begin{proof}
We start with the case when $f$ is Lipschitz. Given $x \in V$ it follows from Theorem \ref{258}(D)  that
\begin{equation*}
\calG_2(f(x),f(y)) \leq \int_{S_{x,y}} \lno Df(z) \rno d\calH^1(z) = \|x-y\| \int_0^1 \lno Df(x+t(y-x)) \rno d\calL^1(t) 
\end{equation*}
for $\calL^m$ almost every $y \in V$, where $S_{x,y}$ denotes the line segment joining $x$ and $y$. Now, given $s > 0$, we observe that
\begin{equation*}
\begin{split}
\int_{V \cap \rmBdry B(x,s)} \calG_2( & f(x),f(y))^q d\calH^{m-1}(y) \\
& \leq s^q \int_{V \cap \rmBdry B(x,s)} d\calH^{m-1}(y) \int_0^1 \lno Df(x+t(y-x)) \rno^q d\calL^1(t)\\
& = s^q \int_0^1 d\calL^1(t) \int_{V \cap \rmBdry B(x,s)} \lno Df(x+t(y-x)) \rno^q d\calH^{m-1}(y) \\
& = s^q \int_0^1 t^{1-m} d\calL^1(t) \int_{V \cap \rmBdry B(x,ts)} \lno Df(z) \rno^q d\calH^{m-1}(z) \\
& \leq s^{q+m-1} \int_0^1 d\calL^1(t) \int_{V \cap \rmBdry B(x,ts)} \frac{\lno Df(z) \rno^q}{\|z-x\|^{m-1}} d\calH^{m-1}(z) \\
& = s^{q+m-2} \int_{V \cap B(x,s)} \frac{\lno Df(z) \rno^q}{\|z-x\|^{m-1}} d\calL^m(z) \,.
\end{split}
\end{equation*}
Hence,
\begin{equation}
\label{eq.202}
\begin{split}
\int_V \calG_2(f(x) & ,f(y))^q d\calL^m(y) \\
& = \int_0^{\rmdiam V} d\calL^1(s) \int_{V \cap \rmBdry B(x,s)} \calG_2(  f(x),f(y))^q d\calH^{m-1}(y) \\
& \leq \int_0^{\rmdiam V} s^{q+m-2} d\calL^1(s) \int_{V \cap B(x,s)} \frac{\lno Df(z) \rno^q}{\|z-x\|^{m-1}} d\calL^m(z) \\
& \leq (\rmdiam V)^{q+m-1} \int_V \frac{\lno Df(z) \rno^q}{\|z-x\|^{m-1}} d\calL^m(z)  \,.
\end{split}
\end{equation}
We now merely assume that $f \in W^1_p(U;\calQ_Q(\ell_2))$ and we choose a sequence of Lipschitz mappings $\{f_n\}$ as in Theorem \ref{451}. Thus \eqref{eq.202} applies to each $f_n$. Let $x \in V$ be such that $\lim_n \calG_2(f(x),f_n(x))=0$. In order to establish our first conclusion we can readily assume that
\begin{equation*}
V \to \R : z \mapsto \frac{\lno \delta f \rno^q(x)}{\|z-x\|^{m-1}}
\end{equation*}
is summable. In that case, it follows from Theorem \ref{451}(A) and (D), from \eqref{eq.202} and from the Dominated Convergence Theorem that
\begin{equation*}
\begin{split}
\int_V \calG_2(f(x),f(y))^q d\calL^m(y) & = \lim_n \int_V \calG_2(f_n(x),f_n(y))^q d\calL^m(y) \\
& \leq (\rmdiam V)^{q+m-1} \lim_n \int_V \frac{ \lno Df_n(z) \rno^q}{\|z-x\|^{m-1}} d\calL^m(z) \\
& = (\rmdiam V)^{q+m-1} \int_V \frac{\lno \delta f \rno^q(z)}{\|z-x\|^{m-1}} d\calL^m(z) \,.
\end{split}
\end{equation*}
\par
We now turn to proving the second conclusion. Integrating the inequality above with respect to $x$, and applying standard potential estimates (see e.g. \cite[Lemma 1.31]{MALY.ZIEMER} applied with $p=1$) we obtain 
\begin{multline*}
\int_V d\calL^m(x) \int_V \calG_2(f(x),f(y))^q d\calL^m(y) \\ \leq (\rmdiam V)^{q+m-1} \int_V d\calL^m(x) \int_V \frac{\lno \delta f \rno^q(z)}{\|z-x\|^{m-1}} d\calL^m(z)\\
\leq C(m) (\rmdiam V)^{q+m-1} \calL^m(V)^\frac{1}{m}\int_V \lno \delta f \rno^q d\calL^m \,.
\end{multline*}
Thus there exists $x \in V$ such that
\begin{equation*}
\int_V \calG_2(f(x),f(y))^q d\calL^m(y) \leq C(m) \frac{(\rmdiam V)^{q+m-1}\calL^m(V)^\frac{1}{m}}{\calL^m(V)} \int_V \lno \delta f \rno^q d\calL^m \,.
\end{equation*}
Letting $v=f(x)$ completes the proof.
\end{proof}

\begin{Proposition}
\label{453}
Let $U = U(0,1)$ be the unit ball in $\Rm$, let $f \in W^1_p(U;\calQ_Q(\ell_2))$ and $t \geq 0$. Define
\begin{equation*}
A_t = U \cap \big\{ x : \calG_2(f(x),Q\lseg 0 \rseg)^p + \left( M \lno \delta E(f) \rno \right)^p(x) \leq t^p \big\} \,,
\end{equation*}
where $M$ denotes the maximal function operator and $E$ denotes the extension operator defined in Theorem \ref{ext}.
There then exists a Lipschitzian map $h : U \to \calQ_Q(\ell_2)$ such that
\begin{enumerate}
\item[(1)] $h(x) = f(x)$ for $\calL^m$ almost every $x \in A_t$;
\item[(2)] $\rmLip h \leq 6 m \balpha(m) \bc_{\ref{234}}(m,Q) t$; 
\item[(3)] $\calG_2(h(x),Q\lseg 0 \rseg) \leq \bc_{\ref{243}}(m,Q) t$ for every $x \in U$;
\item[(4)] For $\calL^m$ almost every $x \in A_t$, $f$ is approximately differentiable at $x$ and $\lno Df(x) \rno = \lno Dh(x) \rno$.
\end{enumerate}
\end{Proposition}

\begin{proof}
The proof is similar to that of Proposition \ref{434}.
We abbreviate $\hat{f} = E(f)$. 
We choose a countable dense set $D \subset \Rm$ and we consider the collection $\calV$ of subsets $V$ of $\Rm$ of the type $V = U(x,r) \cap U(x',r)$ where $x,x' \in D$ and $r \in \Q^+$. Thus $\calV$ is countable and for each $V \in \calV$ there exists $N_V \subset V$ such that $\calL^m(V \setminus N_V) = 0$ and for every $x \in V \setminus N_V$ one has
\begin{equation}
\label{eq.203}
\int_V \calG(\hat{f}(x),\hat{f}(y))d\calL^m(y) \leq (\rmdiam V)^m \int_V \frac{\lno \delta \hat{f} \rno(z)}{\|z-x\|^{m-1}} d\calL^m(z) \,,
\end{equation}
according to Theorem \ref{452} applied with $q=1$. Let $N = \cup_{V \in \calV} N_V$. Given $x,x' \in \Rm \setminus N$ we choose $r \in \Q^+$ such that
\begin{equation*}
0 < r - \|x-x'\| < \frac{\|x-x'\|}{5}
\end{equation*}
and we choose $\tilde{x}, \tilde{x}' \in D$ such that
\begin{equation*}
\max \big\{ \|x-\tilde{x}\| , \|x'- \tilde{x}' \| \big\} < \frac{r}{5} \,.
\end{equation*}
Defining $V = U(\tilde{x},2r) \cap U(\tilde{x}',2r) \in \calV$ we easily infer that $x,x' \in V$. Therefore \eqref{eq.203} applies to both pairs $x, V$ and $x',V$. We define
\begin{equation*}
G = V \cap \left\{ y : \calG(\hat{f}(x),\hat{f}(y)) < \frac{3 (\rmdiam V)^m)}{\calL^m(V)} \int_V \frac{\lno \delta \hat{f} \rno(z)}{\|z-x\|^{m-1}} d\calL^m(z) \right\} \,,
\end{equation*}
as well as
\begin{equation*}
G' = V \cap \left\{ y : \calG(\hat{f}(x'),\hat{f}(y)) < \frac{3 (\rmdiam V)^m)}{\calL^m(V)} \int_V \frac{\lno \delta \hat{f} \rno(z)}{\|z-x'\|^{m-1}} d\calL^m(z) \right\} \,.
\end{equation*}
One readily infer from \eqref{eq.203} that
\begin{equation*}
\max \big\{ \calL^m(V \setminus G) , \calL^m(V \setminus G') \big\}< \frac{\calL^m(V)}{3} \,,
\end{equation*}
and hence $G \cap G' \neq \emptyset$. We choose $y \in G \cap G'$ and we set $v = \hat{f}(y)$. Thus
\begin{equation*}
\calG(\hat{f}(x),v) \leq \frac{3 (2r)^m}{\balpha(m)r^m} \int_V \frac{ \lno \delta \hat{f} \rno(z)}{\|z-x\|^{m-1}} d\calL^m(z)
\end{equation*}
and
\begin{equation*}
\calG(\hat{f}(x'),v) \leq \frac{3 (2r)^m}{\balpha(m)r^m} \int_V \frac{ \lno \delta \hat{f} \rno(z)}{\|z-x'\|^{m-1}} d\calL^m(z) \,.
\end{equation*}
It follows from the potential estimate \cite[Lemma 1.32(i)]{MALY.ZIEMER} that
\begin{multline*}
\int_V \frac{ \lno \delta \hat{f} \rno(z)}{\|z-x\|^{m-1}} d\calL^m(z) \leq \int_{U(x,2r)}\frac{ \lno \delta \hat{f} \rno(z)}{\|z-x\|^{m-1}} d\calL^m(z) \\ \leq m \balpha(m) (2r) M\left( \lno \delta \hat{f} \rno\right)(x) \leq 3 m \balpha(m) \|x-x'\| M\left( \lno \delta \hat{f} \rno\right)(x) \,,
\end{multline*}
and similarly
\begin{equation*}
\int_V \frac{ \lno \delta \hat{f} \rno(z)}{\|z-x'\|^{m-1}} d\calL^m(z) \leq 3 m \balpha(m) \|x-x'\| M\left( \lno \delta \hat{f} \rno\right)(x') \,.
\end{equation*}
If furthermore $x,x' \in A_t$ then $\max \{ M( \lno \hat{f} \rno)(x), M( \lno \hat{f} \rno) (x') \} \leq t$ and it ensues from the above inequalities that
\begin{equation*}
\calG(\hat{f}(x) , \hat{f}(x')) \leq 6 m \balpha(m) \|x-x'\| t \,. 
\end{equation*}
One now concludes like in Proposition \ref{434}.
\end{proof}

The following is the analog of Proposition \ref{447} for an infinite dimensional target.

\begin{Corollary}
\label{454}
Let $U = U(0,1)$ be the unit ball in $\Rm$ and let $f \in W^1_p(U;\calQ_Q(\ell_2))$. It follows that $f$ is approximately differentiable $\calL^m$ almost everywhere and that
\begin{equation*}
\calE_p^p(f;V) = \int_V \lno Df(x) \rno^p d\calL^m(x)
\end{equation*}
for every open set $V \subset U$.
\end{Corollary}

\begin{proof}
Letting $\{t_j\}$ be an increasing unbounded sequence in $\R^+$ we observe that $\calL^m(U \setminus A_{t_j}) \to 0$ as $j \to \infty$ (where $A_{t_j}$ is defined as in the statement of Proposition \ref{453}) because both $\calG(\hat{f}(\cdot),Q\lseg 0 \rseg))$ and $M \left( \lno \delta \hat{f} \rno \right)$ belong to $L_p(\Rm)$. Letting $h_j$ be a Lipschitz mapping $U \to \calQ_Q(\ell_2)$ which coincides with $f$ almost everywhere on $A_{t_j}$, we easily infer that $f$ is approximately differentiable at each Lebesgue point $x \in A_{t_j}$ of $A_{t_j}$ at which $h_j$ is approximately differentiable. Since this is the case of $\calL^m$ almost every $a \in A_{t_j}$ according to Theorem \ref{258}, our first conclusion follows.
\par
In order to prove our second conclusion, consider a point $x \in U$ of approximate differentiability of $f$. Reasoning as in the proof of Proposition \ref{443} we write $Af(x) = \oplus_{i=1}^Q \lseg A_i \rseg$ and we infer that for each integer $n$, $A(\calQ_Q(P_n) \circ f)(a) = \oplus_{i=1}^Q \lseg P_n \circ A_i \rseg$. Since the linear part of $P_n \circ A_i$ is $P_n \circ L_i$, where $L_i$ is the linear part of $A_i$, we see that
\begin{equation*}
\lno D(\calQ_Q(P_n) \circ f)(x) \rno^2 = \sum_{i=1}^Q \vvvert P_n \circ L_i \vvvert^2 = \sum_{i=1}^Q \sum_{j=1}^m \| P_n(L_i(e_j)) \|^2 \,.
\end{equation*}
Thus
\begin{equation*}
\begin{split}
\lim_n \lno D(\calQ_Q(P_n) \circ f)(x) \rno^2 & = \lim_n \sum_{i=1}^Q \sum_{j=1}^m \| P_n(L_i(e_j)) \|^2 \\
& = \sum_{i=1}^Q \sum_{j=1}^m \|L_i(e_j)) \|^2 \\
& = \lno Df(x) \rno^2 \,.
\end{split}
\end{equation*}
Therefore $\lno Df(x) \rno = \lno \delta f \rno(x)$, according to \eqref{eq.204}, and the conclusion follows from \eqref{eq.205}.
\end{proof}

\subsection{Trace}

\begin{Proposition}
\label{471}
Let $U = U(0,1)$ be the unit ball in $\Rm$. For every $\veps > 0$ there exists $\theta > 0$ such that
\begin{equation*}
\int_{\rmBdry U} |u|^p d\calH^{m-1} \leq \theta \int_U |u|^p d\calL^m + \veps \int_U \| \nabla u \|^p d\calL^m
\end{equation*}
whenever $u : \rmClos U \to \R$ is Lipschitz.
\end{Proposition}

\begin{proof}
Given $\hat{\veps} > 0$ we choose a smooth function $\vphi : [0,1] \to [0,1]$ such that $\vphi(0) = \vphi(1)=1$ and
\begin{equation*}
\hat{\veps} = \left( \int_0^1 \vphi^q \right)^\frac{p}{q}
\end{equation*}
where $q$ is the exponent conjugate to $p$, and we put
\begin{equation*}
\hat{\theta} = \left( \int_0^1 |\vphi'|^q \right)^\frac{p}{q} \,.
\end{equation*}
For every $x \in \rmBdry U$ and $y \in U$ we observe that
\begin{equation*}
\begin{split}
|u(x)-u(y)| & = \left| \int_0^1 \frac{d}{dt} \bigg( \vphi(t) u(y + t(x-y)) \bigg) d\calL^1(t) \right| \\
& = \left| \int_0^1 \bigg( \vphi'(t) u(y + t(x-y)) + \vphi(t) \la \nabla u(y + t(x-y)) , x-y \ra \bigg) d\calL^1(t) \right| \\
& \leq \left( \int_0^1 | \vphi'|^q \right)^\frac{1}{q} \left( \int_0^1 |u(y+t(x-y))|^p d\calL^1(t) \right)^\frac{1}{p} \\
& \qquad\qquad + \left( \int_0^1 \vphi^q \right)^\frac{1}{q}  \left( \int_0^1 \|\nabla u (y+t(x-y)) \|^p \|x-y\|^p d\calL^1(t) \right)^\frac{1}{p} \,,
\end{split}
\end{equation*}
Therefore,
\begin{multline}
\label{eq.207}
|u(x)-u(y)|^p \leq 2^{p-1} \hat{\theta} \int_0^1 |u(y+t(x-y))|^p d\calL^1(t) \\+ 2^{p-1} \|x-y\|^p\hat{\veps} \int_0^1 \| \nabla u(y+t(x-y)) \|^p d\calL^1(t) \,.
\end{multline}
In order to integrate with respect to $x \in \rmBdry U$, we first note that the jacobian of the map $[0,1] \times \rmBdry U \to U : (t,x) \mapsto y + t(x-y)$ 
at $(t,x)$ equals $\|x-y\|t^{m-1}$. Since $\|x-y\| \leq 2$, the area formula therefore implies that
\begin{equation*}
\begin{split}
2^{2-m} & \int_{\rmBdry U}  d\calH^{m-1}(x)  \int_0^1 |u(y+t(x-y))|^p d\calL^1(t)\\ &  \leq \int_0^1 d\calL^1(t)  \int_{\rmBdry U} \frac{|u(y+t(x-y))|^p}{\|x-y\|^{m-2}} d\calH^{m-1}(x) \\
& = \int_0^1 d\calL^1(t)  \int_{\rmBdry U} \frac{|u(y + t(x-y))|^p}{\|y - (y + t(x-y))\|^{m-1}} \|x-y\|t^{m-1}d\calH^{m-1}(x) \\
& = \int_U \frac{|u(z)|^p}{\|y-z\|^{m-1}} d\calL^m(z) \,.
\end{split}
\end{equation*}
Since the similar inequality holds for the gradient term, we infer from \eqref{eq.207} that
\begin{multline*}
\int_{\rmBdry U} | u(x) -u(y) |^p d\calH^{m-1}(x) \leq 2^{p+m-3} \hat{\theta} \int_U \frac{|u(z)|^p}{\|y-z\|^{m-1}} d\calL^m(z) \\
+ 2^{2p+m-3} \hat{\veps} \int_U \frac{\|\nabla u(z)\|^p}{\|y-z\|^{m-1}} d\calL^m(z) \,.
\end{multline*}
Thus,
\begin{multline*}
\int_{\rmBdry U} |u(x)|^p d\calH^{m-1}(x) \leq 2^{p-1} m \balpha(m) |u(y)|^p \\ + 2^{2p+m-4} \hat{\theta} \int_U \frac{|u(z)|^p}{\|y-z\|^{m-1}} d\calL^m(z)
+ 2^{3p+m-4} \hat{\veps} \int_U \frac{\|\nabla u(z)\|^p}{\|y-z\|^{m-1}} d\calL^m(z)   \,,
\end{multline*}
according to the triangle inequality. We now integrate with respect to $y \in U$ and, referring to the potential estimate \cite[Lemma 1.31]{MALY.ZIEMER}, we obtain
\begin{equation*}
\begin{split}
\int_{\rmBdry U} |u|^p d\calH^{m-1} & \leq 2^{p-1} m \int_U |u(y)|^p d\calL^m(y) \\
& \qquad + 2^{2p+m-4} \balpha(m)^{-1} \hat{\theta} \int_U d\calL^m(y) \int_U \frac{|u(z)|^p}{\|y-z\|^{m-1}} d\calL^m(z) \\
& \qquad + 2^{3p+m-4} \balpha(m)^{-1} \hat{\veps} \int_U d\calL^m(y) \int_U \frac{\|\nabla u(z)\|^p}{\|y-z\|^{m-1}} d\calL^m(z) \\
& \leq m(2^{p-1}+2^{2p+m-4}\hat{\theta})  \int_U |u|^p d\calL^m \\ 
& \qquad + m 2^{3p+m-4} \hat{\veps} \int_U \|\nabla u \|^p d\calL^m \,.
\end{split}
\end{equation*}
\end{proof}

\begin{Remark}
It follows in particular from Proposition \ref{471} that
\begin{equation*}
\int_{\rmBdry U} |u|^p d\calH^{m-1} \leq C \left( \int_U |u|^p d\calL^m + \int_U \|\nabla u\|^p d\calL^m \right)
\end{equation*}
for some $C > 0$. Thus there exists a unique continuous {\em trace operator}
\begin{equation*}
\calT : \bW^1_p(U) \to \bL_p(\rmBdry U;\calH^{m-1})
\end{equation*}
defined by $\calT(u)=u$ whenever $u$ is Lipschitz. Of course, being continuous, $\calT$ is also weakly continuous. The inequality in Proposition \ref{471} shows that $\calT$ is completely continuous, i.e. if $\{u_k\}$ converges weakly in $\bW^1_p(U)$ then $\{\calT(u_k)\}$ converges strongly in $\bL_p(\rmBdry U;\calH^{m-1})$. Using Proposition \ref{471} (more precisely, an $\R^N$ valued version) in conjunction with the embedding Theorem \ref{334} we obtain that for every $\veps > 0$ there exists $\theta_n > 0$ such that
\begin{multline*}
\int_{\rmBdry U} \calG(f_1,f_2)^p d\calH^{m-1} \leq \theta_n \int_U \calG(f_1,f_2)^p d\calL^m \\+ \veps \left( \int_U \lno Df_1 \rno^p d\calL^m + \lno Df_2 \rno^p d\calL^m \right)
\end{multline*}
whenever $f_1,f_2 : U \to \calQ_Q(\ell_2^n)$ are Lipschitz. The dependence of $\theta$ upon $n$ is caused by a constant $\alpha(n,Q)^{-1}$ (the biLipschitz 
constant of the Almgren embedding). This leads to a proper definition of a trace ``operator'' for maps $f \in \bW^1_p(U;\calQ_Q(\ell_2^n))$ but not for maps $f \in \bW^1_p(U;\calQ_Q(\ell_2))$. We use a different approach in our next result, avoiding altogether the embedding of Theorem \ref{334}.
\end{Remark}

\begin{Theorem}
\label{472}
There exists a map
\begin{equation*}
\calT : W^1_p(U;\calQ_Q(\ell_2)) \to L_p(\rmBdry U;\calQ_Q(\ell_2))
\end{equation*}
verifying the following properties.
\begin{enumerate}
\item[(A)] If $f : \rmClos U \to \calQ_Q(\ell_2)$ is Lipschitz then $\calT(f)(x) = f(x)$ for every $x \in \rmBdry U$;
\item[(B)] For every $\veps > 0$ there exists $\theta > 0$ such that
\begin{multline*}
\int_{\rmBdry U} \calG(\calT(f_1),\calT(f_2))^p d\calH^{m-1} \leq \theta \int_U \calG(f_1,f_2)^p d\calL^m \\+ \veps \left( \int_U \lno Df_1 \rno^p d\calL^m + \lno Df_2 \rno^p d\calL^m \right)
\end{multline*}
whenever $f_1,f_2 \in W^1_p(U;\calQ_Q(\ell_2)$;
\item[(C)] There exists $C > 0$ such that for every $f \in W^1_p(U;\calQ_Q(\ell_2))$,
\begin{equation*}
\int_{\rmBdry U} \lno \calT(f) \rno^p d\calH^{m-1} \leq C \left( \int_U \lno f \rno^p d\calL^m + \int_U \lno Df \rno^p d\calL^m \right) \,.
\end{equation*}
\end{enumerate}
\end{Theorem}

\begin{proof}
Owing to definition of $W^1_p(U;\calQ_Q(\ell_2))$ (the weak density of Lipschitz maps), and to Propositions \ref{411}, \ref{441} and \ref{454}, it suffices to show that the map $\calT$ defined for Lipschitz $f$ by (A), verifies conclusions (B) and (C) for Lipschitz $f_1,f_2,f$.
\par
Given $f_1,f_2 : \rmClos U \to \calQ_Q(\ell_2)$ we define $u : \rmClos U \to \R$ by the formula $u(x) = \calG(f_1(x),f_2(x))$, $x \in U$. Given $x \in U$ and $h \in \Rm$ such that $x+h \in U$ we infer from the triangle inequality that
\begin{equation*}
\begin{split}
|u(x+h)-u(x)| & \leq | \calG(f_1(x+h),f_2(x+h)) - \calG(f_1(x),f_2(x+h))| \\
&\qquad+ | \calG(f_1(x),f_2(x+h)) - \calG(f_1(x),f_2(x))| \\
& \leq \calG(f_1(x+h),f_1(x)) + \calG(f_2(x+h),f_2(x)) \,.
\end{split}
\end{equation*}
This shows at once that $u$ is Lipschitz. Furthermore Proposition \ref{259} implies that
\begin{equation*}
\|\nabla u(x)\| \leq \lno Df_1(x) \rno + \lno Df_2(x) \rno
\end{equation*}
at each $x \in U$ where $u$, $f_1$ and $f_2$ are differentiable. Conclusion (B) now follows from Proposition \ref{471}, and conclusion (C) is a consequence of (B) with $f_1 = f$, $f_2 = Q \lseg 0 \rseg$ and $\veps = 1$.
\end{proof}

\subsection{Analog of the Rellich compactness Theorem}

\begin{Lemma}
\label{481}
Let $f \in W^1_p(\Rm;\calQ_Q(\ell_2))$ and $h \in \Rm$. It follows that
\begin{equation*}
\int_{\Rm} \calG(f(x+h),f(x))^p d\calL^m(x) \leq \|h\|^p \int_{\Rm} \lno Df \rno^p d\calL^m \,.
\end{equation*}
\end{Lemma}

\begin{proof}
According to Propositions \ref{411} and \ref{441}, it suffices to prove it when $f$ is Lipschitz as well. In that case it follows from Theorem \ref{258}(D) and Jensen's inequality that
\begin{equation*}
\begin{split}
\calG(f(x+h),f(x))^p & \leq \left( \|h\| \int_0^1 \lno Df(x+th) \rno d\calL^1(t) \right)^p \\
& \leq \|h\|^p \int_0^1 \lno Df(x+th) \rno^p d\calL^1(t) \,.
\end{split}
\end{equation*}
The conclusion follows upon integrating with respect to $x \in \Rm$.
\end{proof}

\begin{Theorem}
\label{482}
Let $U = U(0,1)$ be the unit ball in $\Rm$ and let $\{f_j\}$ be a sequence in $W^1_p(U;\calQ_Q(\ell_2))$ such that
\begin{enumerate}
\item[(1)] There exists a compact set $C \subset \ell_2$ such that $f_j(x) \in \calQ_Q(C)$ for every $x \in U$ and every $j=1,2,\ldots$;
\item[(2)] $\sup_j \int_U \lno Df_j \rno^p d\calL^m < \infty$.
\end{enumerate}
It follows that there exists a subsequence $\{f_{k(j)}\}$ and $f \in W^1_p(U;\calQ_Q(\ell_2))$ such that $\lim_j d_p(f,f_{k(j)}) = 0$.
\end{Theorem}

\begin{proof}
We show that the compactness Theorem \ref{421} applies to the sequence $\{E(f_j)\}$ in $L_p(\Rm;\calQ_Q(C))$. Our hypothesis (1) and Theorem \ref{ext}(D) guarantee that the extension $E(f_j)$ take their value in $\calQ_Q(C)$. We now check that the hypotheses of Theorem \ref{421} are verified:
\begin{enumerate}
\item[(i)] follows from the fact that $C$ is bounded, thus
\begin{equation*}
\int_{\Rm} \lno E(f_j) \rno^p d\calL^m \leq 2^m \balpha(m) Q^\frac{p}{2} (\rmdiam C)^p \,
\end{equation*}
for every $j=1,2,\ldots$;
\item[(ii)] follows from Lemma \ref{481} and Theorem \ref{ext}(C):
\begin{multline*}
\sup_j \int_{\Rm} \calG(E(f_j)(x+h),E(f_j)(x))^p d\calL^m(x) \leq \|h\|^p \int_{\Rm} \lno DE(f_j) \rno^p d\calL^m \\
\leq \|h\|^p C(m,p,Q) \sup_j \left( \int_U \lno f_j \rno^p d\calL^m + \int_U \lno  Df_j \rno^p d\calL^m \right)
\end{multline*}
\item[(iii)] follows from the fact that $E(f_j) = E(f_j)_K$ for each $j=1,2,\ldots$, where $K = B(0,2)$, according to Theorem \ref{ext}(D).
\end{enumerate}
Thus there exists $\hat{f} \in L_p(\Rm;\calQ_Q(\ell_2))$ such that $\lim_j d_p(E(f_{k(j)}),\hat{f}) = 0$. It remains to notice that the restriction $\hat{f} \restriction U$ belongs to $W^1_p(U;\calQ_Q(\ell_2))$. This is because for each $j=1,2,\ldots$ we can choose a Lipschitz map $g_j : \Rm \to \calQ_Q(\ell_2)$ such that $d_p(E(f_{k(j)},g_j) < j^{-1}$ and $\int_{\Rm} \lno Dg_j \rno^p d\calL^m \leq j^{-1} + \int_{\Rm} \lno E(f_{k(j)}) \rno^p d\calL^m$. Thus $\lim_j d_p(\hat{f} \restriction U , g_j \restriction U) = 0$ and $\sup_j \int_U \lno D(g_j \restriction U) \rno^p d\calL^m < \infty$.
\end{proof}

\begin{Remark}
It would be interesting to know whether or not all the results proved so far in this paper hold when the range $\ell_2$ is replaced by an infinite dimensional Banach space $Y$ which is separable, a dual space, and admits a monotone Schauder basis.
\end{Remark}

\subsection{Existence Theorem}

\begin{Lemma}
\label{491}
Assume that
\begin{enumerate}
\item[(A)] $X$ is a compact metric space;
\item[(B)] $Y$ is a metric space;
\item[(C)] $g : X \to \calQ_Q(Y)$ is continuous.
\end{enumerate}
It follows there exists a compact set $C \subset Y$ such that $g(x) \in \calQ_Q(C)$ for every $x \in Y$.
\end{Lemma}

\begin{proof}
We let $C = Y \cap \{ y : y \in \rmsupp g(x) \text{ for some } x \in X \}$. One easily checks that $C$ is closed, thus it suffices to show it is totally bounded. Since $\rmim g$ itself is compact, given $\veps > 0$ there are $x_1,\ldots,x_\kappa \in X$ such that for each $x \in X$ there exists $k=1,\ldots,\kappa$ with $\calG(f(x),f(x_k)) < \veps$. We write $f(x_k) = \oplus_{i=1}^Q \lseg y_i^k \rseg$. It it now obvious that $C \subset \cup_{k=1}^\kappa \cup_{i=1}^Q B_Y(y_i^k,\veps)$.
\end{proof}

\begin{Theorem}
Let $U = U(0,1)$ be the unit ball in $\Rm$ and let $g : \rmBdry U \to \calQ_Q(\ell_2)$ be Lipschitz. It follows that the minimization problem
\begin{equation*}
\begin{cases}
\text{minimize } \int_U \lno Df \rno^p d\calL^m \\
\text{among } f \in W^1_p(U;\calQ_Q(\ell_2)) \text{ such that } \calT(f) = g
\end{cases}
\end{equation*}
admits a solution.
\end{Theorem}

\begin{proof}
The class of competitors is not empty according to the extension Theorem \ref{243}.
We let $C_0 \subset \ell_2$ be a compact set associated with $g$ in Lemma \ref{491} and we let $C$ be the convex hull of $C_0 \cup \{0\}$ (so that $C$ is compact as well). We denote by $P : \ell_2 \to C$ the nearest point projection. Given a minimizing sequence $\{f_j\}$ we consider the sequence $\{ \calQ_Q(P) \circ f_j \}$ which, we claim, is minimizing as well. That these be Sobolev maps, and form a minimizing sequence, follows from the inequalities
\begin{equation*}
\int_U \calG(\calQ_Q(P) \circ f , \calQ_Q(P) \circ f')^p d\calL^m \leq \int_U \calG(f,f')^p d\calL^m
\end{equation*}
(recall the paragraph preceding Proposition \ref{443}) and
\begin{equation*}
\int_U \lno D( \calQ_Q(P) \circ f ) \rno^p d\calL^m \leq \int_U \lno Df \rno ^p d\calL^m
\end{equation*}
(because $\rmLip P \leq 1$) valid for every Lipschitz $f,f' : U \to \calQ_Q(\ell_2)$, and hence for every $f,f' \in W^1_p(U;\calQ_Q(\ell_2))$ as well. It also follows from these inequalities and Theorem \ref{ext}(B) and (C) that
\begin{equation*}
\calT(\calQ_Q(P) \circ f) = \calT(f)
\end{equation*}
whenever $f \in W^1_p(U;\calQ_Q(\ell_2))$. Thus $\calT(\calQ_Q(P) \circ f_j)=g$, $j=1,2,\ldots$. Since all these $\calQ_Q(P) \circ f_j$ take their values in $\calQ_Q(C)$, it follows from Theorem \ref{482} that there are integers $k(1) < k(2) < \ldots$ and $f \in W^1_p(U;\calQ_Q(\ell_2))$ such that $\lim_j d_p(f,f_{k(j)})=0$. Theorem \ref{472}(B) implies that $\calT(f) =g$. Proposition \ref{441} and Corollary \ref{454} guarantee the required lower semicontinuity.
\end{proof}


\bibliographystyle{amsplain}
\bibliography{../../../Bibliography/thdp}


\end{document}